\documentclass[11pt,a4paper]{article}
\usepackage[margin=1in]{geometry}
\usepackage[T1]{fontenc}
\usepackage{lmodern}
\usepackage{authblk}
\usepackage{amsmath,amsthm,amssymb,amsfonts,mathtools}
\usepackage[expansion=false]{microtype}
\usepackage{cite}
\usepackage{fancyhdr}
\usepackage[hidelinks,unicode]{hyperref}
\allowdisplaybreaks[2]
\newtheorem{theorem}{Theorem}[section]
\newtheorem{proposition}[theorem]{Proposition}
\newtheorem{lemma}[theorem]{Lemma}
\newtheorem{corollary}[theorem]{Corollary}
\theoremstyle{definition}

\theoremstyle{remark}
\newtheorem{remark}[theorem]{Remark}
\theoremstyle{definition}
\newtheorem{example}[theorem]{Example}
\numberwithin{equation}{section}

\renewenvironment{abstract}{\small\noindent\textbf{\abstractname.}\enspace\ignorespaces}{\par}
\hypersetup{pdftitle={Universal weighted sampling of positive and normal operator orbits},pdfauthor={Jian Wu}}

\title{Universal weighted sampling of positive and normal operator orbits}
\author{Jian Wu\thanks{Corresponding author: \texttt{xingxingwu2022@163.com}}}
\affil{}

\date{}

\begin{document}
\maketitle

\begin{abstract}
We study common weighted sampling designs for operator systems that are
exactly observable at nonnegative integer times. For positive operators with a prescribed bound $\kappa$ on the original frame condition number, any prescribed accuracy in relative Gramian approximation can be achieved by a common weighted integer sampling design using $O(\log R)$ distinct times up to $R$. We refer to sampling with this counting bound as logarithmic sampling. For uniform preservation of the frame property, the optimal logarithmic counting coefficient satisfies $C^*(\kappa)\sim\pi^{-2}\log\kappa$ as $\kappa\to\infty$. Real and integer sampling times yield the same infimum, and we also determine its asymptotics for nearly tight frames as $\kappa\downarrow1$. Prescribing a finite upper bound on the sampled condition number strictly increases the optimal logarithmic counting coefficient. For tight frames, we obtain an exact relation between sampling density and the condition number and a design that attains the optimum, whereas the optimum without a prescribed bound on the sampled condition number is never attained. For positive operators, a single design valid for every finite bound on the original frame condition number requires superlogarithmic counting, although the ratio to $\log R$ can diverge arbitrarily slowly. For unrestricted normal operators, every common design under a fixed bound $\kappa>1$ on the original frame condition number has a counting function that grows at least linearly. If the nonzero spectral values satisfy $|\arg\lambda|\le M(-\log|\lambda|)$ for a fixed $M\ge0$, common logarithmic sampling is possible with a sharp leading coefficient. Under suitable assumptions on the radial distribution of a fixed Borel spectral set, we characterize when common logarithmic sampling is possible and determine the minimum upper counting exponent, which is attained by an integer sampling design. Examples show that the conclusions for general spectral sets can fail without the radial density assumption.
\end{abstract}



\section{Introduction}
\label{sec:introduction}

Dynamical sampling studies the recovery of an initial state from observations of its evolution. Measurements taken at different times can compensate for incomplete information at any one time, and frame inequalities express the stability of this recovery. Once stable recovery has been established, a further question is how many observation times are actually needed. This question is particularly natural for systems observed over an unbounded time interval: sampling at every nonnegative integer may use far more observation times than are necessary to preserve stability. A recent survey by Aldroubi, Cabrelli, Krishtal and Molter \cite{ACKMSurvey2026} reviews dynamical sampling and its connections with frame theory, including discrete and continuous evolutions.

Early work on dynamical sampling studied the tradeoff between
spatial and temporal sampling \cite{ADK2013}. Subsequent work
established sufficient conditions for reconstructing signals in
$\ell^2(\mathbb Z)$ and in shift invariant subspaces of $L^2(\mathbb R)$ from samples taken during their evolution under convolution operators \cite{ADK2015}. The connection between dynamical sampling and frames generated by operator orbits was studied in \cite{ACMT}. Frames generated by a single operator orbit were characterized in \cite{CHPorbits}. The frame properties of orbits of normal operators were studied
using the spectral theorem in \cite{ACCMP}, and frames generated by iterating an operator on finitely many vectors were characterized in \cite{CMMP}. Frames formed by the union of finitely many orbits of a normal operator were characterized using model subspaces of $H^2(\mathbb D)$ and the pseudohyperbolic metric on $\mathbb D$ \cite{CMSmodel}. Cabrelli, Molter and Su\'arez \cite{CMSiterations} also characterized frames formed by nonnegative iterates of a bounded operator acting on at most countably many vectors, using model spaces in vector valued Hardy spaces. Dynamical sampling in continuous time and the discretization of
the time variable were studied in \cite{AHP,DMMcontinuous}.
The connection between dynamical sampling and observability
was discussed in \cite{DMMcontrol}. A systematic treatment of observation operators, admissibility and exact observability for operator semigroups is given by Tucsnak and Weiss \cite{TucsnakWeiss2009}.
A central question in this literature is when observations of a
given evolution allow stable recovery of the initial state.

Here we seek a single choice of sampling times and positive
weights that ensures stable recovery for every system in a given
class. The design must be independent of the particular spectrum,
its multiplicities, the observation operator and the state space
dimension. It may depend on a prescribed upper bound on the
original frame condition number and, when the spectrum is
restricted to a fixed set, on that set.

This requirement distinguishes our problem from the discretization or sparsification of a given frame. The discretization of continuous frames was studied by Fornasier and Rauhut \cite{FR} and Freeman and Speegle \cite{FS}. More recently, Bownik and Yu \cite{BownikYu2026} obtained nearly tight frames by sampling bounded continuous tight frames on uniformly discrete sets, under suitable assumptions on the underlying metric measure space. In finite dimensions, sparsification of sums of positive operators \cite{BSS,dCHS2016} and constructive weighted subsampling of finite frames \cite{BSU2023} provide quantitative control of the resulting frame bounds. In these results, the sampling choices may depend on the particular frame being discretized or sparsified. The requirement of a common sampling set also arises in universal discretization. Temlyakov \cite{Temlyakov2018} constructs common sampling sets for discretizing integral norms on certain collections of finite dimensional spaces of multivariate trigonometric polynomials. We ask how slowly the number of distinct observation times up to $R$ can grow as $R\to\infty$.

Related questions arise in actuator and sensor scheduling for
linear systems in finite dimensions. For a fixed controllable
system in discrete time, Siami, Olshevsky and Jadbabaie
\cite{SOJ2021} construct sparse weighted actuator schedules.
Over a sufficiently long finite time interval, these schedules
yield a relative approximation of the original controllability
Gramian. By duality, their results
also apply to sensor scheduling and the observability Gramian.
These methods use the system matrices to select and weight
individual channels at different times. Within dynamical sampling, Carleson frames whose eigenvalue moduli form a strictly increasing Carleson sequence remain frames when sampled at times $Nk$, $k\geq0$, for any positive integer $N$ \cite{CHPSmystery}. For positive eigenvalues, this conclusion extends to sampling times $Nk+j_k$ with arbitrary real $j_k\in[0,N)$ \cite{KM}. In our setting, all observation channels are retained at each
selected time.

Relations between sampling density and stability also arise in
other settings. Aldroubi, Gr\"ochenig, Huang, Jaming, Krishtal
and Romero \cite{AGHJKR2021} show that, under suitable assumptions
on the convolution kernels, the reconstruction condition number
controls the maximal spatial sampling gap for evolving bandlimited
functions. Bownik and van Velthoven \cite{BownikVanVelthoven2025}
construct exponential frames with density arbitrarily close to
the critical density and with quantitative estimates for the
frame bounds. In our setting, we seek sharp bounds on the growth of the number of distinct observation times.

To formulate the problem, let $H$ and $Y$ be separable complex Hilbert
spaces, let $A\in\mathcal B(H)$ be positive and selfadjoint, and let
$C\in\mathcal B(H,Y)$. We assume exact observability at the nonnegative
integers:
\begin{equation}\label{eq:frame}
 a\|x\|^2\leq\sum_{n=0}^{\infty}\|CA^nx\|^2\leq b\|x\|^2
 \qquad x\in H,\qquad 0<a\leq b<\infty.
\end{equation}
If $(e_j)$ is an orthonormal basis of $Y$ and $f_j=C^*e_j$, then
\eqref{eq:frame} is equivalent to the frame inequalities for
$(A^nf_j)_{n,j}$ with bounds $a$ and $b$. With $B=C^*C$, the
observation Gramian at integer times is
\begin{equation}\label{eq:gramian}
 S=\sum_{n=0}^{\infty}A^nBA^n,\qquad aI\leq S\leq bI,
\end{equation}
where the sum converges strongly. We assume that $b/a\leq\kappa$;
the frame bounds $a$ and $b$ need not be optimal.

A sampling design is a positive atomic measure
$\mu=\sum_{t\in T}w_t\delta_t$, where $T\subset[0,\infty)$ is locally
finite and $0<w_t<\infty$ for every $t\in T$. Here local finiteness
means that $T\cap[0,R]$ is finite for every $R\geq0$.
Its sampled energy, cumulative mass and counting function are
\[
 \langle G_\mu x,x\rangle=\sum_{t\in T}w_t\|CA^tx\|^2,
 \qquad M_\mu(R)=\sum_{t\in T\cap[0,R]}w_t,
 \qquad N_\mu(R)=\#(T\cap[0,R]).
\]
At each selected time $t$, the full observation $CA^tx\in Y$ is
retained. Thus $N_\mu(R)$ counts each selected time in $[0,R]$
once, regardless of the dimension of $Y$.

For $\kappa\geq1$, let $\mathcal U_w(\kappa)$ be the class of sampling
designs $\mu$ such that, for every $0<a\leq b<\infty$ with
$b/a\leq\kappa$, there exist constants $0<\alpha\leq\beta<\infty$,
depending only on $\mu,a,b$, for which
$\alpha I\leq G_\mu\leq\beta I$ for every positive system with
$aI\leq S\leq bI$. Thus, for fixed $a$ and $b$, the same frame bounds
after sampling apply to every such system. A design in
$\mathcal U_w=\bigcap_{\kappa\geq1}\mathcal U_w(\kappa)$ is called
\emph{fully universal}. The same sampling times and weights must
work for every finite $\kappa\geq1$. The frame bounds $\alpha$
and $\beta$ after sampling may depend on the original frame
bounds $a$ and $b$. We measure sparsity by
\[
 D(\mu)=\limsup_{R\to\infty}\frac{N_\mu(R)}{\log R},
 \qquad C^*(\kappa)=\inf_{\mu\in\mathcal U_w(\kappa)}D(\mu).
\]
Restricting the support to $\mathbb N_0$ gives the corresponding
infimum $C^*_{\mathbb Z}(\kappa)$.

The unweighted universal problem in \cite{Wu} provides a useful
comparison: even for tight frames, a common unweighted sampling set
must satisfy $N_T(R)\asymp R$. Positive weights allow the number of
sampling times to grow more slowly than the cumulative mass.
Proposition~\ref{prop:tight} shows that a sampling design that works
for all positive systems with tight frames must satisfy
$M_\mu(R)\asymp R$. Such designs can be chosen so that
$N_\mu(R)=O(\log R)$ as $R\to\infty$.
Our first objective is to determine the optimal coefficient
$C^*(\kappa)$. Our second is to quantify the increase in this
coefficient when a bound on the condition number after sampling
is prescribed. Our third is to determine how spectral geometry
affects the growth of the counting function required for common
sampling designs on normal systems.

For positive systems, we first prove that, for every fixed
$\kappa\geq1$ and $0<\varepsilon<1$, there exists a sampling design
$\mu$ supported on $\mathbb N_0$ with $N_\mu(R)=O(\log R)$.
The same design satisfies
\(
 (1-\varepsilon)S\leq G_\mu\leq(1+\varepsilon)S
\)
for every positive system whose original frame bounds satisfy
$b/a\leq\kappa$. This follows from
Theorems~\ref{thm:relative}, \ref{s4-integer} and \ref{s6:joint}.

The optimal coefficients satisfy
$C^*_{\mathbb Z}(\kappa)=C^*(\kappa)$ for every $\kappa\geq1$,
with the asymptotic behavior
\begin{align*}
 C^*(\kappa)&\sim\frac{\log\kappa}{\pi^2}
 &&\kappa\to\infty,\\
 \frac{1}{C^*(\kappa)}&=2\log\frac{1}{\kappa-1}+O(1)
 &&\kappa\downarrow1.
\end{align*}
Moreover, $C^*(1)=0$ and $C^*(\kappa)>0$ for $\kappa>1$;
see Section~\ref{sec:fixed}.
Neither the infimum over real sampling times nor the infimum
over integer sampling times is attained for any finite
$\kappa\geq1$ (Theorem~\ref{s5-nonattainment}).

Fixing an upper bound on the original frame condition number is
essential for obtaining a common sampling design with
$N_\mu(R)=O(\log R)$. Every fully universal design satisfies
$N_\mu(R)/\log R\to\infty$ as $R\to\infty$.
Conversely, for every nondecreasing function $h$ with
$h(R)\to\infty$ as $R\to\infty$, there exists a fully universal
design supported on $\mathbb N_0$ such that
$N_\mu(R)\leq h(R)\log R$ for all sufficiently large $R$.
Thus the ratio $N_\mu(R)/\log R$ must tend to infinity, but its
divergence can be arbitrarily slow; see
Corollary~\ref{s4-superlog-necessary}
and Theorem~\ref{s4-universal-design}.

We next ask how the counting problem changes when an upper bound
on the condition number after sampling is prescribed.
For a bounded positive invertible operator $G$, its condition
number is defined by
\(
 \operatorname{cond}(G)=\|G\|\|G^{-1}\|.
\)
When $G$ is a frame operator, its optimal lower and upper frame
bounds are $\|G^{-1}\|^{-1}$ and $\|G\|$, respectively.
Thus $\operatorname{cond}(G)$ is the ratio of the optimal upper
frame bound to the optimal lower frame bound.
For $\kappa,\Gamma\geq1$, let $\mathcal U_w(\kappa;\Gamma)$
consist of the designs $\mu\in\mathcal U_w(\kappa)$ for which
$\operatorname{cond}(G_\mu)\leq\Gamma$ for every positive system
with $\operatorname{cond}(S)\leq\kappa$.
Denote the infima of $D(\mu)$ over these designs with real and
integer sampling times by $C^*(\kappa;\Gamma)$ and
$C^*_{\mathbb Z}(\kappa;\Gamma)$, respectively.
No such design exists if $\Gamma<\kappa$.
If $\Gamma=\kappa$, the only possible designs are
$\mu=c\sum_{n=0}^{\infty}\delta_n$ with $c>0$.
If $\Gamma>\kappa$, such designs with $N_\mu(R)=O(\log R)$ exist.
Writing $X=\kappa(\Gamma+\kappa)/(\Gamma-\kappa)$, we have
\[
 C^*(\kappa;\Gamma)\sim C^*_{\mathbb Z}(\kappa;\Gamma)
 \sim\frac{\log X}{\pi^2},
 \qquad X\to\infty.
\]
These asymptotics hold along every parameter family with
$\kappa\geq1$, $\Gamma>\kappa$ and $X\to\infty$;
see Theorem~\ref{s6:joint}.
Moreover, prescribing any finite bound $\Gamma>\kappa$ strictly
increases the optimal coefficient:
\(
 C^*(\kappa;\Gamma)>C^*(\kappa).
\)
Consequently, for fixed $\kappa$, if
$\mu_n\in\mathcal U_w(\kappa)$ and $D(\mu_n)\to C^*(\kappa)$,
then the supremum of $\operatorname{cond}(G_{\mu_n})$ over all
positive systems with $\operatorname{cond}(S)\leq\kappa$
tends to infinity; see Theorem~\ref{s6-strict-gap} and
Corollary~\ref{s6-strict-conditioning}.

For tight frames, Subsection~\ref{s7-exact} determines the exact value of $C^*(1;\Gamma)$ for every $\Gamma>1$.
The real and integer infima agree. For each fixed $\Gamma>1$,
there exists a sampling design whose sampling times are
nonnegative integers and which attains both infima.
Theorem~\ref{s7-local-counting} also gives uniform lower bounds
for the number of sampling times in intervals $(R,e^\ell R]$,
with the same sharp asymptotic coefficient as $\ell\to\infty$.
Here $\Gamma$ is a prescribed upper bound on the condition number
after sampling. Without such a prescribed bound, the infimum
$C^*(1)=0$ is not attained.

The extension to normal operators shows how the arguments of
spectral values affect sampling. In this setting, the observation
Gramian is $S=\sum_{n=0}^{\infty}(A^n)^*BA^n$.
For tight frames, the classes of common sampling designs for
positive and normal operators coincide;
see Proposition~\ref{nrm-tight}.
For every fixed $\kappa>1$, however, any sampling design common
to all normal systems with $\operatorname{cond}(S)\leq\kappa$
must satisfy $N_\mu(R)\gtrsim R$.
This lower bound still holds when the spectra are restricted
to a fixed sector with positive opening angle.
Common designs with $N_\mu(R)=O(\log R)$ exist under the
logarithmic sector condition
$|\arg\lambda|\leq M(-\log|\lambda|)$ for $0<|\lambda|<1$,
with any fixed $M\geq0$.
The corresponding optimal coefficients for real and integer
sampling times satisfy
\[
 C_M^*(\kappa)=C_{M,\mathbb Z}^*(\kappa)
 \sim\frac{\log\kappa}{\pi(\pi-2\arctan M)},
 \qquad \kappa\to\infty.
\]
Section~\ref{nrm-section} proves the linear lower bound and
the sharp asymptotic formula under this spectral restriction.

Let $\mathbb D$ denote the open unit disk.
For a fixed Borel set $\Omega\subset\mathbb D$, we consider
normal systems satisfying $E_A(\mathbb C\setminus\Omega)=0$,
where $E_A$ is the spectral measure of $A$.
Section~\ref{sec:spectral-geometry} relates sampling growth
to the geometry of $\Omega$.
For $j\geq1$, let $m_j(\Omega)$ be the minimum number of circular
arcs of length $2^{-j}/10$ needed to cover the arguments of
nonzero points $\lambda\in\Omega$ satisfying
$2^{-j}\leq-\log|\lambda|<2^{1-j}$.
Set $m_j(\Omega)=0$ when there are no such points, and put
$\mathcal A_\Omega(J)=1+\sum_{j=1}^Jm_j(\Omega)$.
For every fixed $\kappa\geq1$, the condition
$\mathcal A_\Omega(J)=O(J)$ guarantees common designs with
integer sampling times, $N_\mu(R)=O(\log R)$, and arbitrarily
small prescribed relative error in the Gramian approximation.
Under the radial density assumption stated in that section,
this condition is also necessary for every fixed $\kappa>1$;
see Theorem~\ref{geo:classification}.
For the enlarged spectral set $\Omega\cup[0,1)$, the same
condition is necessary and sufficient for common sampling with
$N_\mu(R)=O(\log R)$ at any fixed $\kappa>1$;
see Corollary~\ref{geo:positive-completion}.
Under the same radial density assumption and an additional
regularity condition on the growth of $\mathcal A_\Omega$,
Theorem~\ref{geo:matching-growth} identifies the optimal counting
order as $\mathcal A_\Omega(\lfloor\log_2R\rfloor)$ for every
fixed $\kappa>1$.
Without this additional growth condition, but retaining the
radial density assumption, Theorem~\ref{geo:optimal-exponent}
determines the minimum upper counting exponent.
This minimum is attained by a design with integer sampling times
and any prescribed relative Gramian accuracy.
Examples involving spectral curves and angular sets with
prescribed upper box dimension give optimal growth rates that
are powers of $R$ or $\log R$.
Examples with gaps in the radial distribution show that the
necessity statements can fail without the radial density
assumption.

For positive operators, the Stein identity $B=S-ASA\geq0$
relates the original frame bounds to the weighted Gramian.
To obtain upper bounds on counting growth, we use Laplace
quadrature in logarithmic time, based on approximation by
exponential sums \cite{BM} and the trapezoidal rule for functions
analytic in a strip \cite{TrefethenWeideman2014}.
Related scalar approximation estimates are discussed in
\cite{Hackbusch2019}.
The main difficulty is that $A$ and $B$ need not commute.
Uniform estimates for the observation energy at complex times
give quadrature error bounds valid for every system in the class.
Observations at finitely many initial integer times control
the part of the spectrum away from $1$.
For lower bounds, geometric sequences and the Cayley transform
give systems with scalar observations and Toeplitz Gramians.
A dimension argument on finite spectral subspaces, together
with estimates for the energy outside a given time interval,
gives a lower bound on the number of sampling times in that
interval.
We also prove that each fixed common design remains valid
after a sufficiently small increase in the bound on the
original frame condition number.
Combining this fact with the change of variables
$t\mapsto t^p$, $p>1$, and suitable changes in the weights
proves nonattainment of the infima and the strict increase
in the optimal counting coefficient when a finite bound on
the condition number after sampling is prescribed.

For tight frames, the analysis reduces to a convolution problem
after a logarithmic change of the time variable.
The theory of variation diminishing transformations
\cite{Schoenberg1951,Schoenberg1953,MSW,Hoa1987}
gives sharp bounds for periodic sampling measures.
A periodization argument extends these bounds to general
sampling designs.
For general spectral sets, we represent the observations using
a common family of scalar functions determined by the prescribed
spectral set. Sparsification estimates in terms of stable rank
\cite{CNW2016} then yield a common choice of sampling times and
positive weights. The definition of stable rank is given in
Subsection~\ref{geo:covering-subsection}.
Under the radial density assumption, finite dimensional systems
constructed from spectral points with separated arguments yield
the necessary counting bounds. Their original frame condition
numbers are bounded by the prescribed $\kappa$.
The functional calculus and sectorial results in
\cite{KMMfunctional,KMMsectorial} give conditions under which
sampled operator orbits form frames.
Here we optimize counting growth using a common choice of
sampling times and positive weights, with frame bounds uniform
over the prescribed class of systems.

Section~\ref{sec:preliminaries} provides the basic estimates.
Section~\ref{sec:fixed} establishes sharp asymptotic formulas
for the optimal counting coefficient for positive operators.
Section~\ref{s4-section} treats sampling at integer times,
fully universal designs, and the stability argument used to
prove nonattainment.
Section~\ref{sec:output} quantifies the increase in the optimal
counting coefficient when a bound on the condition number after
sampling is prescribed, and determines $C^*(1;\Gamma)$ exactly
for $\Gamma>1$.
Sections~\ref{nrm-section} and \ref{sec:spectral-geometry}
treat normal operators and the dependence of sampling growth
on spectral geometry, respectively.
The appendices contain results on prescribed sampling sets,
sharpness results for a fixed model based on the Cayley transform,
arithmetic sampling examples, and details of the scalar and
spectral constructions.

\section{Stein identities and weighted energy estimates}
\label{sec:preliminaries}

We use the Stein identity to estimate the weighted observation
energy in terms of the original frame bounds.
We first derive the energy estimates needed for the constructions
and perturbation arguments. We then characterize common sampling
designs for tight frames by requiring the associated scalar
function to be bounded above and bounded away from zero.
Throughout Sections~\ref{sec:preliminaries}--\ref{sec:output}
and Appendices~\ref{app:arithmetic-designs}, \ref{app:fixed-model},
and \ref{app:prescribed-supports}, $A$ is positive and selfadjoint.
The conventions for normal operators are given in
Section~\ref{nrm-section}.
Write $\mathbb N=\{1,2,\ldots\}$ and
$\mathbb N_0=\{0,1,\ldots\}$.
For $L\geq0$, $\tau_L\mu$ denotes the translation of $\mu$ by $L$,
so that $\tau_L\mu=\sum_{t\in T}w_t\delta_{t+L}$.
All logarithms are natural unless another base is specified.
For positive quantities $u$ and $v$, the notation $u\asymp v$
means that $u/v$ is bounded above and below by positive constants
throughout the stated range.
Inner products are linear in the first variable.
For $t\geq0$, the powers $A^t$ are defined by spectral calculus,
with $A^0=I$, including on $\ker A$.

The weights $w_t$ need not be bounded uniformly in $t$.
The counting function $N_\mu$ records distinct sampling times,
and all observation channels are retained at each selected time.
All frame inequalities concern the full sampling family.
Positive sums and integrals, including the expression defining
$G_\mu$, are initially interpreted through their nonnegative
quadratic forms, allowing the value $+\infty$.
Once such a form is bounded above by $K\|x\|^2$ for all $x\in H$,
with a finite constant $K$, it defines a bounded positive operator.

The definition of $\mathcal U_w(\kappa)$ requires frame bounds
after sampling that are uniform over systems with the same
original frame bounds $a,b$.
For a fixed design $\mu$, this uniformity already follows if
the sampled family is a frame for every such system.
Indeed, if a common finite upper bound or a common positive
lower bound did not exist, a countable orthogonal direct sum
of suitably chosen systems would still satisfy $aI\leq S\leq bI$,
while its sampled family would violate the corresponding
frame inequality.
The direct sum operators are bounded: each $A$ is a contraction
by Lemma~\ref{lem:stein}, and $\|C\|^2\leq b$ follows from
$B\leq S\leq bI$ for every component system.

\begin{lemma}
\label{lem:stein}
Under \eqref{eq:frame}, the integer Gramian satisfies
$B=S-ASA\geq0$.
Moreover, $0\leq A\leq I$ and $A^n\to0$ strongly.
In particular, $E_A(\{1\})=0$.
The continuous observation Gramian satisfies
\begin{equation}\label{eq:continuous}
 G_c:=\int_0^\infty A^tBA^t\,dt
 =\int_0^1 A^sSA^s\,ds\leq bI.
\end{equation}
If $b/a\leq\kappa$, set
$V=\int_{1/2}^\infty A^tBA^t\,dt$ and $G_*=B+V$.
Then $\kappa^{-1}S\leq G_*\leq2bI$.
\end{lemma}

\begin{proof}
Splitting off the term $n=0$ in the series for $S$ gives
$S=B+ASA$.
For every $x\in H$,
\[
 a\|A^nx\|^2
 \leq\langle SA^nx,A^nx\rangle
 =\sum_{j=n}^\infty\|CA^jx\|^2
 \longrightarrow0,
 \qquad n\to\infty.
\]
Thus $A^n\to0$ strongly and $\|A^n\|\leq\sqrt{b/a}$.
Since $A$ is positive and selfadjoint,
$\|A\|^n=\|A^n\|$, so $\|A\|\leq1$.
The identity $A^nE_A(\{1\})=E_A(\{1\})$ and strong convergence
give $E_A(\{1\})=0$.

For every $x\in H$, Tonelli's theorem gives
\[
 \int_0^\infty\|CA^tx\|^2\,dt
 =\int_0^1\sum_{n=0}^\infty\|CA^{n+s}x\|^2\,ds
 =\int_0^1\langle SA^sx,A^sx\rangle\,ds
 \leq b\|x\|^2.
\]
This proves \eqref{eq:continuous}.
The same argument, with the integral starting at $1/2$, gives
$V=\int_{1/2}^{3/2}A^sSA^s\,ds$.
Jensen's inequality gives
$\int_{1/2}^{3/2}v^s\,ds\geq v$ for $0<v\leq1$,
and the inequality also holds for $v=0$.
Since $S\geq aI$, the spectral theorem yields
\[
 V\geq a\int_{1/2}^{3/2}A^{2s}\,ds\geq aA^2.
\]
Consequently,
\(
 S=B+ASA
 \leq B+bA^2
 \leq B+\kappa V
 \leq\kappa G_*.
\)
Finally, $B\leq S\leq bI$ and $V\leq G_c\leq bI$
give $G_*\leq2bI$.
\end{proof}

The observation at time zero is essential when $\ker A\neq\{0\}$,
since $A^tx=0$ for every $x\in\ker A$ and $t>0$.
On $H_+=(\ker A)^\perp$, define
$L=-\log(A|_{H_+})$ by spectral calculus.
The operator $L$ is nonnegative and selfadjoint but need not be
bounded.
For each fixed $t>0$, the functions $\lambda^t$ and
$(-\log\lambda)\lambda^t$ are bounded on $(0,1]$.
Thus $A^t$ and $LA^t$ are bounded on $H_+$.
We extend both operators by zero on $\ker A$.

\begin{lemma}
\label{lem:energy}
For $x\in H$, set $F_x(t)=\|CA^tx\|^2$ for $t>0$.
Then
\[
 \int_0^\infty\|CA^tx\|^2\,dt\leq b\|x\|^2,
 \qquad
 \int_0^\infty t^2\|CLA^tx\|^2\,dt
 \leq\frac b2\|x\|^2.
\]
Moreover, $F_x$ is continuously differentiable on $(0,\infty)$
and satisfies
$\int_0^\infty t|F_x'(t)|\,dt\leq\sqrt2\,b\|x\|^2$.
\end{lemma}

\begin{proof}
The first inequality follows from \eqref{eq:continuous}.
For the second, first take $x\in H_+$ and set
$x_m=E_L([m^{-1},m])x$ for $m\in\mathbb N$.
Since $E_A(\{1\})=0$, we have $x_m\to x$.

Using $B=S-ASA$ and shifting the time variable in the term
containing $ASA$, we obtain
\[
\begin{aligned}
 \int_0^\infty t^2\|CLA^tx_m\|^2\,dt
 &=\int_0^\infty
   \bigl(t^2-(t-1)_+^2\bigr)
   \langle SLA^tx_m,LA^tx_m\rangle\,dt\\
 &\leq b\int_0^\infty2t\|LA^tx_m\|^2\,dt\\
 &=\frac b2\|x_m\|^2.
\end{aligned}
\]
Here $(t-1)_+=\max\{t-1,0\}$.
The spectral restriction makes the integrals in the change
of variables finite.
The inequality uses
$0\leq t^2-(t-1)_+^2\leq2t$ and $S\leq bI$.
The last equality follows from the spectral theorem and
$\int_0^\infty2t\lambda^2e^{-2t\lambda}\,dt=1/2$
for $\lambda>0$.

For each $t>0$, the operator $LA^t$ is bounded.
Fatou's lemma therefore gives the second estimate for $x\in H_+$.
Since $LA^t$ vanishes on $\ker A$, the estimate holds for
every $x\in H$.

Finally, differentiation for $t>0$ gives
$F_x'(t)=-2\operatorname{Re}
\langle CLA^tx,CA^tx\rangle$.
Hence $|F_x'(t)|\leq2\|CA^tx\|\,\|CLA^tx\|$.
The Cauchy--Schwarz inequality and the two energy estimates yield
\[
 \int_0^\infty t|F_x'(t)|\,dt
 \leq
 2\left(\int_0^\infty\|CA^tx\|^2\,dt\right)^{1/2}
 \left(\int_0^\infty t^2\|CLA^tx\|^2\,dt\right)^{1/2}
 \leq\sqrt2\,b\|x\|^2.
\]
\end{proof}

We use the preceding energy estimates to bound the sampled
Gramian in terms of the cumulative mass of the sampling measure.

\begin{proposition}
\label{prop:mass}
Let $\mu$ be a positive Borel measure on $[0,\infty)$ such that
$w_0=\mu(\{0\})<\infty$ and
$K_+=\sup_{R>0}\mu((0,R])/R<\infty$.
For every system satisfying \eqref{eq:frame}, the quadratic form
$\langle G_\mu x,x\rangle
=\int_{[0,\infty)}\|CA^tx\|^2\,d\mu(t)$
defines a bounded positive operator satisfying
\(
 G_\mu\leq b(w_0+\sqrt2\,K_+)I.
\)
\end{proposition}

\begin{proof}
Fix $x\in H$ and set $W(t)=\mu((0,t])$.
Since $A$ is a contraction and $A^n\to0$ strongly,
$F_x(t)\to0$ as $t\to\infty$.
Lemma~\ref{lem:energy} therefore gives
$F_x(t)=-\int_t^\infty F_x'(s)\,ds$ for $t>0$, and hence
\[
 tF_x(t)\leq\int_t^\infty s|F_x'(s)|\,ds
 \longrightarrow0,
 \qquad t\to\infty.
\]

For $0<\delta<R$, Stieltjes integration by parts gives
\[
 \int_{(\delta,R]}F_x(t)\,d\mu(t)
 =W(R)F_x(R)-W(\delta)F_x(\delta)
  -\int_\delta^R W(t)F_x'(t)\,dt.
\]
Since $W(t)\leq K_+t$, the boundary term at infinity vanishes.
The boundary term at zero also vanishes because
$F_x(t)\leq b\|x\|^2$.
Moreover,
$W(t)|F_x'(t)|\leq K_+t|F_x'(t)|$ is integrable by
Lemma~\ref{lem:energy}.
Letting $\delta\downarrow0$ and $R\to\infty$, we obtain
\[
 \int_{(0,\infty)}F_x(t)\,d\mu(t)
 =-\int_0^\infty W(t)F_x'(t)\,dt
 \leq K_+\int_0^\infty t|F_x'(t)|\,dt
 \leq\sqrt2\,K_+b\|x\|^2.
\]
Finally, the contribution at time zero is
$w_0\|Cx\|^2=w_0\langle Bx,x\rangle\leq w_0b\|x\|^2$.
Adding this term proves the assertion.
\end{proof}

When the original frame is tight, $S=aI$ and the Stein identity
gives $B=a(I-A^2)$.
The resulting scalar criterion relates preservation of the frame
property to the growth of the cumulative mass $M_\mu$, while
the counting function $N_\mu$ records the number of sampling
times carrying this mass.

\begin{proposition}
\label{prop:tight}
Let $\mu$ be a positive atomic measure on $[0,\infty)$
with locally finite support and finite weights, and write
$w_0=\mu(\{0\})$.
The following conditions are equivalent:
\begin{enumerate}
\item $\mu\in\mathcal U_w(1)$;
\item $w_0>0$ and $M_\mu(R)\asymp R$ as $R\to\infty$;
\item the function
$\phi_\mu(z)=(1-z)\int_{[0,\infty)}z^t\,d\mu(t)$,
defined for $0\leq z<1$ with $0^0=1$, satisfies
\[
 0<\inf_{0\leq z<1}\phi_\mu(z)
 \leq\sup_{0\leq z<1}\phi_\mu(z)<\infty.
\]
\end{enumerate}
Every measure satisfying these conditions also satisfies the
hypotheses of Proposition~\ref{prop:mass}.
For a system with $S=aI$, we have $G_\mu=a\phi_\mu(A^2)$.
\end{proposition}

\begin{proof}
Assume (1).
For $0\leq r<1$, take $H=Y=\mathbb C$,
$A=rI$ and $C=\sqrt{1-r^2}I$.
Then $S=I$ and $G_\mu=\phi_\mu(r^2)I$.
The uniform frame bounds after sampling therefore imply (3).

Next assume (3).
Since $\phi_\mu(0)=w_0$, we have $w_0>0$.
For $u>0$, set
$Q(u)=\int_{[0,\infty)}e^{-ut}\,d\mu(t)$.
The identity $\phi_\mu(e^{-u})=(1-e^{-u})Q(u)$ gives
$c/u\leq Q(u)\leq C/u$ for $0<u\leq u_0$, with suitable
constants $c,C,u_0>0$.
Since $M_\mu(R)\leq eQ(1/R)$, it follows that
$M_\mu(R)\leq C_1(1+R)$ for all $R\geq0$.

For $K\geq1$ and $R\geq1$, integration by parts yields
\[
 \int_{(KR,\infty)}e^{-t/R}\,d\mu(t)
 \leq\frac1R\int_{KR}^\infty e^{-t/R}M_\mu(t)\,dt
 \leq C_1R(K+2)e^{-K}.
\]
Choose $K$ so large that $C_1(K+2)e^{-K}\leq c/2$.
For sufficiently large $R$, we then have
\[
 M_\mu(KR)
 \geq Q(1/R)-\int_{(KR,\infty)}e^{-t/R}\,d\mu(t)
 \geq\frac c2R.
\]
Replacing $KR$ by $s$ gives $M_\mu(s)\geq cs/(2K)$
for all sufficiently large $s$.
Together with the upper bound, this proves (2).

Conversely, assume (2).
Choose positive constants $c,K,R_0$ such that
$M_\mu(R)\geq cR$ for $R\geq R_0$ and
$M_\mu(R)\leq K(1+R)$ for all $R\geq0$.
Tonelli's theorem gives
$Q(u)=u\int_0^\infty e^{-uR}M_\mu(R)\,dR$, and hence
\[
 c(R_0+u^{-1})e^{-uR_0}
 \leq Q(u)\leq K(1+u^{-1}),
 \qquad u>0.
\]
These estimates bound $(1-e^{-u})Q(u)$ above and away from
zero for sufficiently small $u>0$.
For $u\geq u_0>0$, we have
\[
 (1-e^{-u_0})w_0
 \leq(1-e^{-u})Q(u)\leq Q(u_0).
\]
Together with $\phi_\mu(0)=w_0$, this proves (3).

Finally, suppose (3) holds and $S=aI$.
The Stein identity gives $B=a(I-A^2)$.
The spectral theorem and Tonelli's theorem therefore yield
$G_\mu=a\phi_\mu(A^2)$.
Any finite value may be assigned to $\phi_\mu(1)$,
since $E_A(\{1\})=0$.
Consequently,
\[
 a\inf_{0\leq z<1}\phi_\mu(z)\,I
 \leq G_\mu
 \leq a\sup_{0\leq z<1}\phi_\mu(z)\,I,
\]
which proves (1).

To verify the hypotheses of Proposition~\ref{prop:mass},
note that $w_0<\infty$ and local finiteness gives
$\delta>0$ such that $\mu((0,\delta])=0$.
Combining this with $M_\mu(R)\leq C_1(1+R)$ gives
$\sup_{R>0}\mu((0,R])/R<\infty$.
\end{proof}

\begin{remark}
If $S=aI$, then $C^*C=a(I-A^2)$, and strong stability gives
$\ker C=\ker(I-A^2)=\{0\}$.
Thus $C$ is injective, and
$\operatorname{rank}C=m<\infty$ implies $\dim H=m$.
In particular, if a single orbit $(A^nf)_{n\geq0}$ is a tight
frame for $H$, then $\dim H\leq1$.
When $H$ is infinite dimensional, the results for tight frames
therefore concern observation operators of infinite rank.
\end{remark}

Proposition~\ref{prop:tight} gives $M_\mu(R)\asymp R$
for every $\mu\in\mathcal U_w(1)$, but does not require
the same growth for $N_\mu(R)$.
If $w_t\leq W$ for all $t\in T$, then
$M_\mu(R)\leq WN_\mu(R)$ implies $N_\mu(R)\gtrsim R$
for sufficiently large $R$.
Thus, for any $\kappa\geq1$, a design
$\mu\in\mathcal U_w(\kappa)$ with $N_\mu(R)=o(R)$
must have unbounded weights.

For later comparison, take $\mu\in\mathcal U_w(1)$
and consider an operator with finite spectrum,
$A=\sum_{i=1}^m\lambda_iP_i$, where $0\leq\lambda_i<1$
and the $P_i$ are its spectral projections.
Writing $S_{ij}=P_iSP_j$ and $B_{ij}=P_iBP_j$, we obtain
\[
 B_{ij}=(1-\lambda_i\lambda_j)S_{ij},
 \qquad
 P_iG_\mu P_j=\phi_\mu(\lambda_i\lambda_j)S_{ij}.
\]
If $A$ and $B$ commute, then $A$ and $S$ also commute,
and $G_\mu=\phi_\mu(A^2)S$.
The scalar criterion characterizes $\mathcal U_w(1)$,
but does not characterize $\mathcal U_w(\kappa)$
for a prescribed $\kappa>1$.
In the general case, the estimates must also account for
the blocks with $i\neq j$.
The arguments below use both the original frame bounds
$aI\leq S\leq bI$ and the constraint $S-ASA\geq0$.

Appendix~\ref{app:arithmetic-designs} gives a separate sufficient
criterion for full universality and applies it to weighted prime
sampling. Appendix~\ref{app:prescribed-supports} treats the different
problem of assigning weights to a prescribed set of times.

\section{Optimal density under a fixed condition number bound}
\label{sec:fixed}

For each fixed upper bound $\kappa$ on the original frame condition
number, there is a common sampling design with
$N_\mu(R)=O(\log R)$. We construct such designs by quadrature and
derive lower bounds on the number of sampling times from a geometric
spectral model. Together, these estimates yield
$C^*(\kappa)\sim\pi^{-2}\log\kappa$ as $\kappa\to\infty$.
Combining the lower bounds with a relative perturbation estimate
also determines the asymptotic behavior as $\kappa\downarrow1$.
Section~\ref{s4-section} shows that restricting the sampling times
to integers does not change the optimal coefficient.
Appendix~\ref{app:fixed-model} determines the exact optimum for
the geometric model, and Appendix~\ref{app:prescribed-supports}
studies sampling designs with prescribed supports.

\subsection{Relative Gramian approximation}

We construct a common sampling design for all systems in the class by approximating the continuous Gramian with logarithmic quadrature and then adding observations at finitely many initial integer times to approximate the integer Gramian. Approximation of scalar functions by sums of exponentials, including approximation of $1/x$ on unbounded
intervals, is studied in \cite{BM2005,BraessHackbusch2005,BM}.
Here the error must be controlled uniformly over the class in the sense of operator quadratic forms, even when $A$ and $B$ do not commute. We therefore begin with a uniform energy estimate for complex times.
Set $K_0=4(3+2\sqrt2)$.

Recall that $C\in\mathcal B(H,Y)$ is the observation operator
and $B=C^*C$.
\begin{lemma}
\label{lem:complex-energy}
For $|\theta|<\pi/2$ and $x\in H$,
\[
 \int_0^\infty\|Ce^{-te^{i\theta}L}x\|^2\,dt
 \leq\frac{K_0b}{\cos\theta}\|x\|^2.
\]
For $t>0$, the operator $e^{-te^{i\theta}L}$ is extended by zero
on $\ker A$.
\end{lemma}

\begin{proof}
On $H_+$, let $P_j=E_L([2^j,2^{j+1}))$ for $j\in\mathbb Z$.
Set $A_j=A|_{P_jH}$ and write $X_{ij}=P_iX|_{P_jH}$ for the
blocks of a bounded operator $X$. The diagonal blocks of the Stein
identity satisfy
$B_{jj}=(I-A_j)S_{jj}+A_jS_{jj}(I-A_j)$.
Thus $\|B_{jj}\|\leq2b\|I-A_j\|\leq4b2^j$.
Since $B$ is positive,
$\|B_{ij}\|\leq\|B_{ii}\|^{1/2}\|B_{jj}\|^{1/2}
\leq4b2^{(i+j)/2}$.

Initially interpret
$R_\theta=\int_0^\infty
e^{-te^{-i\theta}L}Be^{-te^{i\theta}L}\,dt$
as a positive quadratic form on vectors with finitely many nonzero
spectral components. Its blocks satisfy
\[
 \|(R_\theta)_{ij}\|
 \leq\frac{4b}{\cos\theta}\frac{2^{(i+j)/2}}{2^i+2^j}
 \leq\frac{4b}{\cos\theta}2^{-|i-j|/2}.
\]
The Schur estimate for operator blocks, together with
$\sum_{j\in\mathbb Z}2^{-|j|/2}=3+2\sqrt2$, gives the claimed bound
for these vectors. Density and Fatou's lemma extend the bound
to all of $H_+$. The extension by zero on $\ker A$ then gives
the result on $H$.
\end{proof}

The estimate at complex times allows us to apply the trapezoidal
rule in logarithmic time with an error bound uniform over
the systems under consideration.

\begin{proposition}
\label{prop:quadrature}
For $h,\sigma>0$, the series of positive operators
\(
 Q_{h,\sigma}=\sum_{j\in\mathbb Z}h\sigma e^{jh}
 A^{\sigma e^{jh}}BA^{\sigma e^{jh}}
\)
converges strongly to a bounded operator. For $0<d<\pi/2$,
\begin{equation}\label{eq:quadrature}
 \|Q_{h,\sigma}-G_c\|\leq bE(h,d),\qquad
 E(h,d)=\frac{2K_0}{\cos d\,(e^{2\pi d/h}-1)}.
\end{equation}
\end{proposition}

\begin{proof}
Let $x\in H_+$ have spectral support in
$[\alpha,\beta]\subset(0,\infty)$ with respect to $L$, and set
$f_x(z)=e^z\langle A^{e^z}BA^{e^z}x,x\rangle$.
This function is analytic in $|\operatorname{Im}z|<\pi/2$.
For $z=s+i\theta$ and $t=e^s$,
$|f_x(s+i\theta)|\leq
t\|CA^{te^{i\theta}}x\|\,\|CA^{te^{-i\theta}}x\|$.
Integrating this bound and applying Lemma~\ref{lem:complex-energy}
gives
\[
 \int_{\mathbb R}|f_x(s+i\theta)|\,ds
 \leq\frac{K_0b}{\cos\theta}\|x\|^2.
\]
Since $B\leq bI$, we also have
$|f_x(s+i\theta)|
\leq be^s\exp(-2\alpha e^s\cos d)\|x\|^2$
for $|\theta|\leq d$.
Thus $f_x(s+i\theta)\to0$ uniformly for $|\theta|\leq d$
as $s\to\pm\infty$.

Applying the trapezoidal estimate for analytic functions in a strip
\cite[Theorem~5.1]{TrefethenWeideman2014}
to the function $z\mapsto f_x(z+\log\sigma)$ yields
\[
 \left|h\sum_{j\in\mathbb Z}f_x(\log\sigma+jh)
       -\int_{\mathbb R}f_x(s)\,ds\right|
 \leq bE(h,d)\|x\|^2.
\]
The integral equals $\langle G_cx,x\rangle$.
Since $G_c\leq bI$ and every summand is positive, each finite
partial sum has quadratic form at most
$b(1+E(h,d))\|x\|^2$ for these vectors.
Such vectors are dense in $H_+$, and all summands vanish on
$\ker A$. Hence every finite partial sum is bounded above by
$b(1+E(h,d))I$ on $H$.
Monotone convergence of positive operators gives strong convergence
of the series. Finally, the quadratic form error estimate extends
by density to $H_+$ and, since both operators vanish on $\ker A$,
to $H$. This proves \eqref{eq:quadrature}.
\end{proof}

The quadrature nodes in Proposition~\ref{prop:quadrature} accumulate at zero, so the full quadrature measure is not an admissible sampling design. Truncating the sequence of nodes removes this accumulation. We then translate the remaining nodes and add observations at finitely many initial integer times to approximate the original
integer Gramian.

\begin{theorem}
\label{thm:relative}
For every $\kappa\geq1$ and $0<\varepsilon<1$, there is a positive
atomic measure $\mu$ with locally finite support such that
$(1-\varepsilon)S\leq G_\mu\leq(1+\varepsilon)S$
for every positive system whose original frame bounds satisfy
$b/a\leq\kappa$.
More precisely, for any $0<d<\pi/2$, choose
\[
 M\in\mathbb N,\qquad M\geq\frac{4\kappa}{\varepsilon},\qquad
 \sigma=\frac{\varepsilon}{4\kappa},\qquad
 h=\frac{2\pi d}{\log(1+8K_0\kappa/(\varepsilon\cos d))},
\]
and set
\(
 \mu=\sum_{n=0}^{M-1}\delta_n
       +\sum_{j\geq0}h\sigma e^{jh}\delta_{M+\sigma e^{jh}}.
\)
This measure satisfies the stated relative bounds and
$N_\mu(R)=h^{-1}\log R+O_{M,\kappa,\varepsilon,d}(1)$
as $R\to\infty$.
\end{theorem}

\begin{proof}
The Stein identity gives
$S=\sum_{n=0}^{M-1}A^nBA^n+A^MSA^M$.
For $0\leq s\leq1$, spectral calculus yields
\[
 \|A^{M+s}-A^M\|
 \leq\sup_{\ell\geq0}e^{-M\ell}(1-e^{-s\ell})
 \leq\frac{s}{eM}.
\]
Consequently, \eqref{eq:continuous} gives
$\|A^M(G_c-S)A^M\|\leq b/(eM)\leq b/M$.

Removing the terms with $j<0$ from $Q_{h,\sigma}$ discards total
weight $h\sigma/(e^h-1)\leq\sigma$.
Since $\|A^tBA^t\|\leq b$ for $t>0$, the discarded sum has norm
at most $b\sigma$.
Translating the remaining sampling times by $M$ and adding
observations at $0,\ldots,M-1$ gives the stated measure.
The preceding bounds and \eqref{eq:quadrature} imply
$\|G_\mu-S\|\leq b\bigl(M^{-1}+\sigma+E(h,d)\bigr)$.
With the chosen parameters,
$E(h,d)=\sigma=\varepsilon/(4\kappa)$ and
$M^{-1}\leq\varepsilon/(4\kappa)$.
Hence
$\|G_\mu-S\|\leq3\varepsilon b/(4\kappa)\leq\varepsilon a$.
Since $S\geq aI$, this proves the relative bounds.

The support is locally finite. For $R\geq M+\sigma$,
its counting function is
$N_\mu(R)=M+1+\lfloor h^{-1}\log((R-M)/\sigma)\rfloor$,
which gives the stated asymptotic formula.
\end{proof}

\begin{corollary}
\label{cor:large-upper}
For every $\kappa\geq1$,
\[
 C^*(\kappa)\leq
 \inf_{0<d<\pi/2}
 \frac{\log(1+16K_0\kappa/\cos d)}{2\pi d}.
\]
In particular,
$C^*(\kappa)\leq
(\log\kappa+\log\log\kappa+O(1))/\pi^2$
as $\kappa\to\infty$.
\end{corollary}

\begin{proof}
The first bound follows from Theorem~\ref{thm:relative}
by taking $\varepsilon=1/2$ and then the infimum over $d$.
For sufficiently large $\kappa$, choose
$d=\pi/2-1/\log\kappa$.
Then $\cos d\sim1/\log\kappa$ and
$2\pi d=\pi^2-2\pi/\log\kappa$.
Substituting this choice of $d$ gives the asymptotic estimate.
\end{proof}

\subsection{Lower bounds from a geometric spectral model}

The construction above gives upper bounds. To obtain lower bounds, we use a single operator orbit with eigenvalues $\lambda_j$ chosen so that $(1-\lambda_j)/(1+\lambda_j)$ forms a geometric sequence. We show that every sufficiently long interval in logarithmic time must contain enough sampling times to distinguish vectors in the corresponding spectral subspace.

\begin{proposition}
\label{prop:cayley}
Fix $h>0$. Let $(e_j)_{j\geq1}$ be the standard orthonormal basis
of $H=\ell^2(\mathbb N)$. Set
$d_j=e^{-hj}/2$, $\lambda_j=(1-d_j)/(1+d_j)$, and
$f_j=\sqrt{1-\lambda_j^2}$.
Define $Ae_j=\lambda_je_j$ and $f=(f_j)_{j\geq1}$.
Then $f\in H$, $A$ is a positive invertible contraction, and
$(A^nf)_{n\geq0}$ is a frame for $H$.
Its frame operator $S_h$ has matrix entries
\[
 (S_h)_{ij}=\frac{f_if_j}{1-\lambda_i\lambda_j}
 =\frac{2\sqrt{d_id_j}}{d_i+d_j}
 =\operatorname{sech}\frac{(i-j)h}{2}.
\]
The associated Toeplitz symbol satisfies
\[
 p_h(\theta)
 =\sum_{n\in\mathbb Z}\operatorname{sech}(nh/2)e^{in\theta}
 =\frac{2\pi}{h}\sum_{k\in\mathbb Z}
   \operatorname{sech}\frac{\pi(\theta+2\pi k)}h>0
 \qquad \theta\in\mathbb R.
\]
For $0<h\leq1$, we have
$(2\pi/h)e^{-\pi^2/h}I\leq S_h\leq(4\pi/h)I$
and hence $\operatorname{cond}(S_h)\leq2e^{\pi^2/h}$.
\end{proposition}

\begin{proof}
Since $f_j^2=4d_j/(1+d_j)^2\leq4d_j$ and
$\sum_{j\geq1}d_j<\infty$, we have $f\in H$.
The inequalities $0<\lambda_1\leq\lambda_j<1$ show that $A$ is
a positive invertible contraction. Dominated convergence also gives
$A^n\to0$ strongly.
Summing the sampling energy at integer times for finitely supported
vectors gives the stated matrix entries of $S_h$.

With the Fourier transform convention
$\widehat g(\xi)=\int_{\mathbb R}g(x)e^{-i\xi x}\,dx$,
Euler's beta integral gives
\[
 \int_{\mathbb R}\operatorname{sech}(cx)e^{-i\xi x}\,dx
 =\frac{\pi}{c}\operatorname{sech}\frac{\pi\xi}{2c}
 \qquad c>0,\quad \xi\in\mathbb R.
\]
Poisson summation gives the formula for $p_h$.
For $-\pi\leq\theta\leq\pi$, the term with $k=0$ gives
$p_h(\theta)\geq(2\pi/h)e^{-\pi^2/h}$.
Since the Fourier coefficients are positive,
$p_h(\theta)\leq p_h(0)$, and
\[
 p_h(0)\leq\frac{2\pi}{h}
 \left(1+\frac{4}{e^{2\pi^2/h}-1}\right)
 \leq\frac{4\pi}{h}
 \qquad 0<h\leq1.
\]
For every fixed $h>0$, the symbol is continuous and strictly positive
on $[-\pi,\pi]$, so it is bounded above and away from zero.
The Toeplitz quadratic form identity
\[
 \langle S_hx,x\rangle
 =\frac1{2\pi}\int_{-\pi}^{\pi}p_h(\theta)
 \left|\sum_{j\geq1}x_je^{ij\theta}\right|^2\,d\theta
\]
gives the frame inequalities for finitely supported vectors.
These inequalities extend to all of $H$ by density and also give
the stated bounds for $S_h$.
\end{proof}

For a locally finite set $T\subset[0,\infty)$, define its lower
logarithmic density by
\begin{equation}\label{eq:lower-log-density}
 \underline D_{\log}(T)
 =\liminf_{L\to\infty}\inf_{R\geq1}
   \frac{\#(T\cap(R,e^LR])}{L}.
\end{equation}
This value may be $+\infty$. The infimum over $R$ ensures that
the asymptotic lower bound as $L\to\infty$ holds uniformly
for $R\geq1$.

\begin{theorem}
\label{thm:local-density}
Fix $h>0$ and consider the system in Proposition~\ref{prop:cayley}.
If a positive atomic measure $\mu$ with locally finite support
gives a frame for this system, then there is a constant
$C_{\mu,h}>0$ such that
\begin{equation}\label{eq:local-density}
 \#\bigl(\operatorname{supp}\mu\cap(R,e^LR]\bigr)
 \geq L/h-C_{\mu,h}
 \qquad R\geq1,\quad L>0.
\end{equation}
In particular,
$\liminf_{R\to\infty}N_\mu(R)/\log R\geq1/h$.
\end{theorem}

\begin{proof}
Let $a_\mu,b_\mu$ be the frame bounds after sampling, and set
$s_j=-\log\lambda_j$.
We use the elementary estimates
$2d_j\leq s_j\leq(8/3)d_j$,
$(16/9)d_j\leq f_j^2\leq4d_j$, and $f_j^2\leq2s_j$.
For sufficiently large $R$, choose the largest $j$ such that
$d_jR\geq1$. Then $1\leq d_jR<e^h$, and the upper frame bound
applied to $e_j$ gives
\[
 b_\mu\geq f_j^2\int_{[0,R]}e^{-2s_jt}\,d\mu(t)
 \geq\frac{16}{9R}e^{-(16/3)e^h}M_\mu(R).
\]
Increasing the constant to cover bounded values of $R$, we obtain
$M_\mu(R)\leq KR$ for $R\geq1$, where $K>0$ depends only on
$\mu$ and $h$.
Integration by parts and $f_j^2\leq2s_j$ then give
\begin{equation}\label{eq:model-tail}
 f_j^2\int_{(Q,\infty)}e^{-2s_jt}\,d\mu(t)
 \leq K(1+2s_jQ)e^{-2s_jQ}
 \qquad Q\geq1.
\end{equation}

Fix $D>h+\log2$, to be increased below.
For $R\geq1$, let $j_0$ be the first index such that
$d_{j_0}\leq e^{-D}/R$.
Then $j_0>1$ and
$e^{-D-h}/R<d_{j_0}\leq e^{-D}/R$.
Suppose $L\geq2D+h$, and set
$m=\lfloor(L-2D-h)/h\rfloor+1$ and
$E=\operatorname{span}\{e_{j_0},\ldots,e_{j_0+m-1}\}$.
If $(R,e^LR]$ contains fewer than $m$ sampling times, then,
since $\dim E=m$, there is a unit vector $x\in E$ whose
observations vanish at every sampling time in this interval.
The energy before the interval satisfies
\begin{equation}\label{eq:model-left-energy}
 \int_{[0,R]}|\langle x,A^tf\rangle|^2\,d\mu(t)
 \leq M_\mu(R)\sum_{j=j_0}^{j_0+m-1}f_j^2
 \leq\frac{4Ke^{-D}}{1-e^{-h}}.
\end{equation}
Set $Q=e^LR$.
Since $d_{j_0+m-1}Q>e^D$, we have
$2s_jQ>4e^De^{rh}$ for $j=j_0+m-1-r$ and
$r=0,\ldots,m-1$.
Using \eqref{eq:model-tail},
$(1+v)e^{-v}\leq2e^{-v/2}$ for $v\geq0$, and
$e^{rh}\geq1+rh$, we obtain
\begin{equation}\label{eq:model-right-energy}
 \begin{split}
 \int_{(Q,\infty)}|\langle x,A^tf\rangle|^2\,d\mu(t)
 \leq2K\sum_{r\geq0}e^{-2e^De^{rh}}
 \leq\frac{2Ke^{-2e^D}}{1-e^{-2he^D}}.
 \end{split}
\end{equation}
Both upper bounds are independent of $R,L,m$ and tend to zero
as $D\to\infty$.
Choose $D$, depending only on $\mu$ and $h$, so that their sum
is less than $a_\mu$.
Together with the vanishing observations in $(R,Q]$, this
contradicts the lower frame bound.
Thus $(R,e^LR]$ contains at least
$m\geq(L-2D-h)/h$ sampling times.
For $L<2D+h$, the same lower bound follows from the
nonnegativity of the count.
Taking $C_{\mu,h}=(2D+h)/h$ proves \eqref{eq:local-density}.

Finally, setting $R=1$ gives
$N_\mu(e^L)\geq L/h-C_{\mu,h}$.
Dividing by $L$ and letting $L\to\infty$ proves the last assertion.
\end{proof}

\begin{corollary}
\label{cor:positive-lower}
For every $\kappa>1$,
\[
 C^*(\kappa)\geq
 \frac1{2\log(1+4(\kappa+1)/(\kappa-1))}>0.
\]
\end{corollary}

\begin{proof}
Let $\delta=(\kappa-1)/(\kappa+1)$ and
$h_\kappa=2\log(1+4/\delta)$.
In each row and column of $S_{h_\kappa}$, the sum of the absolute
values of the entries outside the main diagonal is at most
\[
 2\sum_{n\geq1}\operatorname{sech}(nh_\kappa/2)
 \leq\frac4{e^{h_\kappa/2}-1}=\delta.
\]
Since the diagonal entries are one, the Schur estimate gives
$(1-\delta)I\leq S_{h_\kappa}\leq(1+\delta)I$.
Hence
$\operatorname{cond}(S_{h_\kappa})
\leq(1+\delta)/(1-\delta)=\kappa$.
Every $\mu\in\mathcal U_w(\kappa)$ must give a frame for this
system, so Theorem~\ref{thm:local-density} yields
$D(\mu)\geq1/h_\kappa$.
Taking the infimum over $\mu\in\mathcal U_w(\kappa)$ proves
the stated bound.
\end{proof}

A sharper estimate of the condition number of the geometric
model gives a lower bound with the same leading constant
as the upper bound when $\kappa\to\infty$.

\begin{theorem}
\label{thm:large-lower}
For every $\kappa\geq2$,
\begin{equation}\label{eq:large-lower-refined}
 C^*(\kappa)\geq
 \frac{\log\kappa+\log4-\kappa^{-2}}{\pi^2}.
\end{equation}
Consequently, as $\kappa\to\infty$,
\[
 \frac{\log\kappa+\log4-\kappa^{-2}}{\pi^2}
 \leq C^*(\kappa)
 \leq\frac{\log\kappa+\log\log\kappa+O(1)}{\pi^2},
\]
and $C^*(\kappa)/\log\kappa\to1/\pi^2$.
\end{theorem}

\begin{proof}
Consider the system in Proposition~\ref{prop:cayley}, and set
$x=\pi^2/h$ and $r=e^{-2x}$.
Positivity of the Fourier coefficients and the Poisson representation
of the symbol give
\(
 p_h(\theta)\leq p_h(0)
 \leq(2\pi/h)(1+4r/(1-r)).
\)
For the lower bound, we first show that
$\operatorname{sech}a+\operatorname{sech}(2x-a)
\geq2\operatorname{sech}x$
for $x\geq\operatorname{arcosh}2$ and $0\leq a\leq x$.
Set $c=\cosh x\geq2$ and $z=\cosh(x-a)\in[1,c]$.
The left side equals $2cz/(c^2+z^2-1)$.
Its derivative with respect to $z$ has the sign of
$c^2-1-z^2$, so its minimum on $[1,c]$ occurs at an endpoint.
At $z=1$ the value is $2/c$, and at $z=c$ it is
$2c^2/(2c^2-1)\geq1\geq2/c$.
This proves the inequality.

For $-\pi\leq\theta\leq\pi$, the two arguments closest to zero
in the Poisson sum have absolute values
$a=\pi|\theta|/h$ and $2x-a$.
Hence, for $x\geq\operatorname{arcosh}2$,
\(
 p_h(\theta)\geq(4\pi/h)\operatorname{sech}x=8\pi e^{-x}/(h(1+r)).
\)
The Toeplitz quadratic form identity therefore gives
\[
 \operatorname{cond}(S_h)
 \leq\frac{e^x}{4}
 \left(1+\frac{4r}{1-r}\right)(1+r).
\]

For $\kappa\geq2$, choose
$x=\log(4\kappa)-\kappa^{-2}$ and $h=\pi^2/x$.
Then $x\geq\log8-1/4>\operatorname{arcosh}2$ and
\(
 r=e^{2\kappa^{-2}}/(16\kappa^2)<1/(8\kappa^2)\leq1/32<1/9.
\)
For $0\leq r\leq1/9$, we have
\[
 \left(1+\frac{4r}{1-r}\right)(1+r)
 =1+5r+\frac{8r^2}{1-r}
 \leq1+6r\leq e^{6r}.
\]
Since $6r<\kappa^{-2}$, the condition number bound gives
$\operatorname{cond}(S_h)
\leq\kappa\exp(-\kappa^{-2}+6r)\leq\kappa$.
Every $\mu\in\mathcal U_w(\kappa)$ must therefore give a frame
for this system.
Theorem~\ref{thm:local-density} yields
$D(\mu)\geq1/h=x/\pi^2$.
Taking the infimum over $\mu\in\mathcal U_w(\kappa)$ proves
\eqref{eq:large-lower-refined}.
Corollary~\ref{cor:large-upper} supplies the upper bound,
and the stated limit follows.
\end{proof}

These bounds determine the leading constant but do not determine the next term in the asymptotic expansion as $\kappa\to\infty$. For the fixed geometric model, the coefficient $1/h$ in the local counting bound is optimal: Theorem~\ref{thm:fixed-model-optimum} constructs sampling designs with real or integer times that attain this coefficient. The sampling times and weights may depend on
the particular system $(A,f)$. They are not required to work for every system with the same bound on the original frame condition number.

\subsection{Tight and nearly tight frames}

We first show that the infimum $C^*(1)$ is zero but is not attained.
We then combine a relative perturbation estimate with the preceding
lower bounds to determine the asymptotic behavior of $C^*(\kappa)$
as $\kappa\downarrow1$.

\begin{theorem}
\label{thm:tight-endpoint}
We have $C^*(1)=0$, but $D(\mu)>0$ for every
$\mu\in\mathcal U_w(1)$. Moreover,
$\mathcal U_w(\kappa)\subsetneq\mathcal U_w(1)$
for every $\kappa>1$.
\end{theorem}

\begin{proof}
For $q>1$, set
\(
 \nu_q=\delta_0+\sum_{j\geq0}(q-1)q^j\delta_{q^j}.
\)
For $0\leq R<1$, we have $M_{\nu_q}(R)=1$.
If $q^J\leq R<q^{J+1}$ with $J\geq0$, then
$M_{\nu_q}(R)=q^{J+1}$.
Thus $R\leq M_{\nu_q}(R)\leq1+qR$ for all $R\geq0$.
Proposition~\ref{prop:tight} gives $\nu_q\in\mathcal U_w(1)$.
Since $D(\nu_q)=1/\log q$, letting $q\to\infty$ proves
$C^*(1)=0$.

For later use, define
$Q_q(u)=\int_{[0,\infty)}e^{-ut}\,d\nu_q(t)$ for $u>0$.
Integration by parts and the mass bound give
$Q_q(u)\leq1+q/u$.
Using $\phi_{\nu_q}(e^{-u})=(1-e^{-u})Q_q(u)$ and
$(1-e^{-u})/u\leq1$, together with $\phi_{\nu_q}(0)=1$,
we obtain $0<\phi_{\nu_q}(z)\leq1+q$ for $0\leq z<1$.

Now fix $\mu\in\mathcal U_w(1)$.
By Proposition~\ref{prop:tight}, there are constants
$0<c_1\leq c_2$ and $R_0\geq1$ such that
$c_1R\leq M_\mu(R)\leq c_2R$ for $R\geq R_0$.
Choose $\Lambda>c_2/c_1$.
Every interval $(R,\Lambda R]$ with $R\geq R_0$ contains
a sampling time. Otherwise,
$c_1\Lambda R\leq M_\mu(\Lambda R)=M_\mu(R)\leq c_2R$,
a contradiction.
Counting the disjoint intervals
$(\Lambda^nR_0,\Lambda^{n+1}R_0]$ for $n\geq0$ gives
$\liminf_{R\to\infty}N_\mu(R)/\log R
\geq1/\log\Lambda>0$.
Hence $D(\mu)>0$.

Finally, the inclusion
$\mathcal U_w(\kappa)\subseteq\mathcal U_w(1)$
follows from the definitions.
For a fixed $\kappa>1$, choose $q$ so large that
$1/\log q$ is smaller than the lower bound in
Corollary~\ref{cor:positive-lower}.
Then $\nu_q\notin\mathcal U_w(\kappa)$.
Since $\nu_q\in\mathcal U_w(1)$, the inclusion is strict.
\end{proof}

To determine the asymptotic behavior of $C^*(\kappa)$ as
$\kappa\downarrow1$, we compare the sampled Gramian with
$\phi_{\nu_q}(A^2)$ on the spectral subspace of $A$ near $1$.
This operator is the sampled Gramian when $S=I$.
Normalizing by this operator gives a sharper perturbation estimate than directly bounding the difference between the two Gramians.

\begin{lemma}
\label{lem:slow-scalar}
For $q\geq2$, set $d_q=\log q/(8e)$.
Then $\phi_{\nu_q}(e^{-u})\geq d_q$ for $0<u\leq1/q$.
\end{lemma}

\begin{proof}
Choose an integer $j\geq1$ such that $v=uq^j\in[1,q]$.
The atoms at $q^j$ and $q^{j-1}$ give
\[
 uQ_q(u)\geq(q-1)ve^{-v}
 +\frac{q-1}{q}ve^{-v/q}.
\]
If $v\geq(\log q)/2$, then $v/q\leq1$ and
$(q-1)/q\geq1/2$, so the second term is at least
$v/(2e)\geq\log q/(4e)$.
If $v<(\log q)/2$, then $v\geq1$ and $q-1\geq q/2$,
so the first term is at least
$\sqrt q/2\geq\log q/2$.
Thus $uQ_q(u)\geq\log q/(4e)$ in both cases.

Since $u\leq1/q\leq1/2$, we have $(1-e^{-u})/u\geq1/2$.
Combining this with
$\phi_{\nu_q}(e^{-u})=((1-e^{-u})/u)\,uQ_q(u)$
proves the stated bound.
\end{proof}

\begin{lemma}
\label{lem:relative-perturbation}
Suppose $S=I+E$, $E=E^*$, $\|E\|\leq\delta\leq1/2$, and
$B=S-ASA\geq0$.
Assume also that $E_A(\{1\})=0$ and
$\sigma(A)\subset[e^{-1/(2q)},1]$, where $q\geq2$.
Set $G_q=G_{\nu_q}$ and $D=\phi_{\nu_q}(A^2)$.
There is an absolute constant $C_1>0$ such that
\[
 d_qI\leq D\leq(q+1)I,\qquad
 \|D^{-1/2}G_qD^{-1/2}-I\|
 \leq C_1\sqrt\delta\,q^{1/4}.
\]
\end{lemma}

\begin{proof}
The mass bound for $\nu_q$ and Proposition~\ref{prop:mass}
imply that $G_q$ is bounded.
The spectral assumptions, Lemma~\ref{lem:slow-scalar}, and
the bound $\phi_{\nu_q}\leq1+q$ give the stated bounds for $D$.
If $\delta=0$, then $B=I-A^2$ and $G_q=D$, so the perturbation
estimate holds. We therefore assume $\delta>0$ below.

Set $u_i=2^{-i}$ and
$P_i=E_A([1-u_i,1-u_i/2))$ for $i\geq0$.
Since $E_A(\{1\})=0$, these projections sum strongly to $I$.
Write $A_i=A|_{P_iH}$ and $X_{ij}=P_iX|_{P_jH}$ for the blocks
of a bounded operator $X$.
Set $\beta=\|S\|\leq2$.
The Stein identity and positivity of $B$ give
$\|B_{ii}\|\leq2\beta u_i$ and
$\|B_{ij}\|\leq2\beta\sqrt{u_iu_j}$.
We also have $I-A_i^2\geq(3u_i/4)I$.
For $D_i=D|_{P_iH}$ and $U_i(t)=D_i^{-1/2}A_i^t$,
spectral calculus gives
\begin{equation}\label{eq:normalization-cancellation}
 \int_{[0,\infty)}U_i(t)^2\,d\nu_q(t)
 =(I-A_i^2)^{-1}\leq\frac4{3u_i}I.
\end{equation}
Let $X:P_jH\to P_iH$ be bounded, and let
$F\subset[0,\infty)$ be measurable.
Factoring through the associated analysis operators gives
\begin{equation}\label{eq:block-cauchy}
 \left\|\int_F U_i(t)XU_j(t)\,d\nu_q(t)\right\|
 \leq\|X\|
 \left\|\int_F U_i(t)^2\,d\nu_q(t)\right\|^{1/2}
 \left\|\int_F U_j(t)^2\,d\nu_q(t)\right\|^{1/2}.
\end{equation}
The integrals are understood as strong limits of finite sums.

Write $B=I-A^2+W$, where $W=E-AEA$, and set
$Z=D^{-1/2}(G_q-D)D^{-1/2}$.
The identity
$W_{ij}=(I-A_i)E_{ij}+A_iE_{ij}(I-A_j)$
gives $\|W_{ij}\|\leq\delta(u_i+u_j)$.
Applying \eqref{eq:block-cauchy} and
\eqref{eq:normalization-cancellation}, with $d=|i-j|$, yields
\begin{equation}\label{eq:near-block}
 \|Z_{ij}\|
 \leq\frac{4\delta(u_i+u_j)}{3\sqrt{u_iu_j}}
 \leq\frac83\delta2^{d/2}.
\end{equation}

For $i\neq j$, assume by symmetry that $u=u_i\geq v=u_j$.
Set $\ell=\log(q+1)$ and $T_0=8\ell/u$.
Apply \eqref{eq:block-cauchy}, using
\eqref{eq:normalization-cancellation} for the factor involving $i$
and $D_j\geq d_qI$ for the factor involving $j$.
This gives
\[
 \begin{split}
 \left\|\int_{[0,T_0]}U_i(t)B_{ij}U_j(t)\,d\nu_q(t)\right\|
 \leq\frac{4\beta}{\sqrt{3d_q}}
       \sqrt{vM_{\nu_q}(T_0)}
 \leq C_2\sqrt q\sqrt{v/u}.
 \end{split}
\]
Here we used $M_{\nu_q}(T_0)\leq1+qT_0$, $u\leq1$,
and $\ell/d_q\leq16e$.

For the remaining times, use $\|A_i^t\|\leq e^{-tu/2}$.
Stieltjes integration by parts and the mass bound give
\[
 \int_{(T_0,\infty)}e^{-tu/2}\,d\nu_q(t)
 \leq e^{-uT_0/2}(1+qT_0+2q/u).
\]
Consequently,
\[
 \begin{split}
 \left\|\int_{(T_0,\infty)}
 U_i(t)B_{ij}U_j(t)\,d\nu_q(t)\right\|
 \leq\frac{2\beta}{d_q}\sqrt{v/u}
       (q+1)^{-4}(u+8q\ell+2q)
 \leq C_3\sqrt q\sqrt{v/u}.
 \end{split}
\]
Since $D$ is block diagonal, these two estimates imply
$\|Z_{ij}\|\leq C_4\sqrt q\,2^{-d/2}$ for $i\neq j$,
where $C_4$ is absolute.

For $0<a_0\leq b_0$, splitting the sum near
$d=\log_2(b_0/a_0)$ gives
\(
 \sum_{d\geq1}
 \min\{a_0 2^{d/2},b_0 2^{-d/2}\}
 \leq7\sqrt{a_0b_0}.
\)
Apply this with $a_0=8\delta/3$ and $b_0=C_4\sqrt q$,
increasing $C_4$ if necessary.
By \eqref{eq:near-block}, the diagonal blocks satisfy
$\|Z_{ii}\|\leq8\delta/3$.
The Schur estimate for operator blocks therefore gives
\[
 \|Z\|\leq\frac83\delta+14\sqrt{a_0b_0}
 \leq C_1\sqrt\delta\,q^{1/4}.
\]
\end{proof}

The following lemma extends a uniform estimate on a spectral
subspace near $1$ to the whole space.

\begin{lemma}
\label{lem:spectral-completion}
Fix $\kappa\geq1$ and $s,\gamma,M>0$.
Suppose a sampling design $\nu$ satisfies $G_\nu\leq MI$ and
$PG_\nu P\geq\gamma P$ for every system with
$I\leq S\leq\kappa I$, where $P=E_A([e^{-s},1])$.
Set $\lambda=2\kappa/\gamma$ and choose an integer $m\geq1$
such that $(\kappa+\lambda M)e^{-2ms}\leq1/2$.
Then $\rho=\sum_{n=0}^{m-1}\delta_n+\lambda\tau_m\nu$
belongs to $\mathcal U_w(\kappa)$ and satisfies
$D(\rho)=D(\nu)$.
For $I\leq S\leq\kappa I$, its Gramian satisfies
$(1/2)I\leq G_\rho\leq(\kappa+\lambda M)I$.
\end{lemma}

\begin{proof}
Set $Q=I-P$.
Applying $\|u+v\|^2\geq\|u\|^2/2-\|v\|^2$
to the sampled images of $Px$ and $Qx$ gives
$G_\nu\geq(\gamma/2)P-MQ$.
Since $S\leq\kappa I$, we have
$\lambda G_\nu-S\geq-(\kappa+\lambda M)Q$.
The Stein identity and $\|AQ\|\leq e^{-s}$ therefore yield
\[
 G_\rho=S+A^m(\lambda G_\nu-S)A^m
 \geq I-(\kappa+\lambda M)e^{-2ms}I\geq\frac12I.
\]
The observations at $0,\ldots,m-1$ have Gramian at most $S$, which gives the
upper bound. All parameters are uniform over the normalized
class, so rescaling proves $\rho\in\mathcal U_w(\kappa)$.
The prefix and translated support are disjoint, and
$N_\rho(R)=m+N_\nu(R-m)$ for $R\geq m$.
Thus $D(\rho)=D(\nu)$, including when this value is infinite.
\end{proof}

We now combine the preceding estimates with the lower bound
in Corollary~\ref{cor:positive-lower} to determine the behavior
of $C^*(\kappa)$ as $\kappa\downarrow1$.

\begin{theorem}
There are absolute constants $b_0,\delta_0>0$ such that, if
$0<\delta=(\kappa-1)/(\kappa+1)<\delta_0$ and
$q=\lfloor b_0\delta^{-2}\rfloor$, then $q\geq2$ and
\[
 \frac1{2\log(1+4/\delta)}
 \leq C^*(\kappa)\leq\frac1{\log q}.
\]
There exists $\rho\in\mathcal U_w(\kappa)$ with
$D(\rho)=1/\log q$.
In particular,
$1/C^*(\kappa)=2\log(1/(\kappa-1))+O(1)$
as $\kappa\downarrow1$.
\end{theorem}

\begin{proof}
Choose $b_0>0$ so that $C_1b_0^{1/4}\leq1/2$.
Then choose $0<\delta_0\leq1/2$ so that
$q\geq2$ and $q\geq b_0\delta^{-2}/2$
whenever $0<\delta<\delta_0$.
The definition of $q$ also gives
$C_1\sqrt\delta\,q^{1/4}\leq1/2$.

By scaling, it suffices to consider systems with
$I\leq S\leq\kappa I$.
Set $s=1/(2q)$ and $P=E_A([e^{-s},1])$.
Compress $S$ and $B$ to $PH$, and divide both operators
by $(\kappa+1)/2$.
On $PH$, the rescaled Gramian has the form $I+E$ with
$\|E\|\leq\delta$.
Applying Lemma~\ref{lem:relative-perturbation} and restoring
the factor $(\kappa+1)/2\geq1$ gives
$PG_{\nu_q}P\geq\gamma P$, where
$\gamma=d_q/2=\log q/(16e)$.

The mass bound for $\nu_q$ and Proposition~\ref{prop:mass}
give $G_{\nu_q}\leq\kappa(1+\sqrt2q)I$.
Lemma~\ref{lem:spectral-completion} therefore gives
$\rho\in\mathcal U_w(\kappa)$ with
$D(\rho)=D(\nu_q)=1/\log q$.
Since $q$ is an integer, both $\nu_q$ and the resulting $\rho$
are supported on $\mathbb N_0$.

The lower bound follows from Corollary~\ref{cor:positive-lower}.
Finally,
\[
 2\log(1/\delta)+\log(b_0/2)
 \leq\log q\leq\frac1{C^*(\kappa)}
 \leq2\log(1/\delta)+2\log(4+\delta).
\]
Since $\delta=(\kappa-1)/(\kappa+1)$, these inequalities
give the stated asymptotic formula.
\end{proof}

\section{Realization and stability of universal sampling designs}
\label{s4-section}

The preceding constructions allow real sampling times.
We first prove $C^*_{\mathbb Z}(\kappa)=C^*(\kappa)$ by rounding sufficiently late sampling times and adding observations at finitely many initial integer times.
We then consider fully universal designs, which work for every finite bound on the original frame condition number.
Every such design satisfies $N_\mu(R)/\log R\to\infty$ as
$R\to\infty$. We construct designs supported on $\mathbb N_0$
for which this divergence is arbitrarily slow.
Appendix~\ref{app:prescribed-supports} treats relative perturbations of sampling times and sparsification within prescribed supports;
the rounding argument here is independent of those results.

The final two subsections establish stability under small increases in $\kappa$ and show that the infimum $C^*(\kappa)$ is never attained. These estimates are also used in Section~\ref{sec:output} to prove that prescribing a finite upper bound on the sampled condition number strictly increases the optimal counting coefficient.

\subsection{Integer realization}
\label{s4-integer-subsection}
The following estimate allows us to round sufficiently late sampling times to integers while retaining a uniform positive lower frame bound.

\begin{lemma}
\label{s4-rounding}
Let $M>0$, and let $\tau$ be a positive measure supported on
$[M,\infty)$ such that
$W(s):=\tau([0,s])\leq Ks$ for $s\geq M$.
For a system satisfying \eqref{eq:frame}, define
$T_\tau,\widehat T_\tau:H\to L^2(\tau;Y)$ by
$(T_\tau x)(s)=CA^sx$ and
$(\widehat T_\tau x)(s)=CA^{\lceil s\rceil}x$.
Then $\|\widehat T_\tau-T_\tau\|^2\leq bK/M$.
\end{lemma}

\begin{proof}
Let $\widehat\tau$ be the pushforward of $\tau$ under
$s\mapsto\lceil s\rceil$.
Its cumulative mass is no larger than that of $\tau$.
Proposition~\ref{prop:mass} therefore shows that both
$T_\tau$ and $\widehat T_\tau$ are bounded.

For $\lambda=e^{-u}\in(0,1)$, we have
$|\lambda^{\lceil s\rceil}-\lambda^s|\leq ue^{-us}$.
Stieltjes integration by parts, including a possible atom at $M$,
gives
\[
 u^2\int_{[M,\infty)}e^{-2us}\,d\tau(s)
 =2u^3\int_M^\infty e^{-2us}W(s)\,ds
 \leq Ku^2\left(M+\frac1{2u}\right)e^{-2uM}
 \leq\frac K M.
\]
The last inequality follows by setting $v=uM$ and using
$(v^2+v/2)e^{-2v}\leq1$.
At $\lambda=0$ and $\lambda=1$, the difference of the two powers
is zero.
The spectral theorem, Tonelli's theorem, and $\|C\|^2\leq b$
now give the stated operator norm bound.
\end{proof}

Recall that $\tau_L\mu$ denotes the translation of $\mu$ by $L$. We compare $C^*(\kappa)$ with $C^*_{\mathbb Z}(\kappa)$,
the corresponding infimum over sampling designs supported
on $\mathbb N_0$.

\begin{theorem}
\label{s4-integer}
For every $\kappa\geq1$,
\begin{equation*}
 C^*_{\mathbb Z}(\kappa)=C^*(\kappa).
\end{equation*}
For each $\mu\in\mathcal U_w(\kappa)$, there is
$\widehat\nu\in\mathcal U_w(\kappa)$ supported on $\mathbb N_0$
with $D(\widehat\nu)\leq D(\mu)$.
\end{theorem}

\begin{proof}
Fix $\mu=\sum_{t\in T}w_t\delta_t\in\mathcal U_w(\kappa)$.
By scaling, it suffices to consider systems with
$I\leq S\leq\kappa I$.
There are constants $\alpha,\beta>0$, depending only on
$\mu$ and $\kappa$, such that
$\alpha I\leq G_\mu\leq\beta I$ throughout this class.
Since $\mu\in\mathcal U_w(1)$, Proposition~\ref{prop:tight}
and local finiteness give
$M_\mu(R)\leq K(1+R)$ for all $R\geq0$,
where $K>0$ depends only on $\mu$.

Choose $\eta\geq\max\{1,\kappa/\alpha\}$ and, for an integer
$M\geq1$, set
$\nu_M=\sum_{j=0}^{M-1}\delta_j+\eta\tau_M\mu$.
The Stein identity gives
\(
 G_{\nu_M}
 =S-A^MSA^M+\eta A^MG_\mu A^M
 \geq S\geq I,
\)
and $G_{\nu_M}\leq(\kappa+\eta\beta)I$.
Rounding the translated sampling times gives
\[
 \widehat\nu_M
 =\sum_{j=0}^{M-1}\delta_j
 +\eta\sum_{t\in T}w_t\delta_{M+\lceil t\rceil}.
\]
The measure $\eta\tau_M\mu$ is supported on $[M,\infty)$
and has cumulative mass at most
$\eta K(1+s-M)\leq\eta Ks$ for $s\geq M$.

Retaining the original atom indices places the two analysis
operators in a common space.
Lemma~\ref{s4-rounding} bounds the norm of their difference
by $(\eta\kappa K/M)^{1/2}$.
Choose $M$ so that $\eta\kappa K/M\leq1/4$.
The triangle inequality and the bounds for $G_{\nu_M}$ yield
\[
 \frac14 I\leq G_{\widehat\nu_M}
 \leq\left(\sqrt{\kappa+\eta\beta}+\frac12\right)^2I.
\]
Combining weights at coincident rounded times preserves the
sampling energy.
All parameters depend only on $\mu$ and $\kappa$, so scaling
back gives $\widehat\nu_M\in\mathcal U_w(\kappa)$.

For $R\geq M$, we have
$N_{\widehat\nu_M}(R)\leq M+N_\mu(R-M)\leq M+N_\mu(R)$.
Hence $D(\widehat\nu_M)\leq D(\mu)$, and taking infima gives
$C^*_{\mathbb Z}(\kappa)\leq C^*(\kappa)$.
The reverse inequality follows by restricting the infimum
defining $C^*(\kappa)$ to designs supported on $\mathbb N_0$.
\end{proof}

The asymptotic formulas in Section~\ref{sec:fixed} as
$\kappa\downarrow1$ and $\kappa\to\infty$ therefore also hold
for $C^*_{\mathbb Z}(\kappa)$.
Moreover, $C^*_{\mathbb Z}(1)=0$, and this infimum is not attained by Theorem~\ref{thm:tight-endpoint}.

\subsection{Fully universal designs}
\label{s4-universal-subsection}

We now require a single sampling design to work for every finite
bound on the original frame condition number.

\begin{corollary}
\label{s4-superlog-necessary}
Every $\mu\in\mathcal U_w$ satisfies
$N_\mu(R)/\log R\to\infty$ as $R\to\infty$.
\end{corollary}

\begin{proof}
Fix $\mu\in\mathcal U_w$.
For each $h>0$, the system in Proposition~\ref{prop:cayley}
has a finite original frame condition number, so $\mu$ gives
a frame for this system.
Theorem~\ref{thm:local-density} therefore yields
$\liminf_{R\to\infty}N_\mu(R)/\log R\geq1/h$.
Since $h>0$ is arbitrary, the conclusion follows.
\end{proof}

The preceding growth condition is sharp:
$N_\mu(R)/\log R$ can tend to infinity arbitrarily slowly,
even for designs supported on $\mathbb N_0$.
We construct sampling grids, round their nodes to integers,
and combine translated copies with delays tending to infinity.
The resulting support is locally finite.

We first compare a sampling measure with the reference measure
$\nu_*=\delta_0+\mathbf1_{[1/2,\infty)}(t)\,dt$.
Set $G_*=G_{\nu_*}$.
Lemma~\ref{lem:stein} gives $S/\kappa\leq G_*\leq2bI$
whenever $b/a\leq\kappa$.
Keeping the quadratic form $\langle G_*x,x\rangle$ in the
error estimate gives a sharper dependence on $\kappa$.

\begin{lemma}
\label{s4-relative-discrepancy}
Suppose $b/a\leq\kappa$, where $\kappa\geq1$.
Let $\nu$ be a locally finite positive measure on $[0,\infty)$,
and set $D_\nu(t)=M_\nu(t)-M_{\nu_*}(t)$.
Assume that, for some $\varepsilon>0$,
$D_\nu(t)=0$ for $0\leq t<1/2$ and
$|D_\nu(t)|\leq\varepsilon t$ for $t\geq1/2$.
Then $G_\nu$ is bounded with
$\|G_\nu-G_*\|\leq\sqrt2\,b\varepsilon$, and, for every $x\in H$,
\[
 |\langle(G_\nu-G_*)x,x\rangle|
 \leq\varepsilon\sqrt{2b\langle G_*x,x\rangle}\,\|x\|
 \leq\sqrt2\,\varepsilon\kappa\langle G_*x,x\rangle.
\]
\end{lemma}

\begin{proof}
The assumptions give $\nu(\{0\})=1$ and
$\nu((0,R])\leq(1+\varepsilon)R$ for $R>0$.
Proposition~\ref{prop:mass} therefore shows that $G_\nu$ is bounded.

Fix $x\in H$, and set $F(t)=\|CA^tx\|^2$ and
$v_*=\int_{1/2}^\infty F(t)\,dt$.
Lemma~\ref{lem:energy} and the Cauchy--Schwarz inequality give
\[
 \int_{1/2}^\infty t|F'(t)|\,dt
 \leq2\sqrt{v_*}
       \left(\int_{1/2}^\infty
       t^2\|CLA^tx\|^2\,dt\right)^{1/2}
 \leq\sqrt{2bv_*}\,\|x\|.
\]
Since $A^t\to0$ strongly, we have $F(t)\to0$.
Consequently,
$tF(t)\leq\int_t^\infty s|F'(s)|\,ds\to0$,
and hence $D_\nu(t)F(t)\to0$.

The two measures agree on $[0,1/2)$.
Stieltjes integration by parts on $[1/2,T]$, followed by
$T\to\infty$, yields
\[
 \int_{[0,\infty)}F(t)\,d(\nu-\nu_*)(t)
 =-\int_{1/2}^\infty D_\nu(t)F'(t)\,dt.
\]
Here $D_\nu(1/2-)=0$, and a possible atom at $1/2$
is included on the left.
The integral on the right is absolutely convergent.

It follows that
$|\langle(G_\nu-G_*)x,x\rangle|
\leq\varepsilon\sqrt{2bv_*}\,\|x\|$.
Since $v_*\leq b\|x\|^2$, this proves
$\|G_\nu-G_*\|\leq\sqrt2\,b\varepsilon$.
Also, $v_*\leq g:=\langle G_*x,x\rangle$ and
$g\geq(a/\kappa)\|x\|^2$.
Together with $b/a\leq\kappa$, these inequalities give
the stated relative estimate.

If $0<\eta<1$ and
$\varepsilon\leq\eta/(\sqrt2\kappa)$, the relative estimate gives
$G_\nu\geq(1-\eta)G_*\geq(1-\eta)S/\kappa$.
The norm estimate also gives
$G_\nu\leq b(2+\sqrt2\varepsilon)I$.
\end{proof}

\begin{lemma}
\label{s4-refining-grid}
Let $\varrho:(0,\infty)\to[1,\infty)$ be nondecreasing and
locally bounded, and let $0<\varepsilon\leq1$.
Set $t_0=1/2$ and
$t_{k+1}=t_k(1+\varepsilon/\varrho(t_k))$, and define
$\nu_\varepsilon
=\delta_0+\sum_{k\geq0}(t_{k+1}-t_k)\delta_{t_k}$.
Then $\nu_\varepsilon$ has locally finite support and satisfies
$M_{\nu_\varepsilon}(R)\leq1+2R$ for $R\geq0$.
Moreover,
\[
 N_{\nu_\varepsilon}(R)
 \leq\left(2+\frac4\varepsilon\right)\varrho(R)\log R
 \qquad R\geq e.
\]
If $\varrho(R)\to\infty$, then $M_{\nu_\varepsilon}(R)\sim R$.
If $b/a\leq\kappa$ and
$\varepsilon\leq(2\sqrt2\kappa)^{-1}$, then
\(
 S/(2\kappa)\leq G_{\nu_\varepsilon}
 \leq(2+\sqrt2)bI.
\)
\end{lemma}

\begin{proof}
If $(t_k)$ were bounded, local boundedness of $\varrho$
would give $t_{k+1}/t_k\geq1+c$ for some $c>0$,
a contradiction.
Thus $t_k\to\infty$, and the support is locally finite.

For $R\geq1/2$, choose $m$ such that $t_m\leq R<t_{m+1}$.
Telescoping gives
$M_{\nu_\varepsilon}(R)=1/2+t_{m+1}$, while
$M_{\nu_*}(R)=1/2+R$.
Hence
\[
 0\leq M_{\nu_\varepsilon}(R)-M_{\nu_*}(R)
 =t_{m+1}-R
 \leq\frac{\varepsilon t_m}{\varrho(t_m)}
 \leq\varepsilon R.
\]
For $0\leq R<1/2$, both cumulative masses equal one.
These estimates give $M_{\nu_\varepsilon}(R)\leq1+2R$.
If $\varrho(R)\to\infty$, the discrepancy divided by $R$
is at most $\varepsilon/\varrho(t_m)\to0$.
Since $M_{\nu_*}(R)\sim R$, this proves the mass asymptotic.

For $R\geq e$, the inequality
$\log(1+u)\geq u/2$ for $0\leq u\leq1$ gives
\[
 \log(2R)
 \geq\sum_{k=0}^{m-1}\log\frac{t_{k+1}}{t_k}
 \geq\frac{m\varepsilon}{2\varrho(R)}.
\]
Since $N_{\nu_\varepsilon}(R)=m+2$,
$\log(2R)\leq2\log R$, and $\varrho(R)\log R\geq1$,
the counting estimate follows.

The discrepancy estimate and
Lemma~\ref{s4-relative-discrepancy} give
$G_{\nu_\varepsilon}\leq b(2+\sqrt2\varepsilon)I
\leq(2+\sqrt2)bI$ for every system.
If $b/a\leq\kappa$ and
$\varepsilon\leq(2\sqrt2\kappa)^{-1}$, the relative estimate
in that lemma gives
$G_{\nu_\varepsilon}\geq G_*/2\geq S/(2\kappa)$.
\end{proof}

For $\varrho\equiv1$, the nodes are $t_k=(1+\varepsilon)^k/2$.
Choosing $\varepsilon=(2\sqrt2\kappa)^{-1}$ gives
$\nu_\varepsilon\in\mathcal U_w(\kappa)$ with
$N_{\nu_\varepsilon}(R)=O(\kappa\log R)$.
This bound suffices for the construction of a fully universal
measure below. Section~\ref{sec:fixed} gives the sharper
dependence on $\kappa$.

We combine countably many such grids with delays tending to infinity, which ensures local finiteness.
The following estimate shows that adding an observation at time zero preserves a positive lower frame bound after translation.

\begin{lemma}
\label{s4-delay}
Suppose $G_\nu\geq\alpha I$ for some $\alpha>0$.
Then, for $L\geq1$ and $\eta>0$,
\[
 B+\eta A^LG_\nu A^L
 \geq\min\left\{\frac a2,
       \eta\alpha\left(\frac a{2b}\right)^L\right\}I.
\]
\end{lemma}

\begin{proof}
For a unit vector $x$, set $u=\|Ax\|^2$.
The Stein identity and positivity of $B$ give
$\langle Bx,x\rangle\geq\max\{0,a-bu\}$.
Jensen's inequality for the spectral measure of $A^2$ gives
$\|A^Lx\|^2\geq u^L$.
Hence
\[
 \langle(B+\eta A^LG_\nu A^L)x,x\rangle
 \geq\max\{0,a-bu\}+\eta\alpha u^L.
\]
If $u\leq a/(2b)$, the first term is at least $a/2$.
Otherwise, the second term is at least
$\eta\alpha(a/(2b))^L$.
This proves the estimate.
\end{proof}

We now construct a fully universal design supported on $\mathbb N_0$ with the prescribed counting bound.

\begin{theorem}
\label{s4-universal-design}
Let $h:[1,\infty)\to(0,\infty)$ be nondecreasing with
$h(R)\to\infty$.
There is a measure $\mu\in\mathcal U_w$ supported on
$\mathbb N_0$ such that $N_\mu(R)\leq h(R)\log R$
for all sufficiently large $R$.
The measure also satisfies $\mu(\{0\})=1$,
$M_\mu(R)\sim R$, and
$\sum_{t\in\operatorname{supp}\mu\setminus\{0\}}1/t<\infty$.
\end{theorem}

\begin{proof}
Put $\widetilde h(R)=\min\{h(R),R^{1/4}\}$ and define
\[
 \varrho(R)=
 \begin{cases}
  1,&0<R<e,\\
  \max\{1,\sqrt{\widetilde h(R)}/4\},&R\geq e.
 \end{cases}
\]
Then $\varrho$ is nondecreasing and locally bounded,
$\varrho(e)=1$, and $\varrho(R)\to\infty$.
For all sufficiently large $R$, we have
$\varrho(R)=\sqrt{\widetilde h(R)}/4$.

For $n\geq1$, apply Lemma~\ref{s4-refining-grid} with
$\varepsilon_n=(2\sqrt2\,n)^{-1}$, and set
$\nu_n=\nu_{\varepsilon_n}$ and $c_n=2+4/\varepsilon_n$.
Then $M_{\nu_n}(R)\sim R$,
$M_{\nu_n}(R)\leq1+2R$, and
$N_{\nu_n}(R)\leq c_n\varrho(R)\log R$ for $R\geq e$.
For systems normalized by $I\leq S\leq nI$, we also have
$G_{\nu_n}\geq I/(2n)$.

In the construction of Theorem~\ref{s4-integer}, choose
$\eta_n=2n^2$ and an integer $M_n\geq1$ large enough that
\[
 \widehat\nu_n
 =\sum_{j=0}^{M_n-1}\delta_j
 +\eta_n\sum_{t\in\operatorname{supp}\nu_n}
       \nu_n(\{t\})\delta_{M_n+\lceil t\rceil}
\]
satisfies $G_{\widehat\nu_n}\geq I/4$ whenever
$I\leq S\leq nI$.
Set
$\widetilde\nu_n
=\eta_n^{-1}\widehat\nu_n+(1-\eta_n^{-1})\delta_0$.
Then $\widetilde\nu_n$ is supported on $\mathbb N_0$,
$\widetilde\nu_n(\{0\})=1$, and
$G_{\widetilde\nu_n}\geq aI/(4\eta_n)$ whenever $b/a\leq n$.

For $0\leq R<M_n$, we have
$M_{\widetilde\nu_n}(R)\leq1+R/\eta_n$.
For $R\geq M_n$,
\[
 M_{\widetilde\nu_n}(R)
 \leq1+\frac{M_n-1}{\eta_n}+M_{\nu_n}(R-M_n)
 \leq1+2R.
\]
Translation and rounding by at most one preserve the mass
asymptotic, and the finite sum at the initial integer times
does not affect it. Thus $M_{\widetilde\nu_n}(R)\sim R$.
Also, $N_{\widetilde\nu_n}(R)\leq M_n+N_{\nu_n}(R)$.
With $d_n=M_n+c_n$, this gives
$N_{\widetilde\nu_n}(R)\leq d_n\varrho(R)\log R$
for $R\geq e$.
Since $\widetilde\nu_n((0,R])\leq2R$,
Proposition~\ref{prop:mass} gives
$G_{\widetilde\nu_n}\leq(1+2\sqrt2)bI$
for every system.

Since $\widetilde h(R)\to\infty$, choose integers $L_n\geq e$
such that $L_n>L_{n-1}+1$ for $n\geq2$ and
$\sum_{j=1}^n d_j\leq\sqrt{\widetilde h(L_n)}$.
Define
\(
 \mu=\delta_0+\sum_{n\geq1}2^{-n}\tau_{L_n}\widetilde\nu_n.
\)
Every bounded interval meets only finitely many translated
supports, each containing finitely many atoms.
Thus $\mu$ has locally finite support in $\mathbb N_0$
and finite positive weights, with $\mu(\{0\})=1$.
The common operator bound ensures strong convergence of the
positive Gramian sums and gives
$G_\mu\leq(2+2\sqrt2)bI$.

Fix $\kappa\geq1$ and take $n=\lceil\kappa\rceil$.
For every system with $b/a\leq\kappa$,
Lemma~\ref{s4-delay} gives
\[
 G_\mu
 \geq a\min\left\{\frac12,
       \frac{2^{-n}}{4\eta_n}(2\kappa)^{-L_n}\right\}I.
\]
The measure depends only on $h$, and these bounds are uniform
over all systems with $b/a\leq\kappa$.
Hence $\mu\in\mathcal U_w$.

For each fixed $n$,
$M_{\tau_{L_n}\widetilde\nu_n}(R)/R\to1$.
For $R\geq1$, these quotients are bounded by three,
uniformly in $n$.
Dominated convergence in the definition of $\mu$ therefore gives
$M_\mu(R)/R\to1$.

For large $R$, let $m=\max\{n:L_n\leq R\}$.
If $R-L_n\geq e$, use the counting estimate for
$\widetilde\nu_n$.
Otherwise,
$N_{\widetilde\nu_n}(R-L_n)
\leq N_{\widetilde\nu_n}(e)\leq d_n$,
since $\varrho(e)=1$.
Thus, for all sufficiently large $R$,
\[
 N_\mu(R)
 \leq1+\varrho(R)\log R\sum_{n\leq m}d_n
 \leq1+\frac{\widetilde h(R)}4\log R
 \leq h(R)\log R.
\]
Since $\widetilde h(R)\leq R^{1/4}$, this also gives
$N_\mu(R)=O(R^{1/4}\log R)$.
Summing over the intervals $[2^k,2^{k+1})$ yields
\[
 \sum_{t\in\operatorname{supp}\mu\setminus\{0\}}\frac1t
 \leq\sum_{k\geq0}2^{-k}N_\mu(2^{k+1})<\infty.
\]
\end{proof}

The support and weights in Theorem~\ref{s4-universal-design}
are chosen together.
The obstruction in Appendix~\ref{app:prescribed-supports}
shows that a large sampling count alone does not guarantee
that a prescribed support admits a fully universal design.
When the support may be chosen, the growth condition on the
sampling count has the following equivalent formulations.

\begin{remark}
Let $F:[e,\infty)\to(0,\infty)$ be a function, not necessarily
monotone.
There exists $\mu\in\mathcal U_w$ satisfying
$N_\mu(R)\leq F(R)$ for all sufficiently large $R$
if and only if $F(R)/\log R\to\infty$.

Define
$\mathcal C(F)=\inf_{\mu\in\mathcal U_w}
\limsup_{R\to\infty}N_\mu(R)/F(R)$.
Then
\[
 \mathcal C(F)=
 \begin{cases}
  0,&F(R)/\log R\to\infty,\\
  +\infty,&\text{otherwise}.
 \end{cases}
\]
Restricting the infimum to measures supported on $\mathbb N_0$
gives the same values.
In the first case, zero is attained by such a measure.

For every $\mu\in\mathcal U_w$, there is a measure
$\nu\in\mathcal U_w$ supported on $\mathbb N_0$ such that
$N_\nu(R)=o(N_\mu(R))$.
\end{remark}

\begin{proof}
If $F(R)/\log R$ does not tend to infinity, there are $K>0$
and $R_j\to\infty$ with $F(R_j)\leq K\log R_j$.
Corollary~\ref{s4-superlog-necessary} then gives
$N_\mu(R_j)/F(R_j)\to\infty$ for every
$\mu\in\mathcal U_w$, proving necessity and
$\mathcal C(F)=+\infty$.

Otherwise, choose $R_0\geq e$ so that $F(R)/\log R\geq1$
for $R\geq R_0$, and put
$m(R)=\inf_{s\geq\max\{R,R_0\}}F(s)/\log s$ for $R\geq1$.
Then $m$ is finite, nondecreasing, at least one, and tends to
infinity. Theorem~\ref{s4-universal-design}, applied with
$h=\sqrt m$, gives a design supported on $\mathbb N_0$ with
$N_\mu(R)/F(R)\leq1/\sqrt{m(R)}\to0$ for all sufficiently
large $R$.
This proves sufficiency and attainment of zero in both classes.
Finally, apply the same construction with $F=N_\mu$ to obtain
$N_\nu=o(N_\mu)$ for each $\mu\in\mathcal U_w$.
\end{proof}

These comparisons concern sampling counts.
The new design is not required to retain the frame bounds
of the given design.

\subsection{Stability with respect to the condition number bound}
\label{s5-section}

We now fix a finite upper bound $\kappa\geq1$ on the original frame
condition number. We show that neither $C^*(\kappa)$ nor
$C^*_{\mathbb Z}(\kappa)$ is attained.

Every design in $\mathcal U_w(\kappa)$ also belongs to
$\mathcal U_w(\kappa')$ for some $\kappa'>\kappa$.
Together with a power transformation of the sampling times,
this allows us to reduce the counting coefficient of any design
for which it is finite.
The argument also provides the uniform estimates used in
Theorem~\ref{s6-strict-gap} under a prescribed upper bound on the
condition number after sampling.
Since the classes $\mathcal U_w(\kappa)$ decrease as $\kappa$
increases, $C^*(\kappa)$ is nondecreasing.

The case $\kappa=1$ requires a perturbation estimate for sampling
measures whose scalar multipliers are bounded above and bounded
away from zero. The following estimate applies to the entire
spectrum, whereas Lemma~\ref{lem:relative-perturbation} requires
the spectrum to lie near $1$.

\begin{lemma}
\label{s5-general-perturbation}
Suppose $0\leq A\leq I$, $E_A(\{1\})=0$, and $S=I+E$, where
$E=E^*$ and $\|E\|\leq\delta$ for some $0<\delta<1$.
Assume $B=S-ASA\geq0$.
Let $\mu$ be a sampling design satisfying
$m\leq\phi_\mu(r)\leq Q$ for $0\leq r<1$ and some $m,Q>0$.
Then the sampling series converges strongly to a bounded operator
$G_\mu$, and
\[
 \|G_\mu-\phi_\mu(A^2)\|
 \leq16Q\delta\left(\log_2\frac2\delta+2\right).
\]
\end{lemma}

\begin{proof}
Write $\mu=\sum_{t\in T}w_t\delta_t$ and set
$K_\mu(r)=\sum_{t\in T}w_tr^t$, with $0^0=1$.
For $i\geq0$, set $u_i=2^{-i}$,
$P_i=E_A([1-u_i,1-u_i/2))$, $H_i=P_iH$, $A_i=A|_{H_i}$,
and $r_i=1-u_i/2$.
The projections $P_i$ sum strongly to $I$, and
$\|I-A_i\|\leq u_i$ and $\|A_i\|\leq r_i<1$.
For a bounded operator $X$, write $X_{ij}=P_iX|_{H_j}$.
Set $\beta=\|S\|\leq2$ and $W=E-AEA$.
Then
\(
 W_{ij}=(I-A_i)E_{ij}+A_iE_{ij}(I-A_j), \|W_{ij}\|\leq\delta(u_i+u_j).
\)
Since $1-r_ir_j\geq(u_i+u_j)/4$, each block series converges
absolutely in norm, with
\[
 \left\|\sum_{t\in T}w_tA_i^tW_{ij}A_j^t\right\|
 \leq\|W_{ij}\|K_\mu(r_ir_j)\leq4Q\delta.
\]
The Stein identity gives $\|B_{ii}\|\leq2\beta u_i$, and
positivity of $B$ implies
$\|B_{ij}\|\leq2\beta\sqrt{u_iu_j}$.
Using $1-r_ir_j\geq\max\{u_i,u_j\}/2$, we obtain
\begin{equation}\label{s5-decaying-block-bound}
 \left\|\sum_{t\in F}w_tA_i^tB_{ij}A_j^t\right\|
 \leq4\beta Q\,2^{-|i-j|/2},
 \qquad F\subset T \text{ finite}.
\end{equation}
The block Schur test bounds the positive finite sums uniformly
in $F$, so they converge strongly to a bounded operator $G_\mu$.
Its blocks are the corresponding absolutely convergent series.

Assign any finite value to $\phi_\mu(1)$ and set
$G_0=\phi_\mu(A^2)$.
Since $E_A(\{1\})=0$, this choice does not affect $G_0$.
The spectral theorem gives $mI\leq G_0\leq QI$, and
monotone convergence yields
\(
 \sum_{t\in T}w_tA^t(I-A^2)A^t=G_0
\)
in the strong operator topology.
Since $G_0$ commutes with each $P_i$, the identity
$B=(I-A^2)+W$ and the preceding block estimates imply that
$\Delta=G_\mu-G_0$ satisfies
\[
 \|\Delta_{ii}\|\leq4Q\delta,\qquad
 \|\Delta_{ij}\|\leq4Q\min\{\delta,\beta2^{-|i-j|/2}\},
 \qquad i\ne j.
\]
Set $t_\delta=\log_2(2/\delta)$ and
$D_\delta=\lfloor2t_\delta\rfloor$.
Since $\beta\leq2$,
\[
 \sum_{d\geq1}\min\{\delta,\beta2^{-d/2}\}
 \leq D_\delta\delta+2\sum_{d>D_\delta}2^{-d/2}
 \leq2t_\delta\delta+\frac{\delta}{1-2^{-1/2}}.
\]
A second application of the block Schur test yields
\[
 \|\Delta\|
 \leq4Q\delta
       +8Q\sum_{d\geq1}\min\{\delta,\beta2^{-d/2}\}
 \leq Q\delta\left(16t_\delta+4+\frac8{1-2^{-1/2}}\right)
 \leq16Q\delta(t_\delta+2),
\]
where the last inequality uses $4+8/(1-2^{-1/2})<32$.
\end{proof}

The decay estimate \eqref{s5-decaying-block-bound} uses $S-ASA\geq0$. The lemma gives no boundedness conclusion for arbitrary selfadjoint perturbations without this assumption.

The following quantitative estimate will also be used when a
uniform bound on the sampled condition number is prescribed.

\begin{lemma}
\label{lem:gramian-interpolation}
Let $\mu\in\mathcal U_w(1)$ and
$Q\geq\sup_{0\leq z<1}\phi_\mu(z)$.
Suppose $0<\delta<1$ and $\ell>0$ satisfy $G_\mu\geq\ell I$
for every system with $(1-\delta)I\leq S\leq(1+\delta)I$.
For $\delta<\delta'<1$, set $\theta=\delta/\delta'$.
Every system with $(1-\delta')I\leq S\leq(1+\delta')I$
then satisfies $G_\mu\geq[\ell-(1-\theta)Q]I/\theta$.
\end{lemma}

\begin{proof}
Write $S=I+E$ and set $S_\theta=I+\theta E$.
Its observation quadratic form is
$B_\theta=\theta B+(1-\theta)(I-A^2)\geq0$.
The Stein identity and strong stability give integer Gramian
$S_\theta$, so its sampled Gramian is at least $\ell I$.
Propositions~\ref{prop:tight} and \ref{prop:mass} ensure that
$G_\mu$ is bounded. Linearity in the observation quadratic form
gives $\theta G_\mu+(1-\theta)\phi_\mu(A^2)\geq\ell I$.
Since $\phi_\mu(A^2)\leq QI$, the assertion follows.
\end{proof}

\begin{theorem}
\label{s5-budget-stability}
For every $1\leq\kappa<\infty$,
\begin{equation}\label{s5-class-union}
 \mathcal U_w(\kappa)
 =\bigcup_{\kappa'>\kappa}\mathcal U_w(\kappa').
\end{equation}
Consequently, $C^*$ is right continuous on $[1,\infty)$.
\end{theorem}

\begin{proof}
Only the inclusion from left to right in \eqref{s5-class-union}
requires proof. Fix $\mu\in\mathcal U_w(\kappa)$.
Since $\mu\in\mathcal U_w(1)$, Proposition~\ref{prop:tight} gives
$m,Q>0$ such that $m\leq\phi_\mu(r)\leq Q$ for $0\leq r<1$.
It also ensures that Proposition~\ref{prop:mass} applies, giving
an upper bound for $G_\mu$ depending only on $\mu$ and the
original upper frame bound.
It remains to obtain a uniform lower bound for some
$\kappa'>\kappa$.

Suppose first that $\kappa>1$ and set
$\delta=(\kappa-1)/(\kappa+1)$.
Since $\mu\in\mathcal U_w(\kappa)$, there is $\ell>0$ such that
$G_\mu\geq\ell I$ for every system with
$(1-\delta)I\leq S\leq(1+\delta)I$.
For $\delta<\delta'<1$, Lemma~\ref{lem:gramian-interpolation}
gives $G_\mu\geq[\ell-(1-\theta)Q]I/\theta$ on the larger
interval $(1-\delta')I\leq S\leq(1+\delta')I$, where
$\theta=\delta/\delta'$.
Choose $\delta'$ sufficiently close to $\delta$ that this lower
bound is positive, and put
$\kappa'=(1+\delta')/(1-\delta')>\kappa$.
If $aI\leq S\leq bI$ and $b/a\leq\kappa'$, then
$\|S/c-I\|\leq\delta'$ for $c=(a+b)/2$.
Rescaling the observation by $c^{-1/2}$ also divides $G_\mu$
by $c$. Thus the lower bound above and
Proposition~\ref{prop:mass} give
$\mu\in\mathcal U_w(\kappa')$.

Now suppose $\kappa=1$.
For $0<\delta<1$, Lemma~\ref{s5-general-perturbation} gives
\[
 G_\mu\geq
 \left[m-16Q\delta\left(\log_2\frac2\delta+2\right)\right]I
\]
for every system with $S=I+E$ and $\|E\|\leq\delta$.
Choose $\delta$ so small that
$16Q\delta(\log_2(2/\delta)+2)<m/2$.
The same rescaling argument, together with the upper bound,
gives $\mu\in\mathcal U_w(\kappa')$ for
$\kappa'=(1+\delta)/(1-\delta)>1$.
This proves \eqref{s5-class-union}.

Taking the infimum of $D(\mu)$ over \eqref{s5-class-union} gives
$C^*(\kappa)=\inf_{\kappa'>\kappa}C^*(\kappa')$.
Since $C^*$ is nondecreasing, it is right continuous.
\end{proof}

The larger bound $\kappa'>\kappa$ on the original frame condition number may depend on $\mu$, even when $\kappa=1$.
Thus the identity $\mathcal U_w(1)=\bigcup_{\kappa'>1}\mathcal U_w(\kappa')$ is consistent with the strict inclusion
$\mathcal U_w(\kappa')\subsetneq\mathcal U_w(1)$ for every fixed $\kappa'>1$.

\subsection{Power transformations of sampling times and nonattainment}
\label{s5-power-subsection}

For each $\mu\in\mathcal U_w(\kappa)$,
Theorem~\ref{s5-budget-stability} gives $\kappa'>\kappa$ such that
$\mu\in\mathcal U_w(\kappa')$.
We use this margin to replace $t$ by $t^\alpha$ for some $\alpha>1$,
dividing a finite logarithmic counting coefficient by $\alpha$.
To keep the resulting design in $\mathcal U_w(\kappa)$, we compare
the original system with one obtained by taking a fractional power
of its generator. The next two lemmas give estimates independent
of the spectrum and its multiplicities.

\begin{lemma}
\label{s5-fractional}
Let $L,T\geq0$ be bounded operators, with $E_L(\{0\})=0$,
$\|L\|\leq s$ for some $0<s\leq1$, and $B=LT+TL=C^*C$.
For $1/2\leq\beta\leq1$, set $M=L^\beta$.
The operator $CL^{(\beta-1)/2}$ has a bounded extension $C_\beta$.
There are absolute constants $c_0,c_1>0$ such that
\begin{equation}\label{s5-cbeta-bound}
 \|C_\beta\|^2\leq c_0\|T\|s^\beta.
\end{equation}
The integral
$T_\beta=\int_0^\infty e^{-tM}C_\beta^*C_\beta e^{-tM}\,dt$
and the series
$S_\beta=\sum_{n\geq0}e^{-nM}C_\beta^*C_\beta e^{-nM}$
converge strongly to bounded operators and satisfy
$\|T_\beta-T\|\leq c_1\|T\|(1-\beta)$ and
$\|S_\beta-T_\beta\|\leq2\|T_\beta\|s^\beta$.
\end{lemma}

\begin{proof}
Set $m_j=s2^{-j}$, $P_j=E_L((m_j/2,m_j])$, and
$L_j=L|_{P_jH}$ for $j\geq0$.
The projections $P_j$ sum strongly to $I$.
For a bounded operator $X$, write $X_{ij}=P_iX|_{P_jH}$.
Since $B=LT+TL$, we have $\|B_{jj}\|\leq2\|T\|m_j$.
Positivity of $B$ therefore gives
\begin{equation}\label{s5-b-blocks}
 \|B_{ij}\|\leq2\|T\|\sqrt{m_im_j}.
\end{equation}
Put $p=(\beta-1)/2$.
Then
$\|L_i^pB_{ij}L_j^p\|
 \leq2^{2-\beta}\|T\|m_i^{\beta/2}m_j^{\beta/2}$.
For vectors supported on finitely many bands,
\[
 \|CL^px\|^2
 \leq2^{2-\beta}\|T\|
       \left(\sum_jm_j^{\beta/2}\|P_jx\|\right)^2
 \leq2^{2-\beta}\|T\|
       \left(\sum_jm_j^\beta\right)\|x\|^2
 \leq c_0\|T\|s^\beta\|x\|^2.
\]
By spectral approximation, this estimate holds on
$\mathcal D(L^p)$.
Since this domain is dense, $CL^p$ has a unique bounded extension
$C_\beta$, satisfying \eqref{s5-cbeta-bound}.
Set $B_\beta=C_\beta^*C_\beta$.

Fix $i\leq j$ and put $d=j-i$, $m=m_i$,
$X_i=L_i/m$, and $X_j=L_j/m$.
Their spectra lie in $[1/2,1]$ and
$[2^{-d-1},2^{-d}]$, respectively.
The change of variables $u=tm^\beta$ gives the absolutely
convergent block integral
\[
 (T_\beta)_{ij}
 =m^{-1}\int_0^\infty
 X_i^{(\beta-1)/2}e^{-uX_i^\beta}B_{ij}
 X_j^{(\beta-1)/2}e^{-uX_j^\beta}\,du.
\]
The factor $m^{-1}$ is independent of $\beta$.
Using \eqref{s5-b-blocks}, $X_i^\beta\geq I/2$,
$\|X_i^{(\beta-1)/2}\|\leq2^{(1-\beta)/2}$, and
$\|X_j^{(\beta-1)/2}\|\leq2^{(d+1)(1-\beta)/2}$, we obtain
$\|(T_\beta)_{ij}\|\leq c\|T\|2^{-\beta d/2}$.
The same bound holds for truncated integrals.
Since $\beta\geq1/2$, the block Schur test bounds these positive
operators uniformly, so they converge strongly to a bounded
operator $T_\beta$.

Functional calculus gives
\[
 \frac{\partial}{\partial\beta}
 \left(X^{(\beta-1)/2}e^{-uX^\beta}\right)
 =X^{(\beta-1)/2}e^{-uX^\beta}
     \left(\frac12\log X-uX^\beta\log X\right).
\]
The last factor is bounded in norm by $c(1+u)$ for $X_i$
and by $c(d+1)(1+u)$ for $X_j$.
Together with the preceding estimates, these bounds justify
differentiation under the integral by dominated convergence and give
$\|\partial_\beta(T_\beta)_{ij}\|
 \leq c\|T\|(d+1)2^{-\beta d/2}$.
The reverse blocks satisfy the same bound by selfadjointness.
Integrating from $\beta$ to one and applying the block Schur test,
using $\sum_{d\in\mathbb Z}(|d|+1)2^{-|d|/4}<\infty$, yields
$\|T_\beta-T_1\|\leq c_1\|T\|(1-\beta)$.
Since
$\int_0^R e^{-tL}Be^{-tL}\,dt=T-e^{-RL}Te^{-RL}$
and $E_L(\{0\})=0$, strong stability gives $T_1=T$.

The defining integral, boundedness of $M$, and strong stability
also give $B_\beta=MT_\beta+T_\beta M$.
For $F_x(t)=\|C_\beta e^{-tM}x\|^2$, Cauchy--Schwarz gives
\[
 \int_0^\infty|F_x'(t)|\,dt
 \leq2\langle T_\beta x,x\rangle^{1/2}
        \langle T_\beta Mx,Mx\rangle^{1/2}
 \leq2\|T_\beta\|s^\beta\|x\|^2.
\]
For every integer $N\geq0$, comparison on each interval
$[n,n+1]$ gives
\[
 \left|\sum_{n=0}^N F_x(n)-\int_0^{N+1}F_x(t)\,dt\right|
 \leq\int_0^{N+1}|F_x'(t)|\,dt
 \leq2\|T_\beta\|s^\beta\|x\|^2.
\]
The positive partial sums are therefore uniformly bounded and
converge strongly to $S_\beta$.
Passing to the limit gives
$|\langle(S_\beta-T_\beta)x,x\rangle|
 \leq2\|T_\beta\|s^\beta\|x\|^2$,
which proves the asserted norm bound.
\end{proof}

\begin{lemma}
\label{s5-profile}
Let $M,T\geq0$ be bounded operators such that
$E_M(\{0\})=0$ and $MT+TM=C^*C$.
Let $\sigma$ be a positive atomic measure satisfying
$\sigma((0,R])\leq KR$ for every $R>0$.
For $1<\alpha\leq2$, set
$f_\alpha(z)=\sqrt\alpha\,z^{(\alpha-1)/2}e^{-z^\alpha}-e^{-z}$
for $z>0$.
There is an absolute constant $c_2>0$ such that
\[
 \sum_{t>0}\sigma(\{t\})\|Cf_\alpha(tM)x\|^2
 \leq c_2K\|T\|(\alpha-1)^2\|x\|^2,
 \qquad x\in H.
\]
\end{lemma}

\begin{proof}
We use the continuous extension $f_\alpha(0)=-1$.
Differentiating
$\sqrt\alpha\,z^{(\alpha-1)/2}e^{-z^\alpha}$ with respect to
$\alpha$ multiplies it by
$(2\alpha)^{-1}+(\log z)/2-z^\alpha\log z$.
For $0<z\leq1$, the absolute value of this derivative is at most
$c(1+|\log z|)$.
For $z\geq1$, the inequality $z^\alpha\geq z$ and exponential
decay give the bound $c(1+\log z)e^{-z/4}$.
Both bounds are uniform for $1\leq\alpha\leq2$.
Since $f_1(z)=0$ for $z>0$, integration in $\alpha$ gives
$|f_\alpha(z)|
 \leq c(\alpha-1)(1+|\log z|)e^{-z/4}$.

We may assume $q_0=\|M\|>0$.
Set $B=MT+TM$, $q_j=q_02^{-j}$, and
$P_j=E_M((q_j/2,q_j])$ for $j\geq0$.
Write $M_j=M|_{P_jH}$ and use block notation relative to
these projections.
As in \eqref{s5-b-blocks},
$\|B_{ij}\|\leq2\|T\|\sqrt{q_iq_j}$.
For $i\leq j$, put $d=j-i$ and $u=tq_i$.
The scalar estimate gives
\[
 \|f_\alpha(tM_i)\|\|f_\alpha(tM_j)\|
 \leq c(\alpha-1)^2(1+d)(1+|\log u|)^2e^{-u/8}.
\]
The function
$H(u)=\sup_{v\geq u}(1+|\log v|)^2e^{-v/8}$
is nonincreasing and integrable on $(0,\infty)$:
it is $O(1+|\log u|^2)$ near zero and decays exponentially
at infinity.
Applying the mass bound to the level sets of $H$ gives
\[
 \sum_{t>0}\sigma(\{t\})H(tq_i)
 \leq\frac K{q_i}\int_0^\infty H(u)\,du.
\]

For a finite set $F$ of positive atoms of $\sigma$, set
$E_F=\sum_{t\in F}\sigma(\{t\})
f_\alpha(tM)Bf_\alpha(tM)$.
These operators are positive, and the preceding estimates give
\[
 \|(E_F)_{ij}\|
 \leq cK\|T\|(\alpha-1)^2
       (1+|i-j|)2^{-|i-j|/2}.
\]
The block Schur test bounds $E_F$ uniformly in $F$.
Monotone convergence of the positive sums proves the stated bound.
\end{proof}

We now combine the two estimates to obtain a power transformation
of the sampling design, with a common exponent when the frame
and mass bounds are uniform.

\begin{theorem}
\label{s5-power-dilution}
Suppose $1\leq\kappa<\kappa'<\infty$ and
$\mu\in\mathcal U_w(\kappa')$.
There are $\alpha>1$ and $\rho\in\mathcal U_w(\kappa)$ such that
$D(\rho)=D(\mu)/\alpha$.
The exponent can be chosen uniformly for any collection of such
measures with common constants $0<\eta<1$ and $a_0,K>0$
satisfying $\mu((0,R])\leq KR$ for $R>0$ and $G_\mu\geq a_0I$
for every system with $(1-\eta)I\leq S\leq(\kappa+\eta)I$.
\end{theorem}

\begin{proof}
Proposition~\ref{prop:tight} gives $w_0=\mu(\{0\})>0$ and
a finite constant $K>0$ such that $\mu((0,R])\leq KR$ for $R>0$.
For $1<\alpha\leq2$, to be chosen, set
\[
 \nu_\alpha=w_0\delta_0+
 \sum_{\substack{t\in\operatorname{supp}\mu\\t>0}}
 \alpha t^{\alpha-1}w_t\delta_{t^\alpha}.
\]
Its support is locally finite, and
\(
 \nu_\alpha((0,R])
 \leq\alpha R^{(\alpha-1)/\alpha}\mu((0,R^{1/\alpha}])
 \leq\alpha KR.
\)
By Proposition~\ref{prop:mass}, every system with
$I\leq S\leq\kappa I$ satisfies
$0\leq G_{\nu_\alpha}\leq B_0\kappa I$, where
$B_0=w_0+\sqrt2\alpha K$.

Choose $0<\eta<1$ and $a_0>0$ such that $G_\mu\geq a_0I$
whenever $(1-\eta)I\leq S\leq(\kappa+\eta)I$.
Such constants exist because $\mu\in\mathcal U_w(\kappa')$;
one may first take $(\kappa+\eta)/(1-\eta)<\kappa'$.
For the uniform assertion, use the prescribed common constants.
Choose $1<\alpha\leq2$ sufficiently close to one that, with
$\beta=1/\alpha$,
\(
 c_1\kappa(1-\beta)\leq\eta/3,\
 \sqrt{2c_2K\kappa}\,(\alpha-1)\leq\sqrt{a_0}/4.
\)
Then choose $0<s\leq1$ such that
$\kappa s\leq\eta/3$, $4\kappa s^\beta\leq\eta/3$, and
$c_0w_0\kappa s^\beta\leq a_0/4$.

We first prove a lower bound when $A=e^{-L}$, with
$L\geq0$, $E_L(\{0\})=0$, and $\|L\|\leq s$.
The continuous Gramian is
\[
 T=\int_0^\infty e^{-tL}Be^{-tL}\,dt
  =\int_0^1e^{-uL}Se^{-uL}\,du.
\]
Thus $\|T\|\leq\kappa$, $B=LT+TL$, and
$\|T-S\|\leq\kappa s$, since
$\|e^{-uL}Se^{-uL}-S\|\leq2\kappa us$.
Set $M=L^\beta$.
Lemma~\ref{s5-fractional} gives
$\|T_\beta\|\leq\kappa+c_1\kappa(1-\beta)\leq2\kappa$ and
\(
 \|S_\beta-S\|
 \leq\kappa s+c_1\kappa(1-\beta)+2\|T_\beta\|s^\beta
 \leq\eta.
\)
Hence the system $(e^{-M},C_\beta)$ satisfies
$(1-\eta)I\leq S_\beta\leq(\kappa+\eta)I$, so
$\sum_t w_t\|C_\beta e^{-tM}x\|^2\geq a_0\|x\|^2$.
By \eqref{s5-cbeta-bound}, the contribution at time zero is
at most $a_0\|x\|^2/4$.
The samples at positive times therefore have energy at least
$3a_0\|x\|^2/4$.

The identity $C=C_\beta M^{(\alpha-1)/2}$ holds first on finite
spectral sums and then on the whole space by boundedness.
Consequently,
\[
 \sqrt{\alpha t^{\alpha-1}}\,Ce^{-t^\alpha L}
 -C_\beta e^{-tM}=C_\beta f_\alpha(tM),
 \qquad t>0.
\]
Apply Lemma~\ref{s5-profile} to $(M,T_\beta,C_\beta)$ with
$\sigma=\mu$, using $B_\beta=MT_\beta+T_\beta M$.
Our choice of $\alpha$ bounds the $\ell^2$ norm of these
differences, weighted by $\sqrt{w_t}$, by
$\sqrt{a_0}\|x\|/4$.
The triangle inequality gives a lower bound greater than
$\sqrt{a_0}\|x\|/2$ for the norm of the transformed sample
sequence at positive times.
Adding the nonnegative contribution at zero yields
$G_{\nu_\alpha}\geq\gamma I$, where $\gamma=a_0/4$,
uniformly for $\|L\|\leq s$.

For a general system with $I\leq S\leq\kappa I$, set
$P=E_A([e^{-s},1])$.
The compressed system on $PH$ has integer Gramian
$PSP|_{PH}$ between $I_{PH}$ and $\kappa I_{PH}$, and
$L=-\log(A|_{PH})$ satisfies the preceding assumptions.
Thus $PG_{\nu_\alpha}P\geq\gamma P$.
Lemma~\ref{lem:spectral-completion} and the bound
$G_{\nu_\alpha}\leq B_0\kappa I$ give $\rho\in\mathcal U_w(\kappa)$ satisfying
$D(\rho)=D(\nu_\alpha)=D(\mu)/\alpha$.
The last equality follows from
$N_{\nu_\alpha}(R)=N_\mu(R^{1/\alpha})$ and also holds
when $D(\mu)=\infty$.
The choice of $\alpha$ uses only $\kappa,\eta,a_0,K$,
which proves the uniform assertion.
\end{proof}

The preceding results, together with Theorem~\ref{s4-integer},
imply nonattainment for both real and integer sampling times.

\begin{theorem}
\label{s5-nonattainment}
For every $1\leq\kappa<\infty$, neither $C^*(\kappa)$ nor
$C^*_{\mathbb Z}(\kappa)$ is attained.
\end{theorem}

\begin{proof}
Fix $\mu\in\mathcal U_w(\kappa)$ with $D(\mu)<\infty$.
Since $\mu\in\mathcal U_w(1)$,
Theorem~\ref{thm:tight-endpoint} gives $D(\mu)>0$.
By Theorem~\ref{s5-budget-stability}, there is $\kappa'>\kappa$
such that $\mu\in\mathcal U_w(\kappa')$.
Theorem~\ref{s5-power-dilution} then gives $\alpha>1$ and
$\rho\in\mathcal U_w(\kappa)$ satisfying
$D(\rho)=D(\mu)/\alpha<D(\mu)$.
By Theorem~\ref{s4-integer}, there is
$\widehat\rho\in\mathcal U_w(\kappa)$ supported on $\mathbb N_0$
with $D(\widehat\rho)\leq D(\rho)<D(\mu)$.

Theorems~\ref{thm:relative} and~\ref{s4-integer} show that both
infima are finite. An attaining design would therefore have
finite counting coefficient, contradicting the strict improvement
above.
\end{proof}

The choice of $\alpha>1$ depends on the lower frame bound after
sampling and the mass bound of $\mu$.
Thus $\alpha-1$ may tend to zero along a minimizing sequence.
Nonattainment alone therefore does not imply that $C^*$ is strictly
increasing or that
$\mathcal U_w(\kappa_2)\subsetneq\mathcal U_w(\kappa_1)$
for every $1\leq\kappa_1<\kappa_2$.

For $1<\kappa<\Gamma<\infty$, a prescribed upper bound $\Gamma$
on the condition number after sampling allows a uniform reduction
of the counting coefficient after normalization and translation.
Theorem~\ref{s6-strict-gap} uses this uniformity to show that
prescribing $\Gamma$ strictly increases the optimal counting
coefficient for both real and integer sampling times.

\section{Sampling density and the sampled frame condition number}
\label{sec:output}

Casazza and Chen \cite{CasazzaChen2017} study condition number
optimization by rescaling the vectors of a given frame.
Here, each design in $\mathcal U_w(\kappa)$ gives a uniform bound on the
condition number after sampling for all positive systems whose
original frame condition number is at most $\kappa$.
This bound may depend on the design.
We now quantify the increase in the optimal counting coefficient
when the bound is prescribed in advance.

For a bounded positive invertible operator $G$, write
$\operatorname{cond}(G)=\|G\|\|G^{-1}\|$.
For $\kappa,\Gamma\geq1$, recall that
$\mathcal U_w(\kappa;\Gamma)$ consists of the designs
$\mu\in\mathcal U_w(\kappa)$ satisfying
$\operatorname{cond}(G_\mu)\leq\Gamma$
for every positive system with $b/a\leq\kappa$.
Let $C^*(\kappa;\Gamma)$ and $C^*_{\mathbb Z}(\kappa;\Gamma)$
denote the infima of $D(\mu)$ over these designs with real and
integer sampling times, respectively.
The infimum over an empty class is $+\infty$.

\subsection{Feasibility and joint asymptotics}
\label{s6-feasibility-subsection}

The smallest uniform upper bound on the condition number after
sampling equals the prescribed bound $\kappa$ on the original
frame condition number.
At this threshold, every nonnegative integer must be sampled.
For any larger bound, designs with $N_\mu(R)=O(\log R)$ exist.

We first record the scalar estimate used throughout this section.

\begin{lemma}
\label{s6:scalar}
Let $\mu\in\mathcal U_w(\kappa;\Gamma)$ and set
$m_0=\inf_{0\leq z<1}\phi_\mu(z)$ and
$M_0=\sup_{0\leq z<1}\phi_\mu(z)$.
Then $0<m_0\leq M_0<\infty$ and
$M_0/m_0\leq\Gamma/\kappa$.
With $c=(M_0+m_0)/2$ and
$\varepsilon=(\Gamma-\kappa)/(\Gamma+\kappa)$, we have
$|\phi_\mu(z)/c-1|\leq\varepsilon$ for $0\leq z<1$.
\end{lemma}

\begin{proof}
Since $\mu\in\mathcal U_w(1)$, Proposition~\ref{prop:tight} gives
$0<m_0\leq M_0<\infty$.
For $r,s\in[0,1)$, take $A=\operatorname{diag}(r,s)$,
$S=\operatorname{diag}(\kappa,1)$, and $B=S-ASA\geq0$.
Then $G_\mu=\operatorname{diag}(\kappa\phi_\mu(r^2),
\phi_\mu(s^2))$, so $\operatorname{cond}(G_\mu)\leq\Gamma$
implies $\kappa\phi_\mu(r^2)/\phi_\mu(s^2)\leq\Gamma$.
Taking the supremum over $r,s\in[0,1)$ proves the ratio bound.
Finally,
$|\phi_\mu(z)/c-1|
 \leq(M_0-m_0)/(M_0+m_0)\leq\varepsilon$.
\end{proof}

\begin{theorem}
Fix $\kappa\geq1$.
\begin{enumerate}
\item If $\Gamma<\kappa$, then
$\mathcal U_w(\kappa;\Gamma)=\varnothing$.
\item If $\Gamma=\kappa$, then
\(
 \mathcal U_w(\kappa;\kappa)
 =\left\{c\sum_{n\geq0}\delta_n:c>0\right\}.
\)
In particular, $C^*(\kappa;\kappa)=+\infty$.
\item If $\Gamma>\kappa$, there exists
$\mu\in\mathcal U_w(\kappa;\Gamma)$ with
$N_\mu(R)=O(\log R)$.
\end{enumerate}
\end{theorem}

\begin{proof}
If $\mu\in\mathcal U_w(\kappa;\Gamma)$, Lemma~\ref{s6:scalar}
gives $1\leq M_0/m_0\leq\Gamma/\kappa$, so $\Gamma\geq\kappa$.
When $\Gamma=\kappa$, the same lemma implies
$\phi_\mu\equiv c=\phi_\mu(0)=\mu(\{0\})>0$.
Hence
\[
 \int_{[0,\infty)}e^{-ut}\,d\mu(t)
 =\frac{c}{1-e^{-u}}
 =c\sum_{n\geq0}e^{-un},
 \qquad u>0.
\]
The finite measures $e^{-t}\,d\mu(t)$ and
$c\sum_{n\geq0}e^{-n}\delta_n$ have the same Laplace transform
and therefore coincide by uniqueness
\cite[Proposition~1.2]{SchillingSongVondracek2010}.
Thus $\mu=c\sum_{n\geq0}\delta_n$.
The converse follows from $G_\mu=cS$.
Since each such measure has $D(\mu)=+\infty$,
we also obtain $C^*(\kappa;\kappa)=+\infty$.

Finally, suppose $\Gamma>\kappa$ and set
$\varepsilon=(\Gamma-\kappa)/(\Gamma+\kappa)\in(0,1)$.
Theorem~\ref{thm:relative} gives a design with
$N_\mu(R)=O(\log R)$ and
$\operatorname{cond}(G_\mu)
 \leq\kappa(1+\varepsilon)/(1-\varepsilon)=\Gamma$.
\end{proof}

An upper bound on the condition number after sampling ensures that the same design remains admissible for a larger bound on the original frame condition number.
We establish this inclusion and a scalar estimate before deriving the joint asymptotic formula.

\begin{theorem}
\label{s6:enlargement}
Let $\kappa>1$, $\Gamma>\kappa$, and
$K_{\mathrm{eff}}=\kappa(\Gamma-1)/(\Gamma-\kappa)$.
Then $\mathcal U_w(\kappa;\Gamma)\subseteq\mathcal U_w(K)$
for $1\leq K<K_{\mathrm{eff}}$. Consequently,
$C^*(\kappa;\Gamma)\geq
\sup_{1\leq K<K_{\mathrm{eff}}}C^*(K)$.
\end{theorem}

\begin{proof}
For $K\leq\kappa$, the inclusion follows from monotonicity.
Fix $\kappa<K<K_{\mathrm{eff}}$,
$\mu\in\mathcal U_w(\kappa;\Gamma)$, and a system with
$I\leq S\leq KI$.
Set $\alpha=(K-\kappa)/(K-1)$,
$\beta=(\kappa-1)/(K-1)$, and $S'=\alpha I+\beta S$.
Then $I\leq S'\leq\kappa I$ and
$B'=S'-AS'A=\alpha(I-A^2)+\beta B\geq0$.
Strong stability and the Stein identity show that $S'$ is the
integer Gramian of $(A,(B')^{1/2})$.

Let $G'$ be its sampled Gramian.
Since $\mu\in\mathcal U_w(\kappa)$, $G'$ has an upper bound
independent of the system.
The inequality $B'\geq\beta B$ bounds the positive sampling sums
defining $G_\mu$ by $G'/\beta$.
Thus these sums converge strongly, and $G_\mu$ has a uniform
upper bound.
Set $M_0=\sup_{0\leq z<1}\phi_\mu(z)$, which is positive and
finite by Lemma~\ref{s6:scalar}.
With $G_0=\phi_\mu(A^2)\leq M_0I$, we have
$G'=\alpha G_0+\beta G_\mu$.

Take the direct sum with the scalar system
$A_r=r$, $S_r=\kappa$, and $B_r=\kappa(1-r^2)$,
where $0\leq r<1$.
Its integer Gramian still lies between $I$ and $\kappa I$,
so its sampled Gramian
$G'\oplus\kappa\phi_\mu(r^2)$ has condition number at most
$\Gamma$.
Hence $G'\geq\kappa\phi_\mu(r^2)I/\Gamma$.
Taking the supremum over $r$ gives
$\beta G_\mu=G'-\alpha G_0
\geq M_0(\kappa/\Gamma-\alpha)I$.
Substituting $\alpha$ and $\beta$, we obtain
\[
 G_\mu\geq M_0
 \frac{\kappa(\Gamma-1)-K(\Gamma-\kappa)}
      {\Gamma(\kappa-1)}I.
\]
The coefficient is positive because $K<K_{\mathrm{eff}}$.
Together with the uniform upper bound, this proves
$\mu\in\mathcal U_w(K)$ after rescaling.
Taking infima of $D(\mu)$ gives the final inequality.
\end{proof}

The interpolation argument in Theorem~\ref{s6:enlargement}
requires $\kappa>1$. The following estimate also applies at $\kappa=1$ and shows that small relative oscillation of the scalar sampling multiplier requires a large counting coefficient.

\begin{lemma}
\label{s6:small-oscillation}
Suppose $\mu\in\mathcal U_w(1)$ and
$|\phi_\mu(z)/c-1|\leq\varepsilon$ for $0\leq z<1$,
where $c>0$ and $0<\varepsilon<e^{-\pi^2}/2$.
Then
\[
 \liminf_{R\to\infty}\frac{N_\mu(R)}{\log R}
 \geq\frac{\log(1/(2\varepsilon))}{\pi^2}.
\]
\end{lemma}

\begin{proof}
Set $h_*=\pi^2/\log(1/(2\varepsilon))<1$ and fix
$h_*<h\leq1$.
For the system in Proposition~\ref{prop:cayley}, the Gramian
satisfies $a_hI\leq S_h\leq b_hI$, where
$a_h=(2\pi/h)e^{-\pi^2/h}$ and $b_h=4\pi/h$.
Its entries are nonnegative, and its row and column sums are
at most $\sum_{n\in\mathbb Z}\operatorname{sech}(nh/2)\leq b_h$.

Proposition~\ref{prop:mass} ensures that $G_\mu$ is bounded.
Its matrix entries satisfy
\[
 (G_\mu)_{ij}=\phi_\mu(\lambda_i\lambda_j)(S_h)_{ij},\qquad
 |(G_\mu-cS_h)_{ij}|\leq c\varepsilon(S_h)_{ij}.
\]
The Schur test gives $\|G_\mu-cS_h\|\leq c\varepsilon b_h$.
Since $h>h_*$, we have
$\varepsilon<e^{-\pi^2/h}/2$ and hence
$G_\mu\geq c(a_h-\varepsilon b_h)I>0$.
Theorem~\ref{thm:local-density} therefore yields
$\liminf_{R\to\infty}N_\mu(R)/\log R\geq1/h$.
Letting $h\downarrow h_*$ proves the claim.
\end{proof}

We now combine the two estimates for the lower bound.
For the upper bound, we reserve part of the allowed approximation error for rounding the sampling times to integers. The parameter $X$ below combines the bound $\kappa$ on the original frame condition number with the relative approximation error allowed by the condition number bound $\Gamma$ after sampling.

\begin{theorem}
\label{s6:joint}
Let $\kappa\geq1$ and $\Gamma>\kappa$, and set
$\varepsilon=(\Gamma-\kappa)/(\Gamma+\kappa)$ and
$X=\kappa/\varepsilon$.
If $X>6e^{\pi^2}$, then
$C^*(\kappa;\Gamma)\geq(\log X-\log6)/\pi^2$.
Moreover,
\[
 C^*_{\mathbb Z}(\kappa;\Gamma)
 \leq\inf_{0<d<\pi/2}
 \frac{\log(1+16K_0X/\cos d)}{2\pi d},
\]
where $K_0=4(3+2\sqrt2)$ is the constant in
Lemma~\ref{lem:complex-energy}.
Consequently, as $X\to\infty$,
\[
 \frac{\log X-\log6}{\pi^2}
 \leq C^*(\kappa;\Gamma)
 \leq C^*_{\mathbb Z}(\kappa;\Gamma)
 \leq\frac{\log X+\log\log X+O(1)}{\pi^2},
\]
with an absolute constant in the $O(1)$ term.
\end{theorem}

\begin{proof}
For the lower bound, assume $X>6e^{\pi^2}$ and set
$K_{\mathrm{eff}}=\kappa(\Gamma-1)/(\Gamma-\kappa)$ and
$J=\max\{K_{\mathrm{eff}},1/\varepsilon\}$.
If $1\leq\kappa\leq3$, then $1/\varepsilon=X/\kappa\geq X/3$.
If $\kappa\geq3$, then
$K_{\mathrm{eff}}/X=(\Gamma-1)/(\Gamma+\kappa)
>(\kappa-1)/(2\kappa)\geq1/3$.
Thus $J\geq X/3>2e^{\pi^2}$.

If $J=1/\varepsilon$, Lemmas~\ref{s6:scalar} and
\ref{s6:small-oscillation} give
$C^*(\kappa;\Gamma)\geq\pi^{-2}\log(J/2)$.
Otherwise, $J=K_{\mathrm{eff}}$ and $\kappa>1$.
Theorems~\ref{s6:enlargement} and \ref{thm:large-lower} give
$C^*(\kappa;\Gamma)\geq C^*(K)\geq\pi^{-2}\log(K/2)$
for $2\leq K<K_{\mathrm{eff}}$.
Letting $K\uparrow K_{\mathrm{eff}}$ gives the same bound.
Since $J\geq X/3$, the lower bound follows.

For the upper bound, fix $0<d<\pi/2$ and set
$\sigma=\varepsilon/(8\kappa)$ and
$h=2\pi d/\log(1+16K_0X/\cos d)$.
For any integer $M\geq8X$, Theorem~\ref{thm:relative}
with error $\varepsilon/2$ gives the measure
\(
 \mu_M=\sum_{n=0}^{M-1}\delta_n
       +\sum_{j\geq0}h\sigma e^{jh}\delta_{M+\sigma e^{jh}},
\)
which satisfies
$(1-\varepsilon/2)S\leq G_{\mu_M}\leq(1+\varepsilon/2)S$.
By rescaling, it suffices to consider $I\leq S\leq\kappa I$.
The tail is supported in $[M,\infty)$ and has cumulative mass
at most $K_hs$ up to time $s$, where $K_h=he^h/(e^h-1)$.
Rounding its sampling times upward gives
\[
 \widehat\mu_M=\sum_{n=0}^{M-1}\delta_n
       +\sum_{j\geq0}h\sigma e^{jh}
          \delta_{M+\lceil\sigma e^{jh}\rceil}.
\]
Let $T_M$ and $\widehat T_M$ be the corresponding analysis
operators, keeping the same sample indices before and after
rounding.
Lemma~\ref{s4-rounding} gives
$\|\widehat T_M-T_M\|^2\leq\kappa K_h/M$.
Choose $M$ large enough that also
$M\geq64\kappa K_h/\varepsilon^2$.
Then $\|\widehat T_M-T_M\|\leq\varepsilon/8$.
Writing $s_x=\langle Sx,x\rangle^{1/2}\geq\|x\|$, we obtain
\[
 (\sqrt{1-\varepsilon/2}-\varepsilon/8)s_x
 \leq\|\widehat T_Mx\|
 \leq(\sqrt{1+\varepsilon/2}+\varepsilon/8)s_x.
\]
Since $0<\varepsilon<1$,
\(
 \sqrt{1-\varepsilon/2}-\sqrt{1-\varepsilon}
 \geq\varepsilon/4,\
 \sqrt{1+\varepsilon}-\sqrt{1+\varepsilon/2}
 \geq\varepsilon/(4\sqrt2)>\varepsilon/8.
 \)
Hence
$(1-\varepsilon)S\leq G_{\widehat\mu_M}\leq(1+\varepsilon)S$,
so $\operatorname{cond}(G_{\widehat\mu_M})\leq\Gamma$.
Rounding and combining weights at coincident times do not
increase the counting coefficient.
Thus $D(\widehat\mu_M)\leq1/h$.
Taking the infimum over $d$ proves the upper bound.

Finally, choose $d=\pi/2-1/\log X$ for large $X$.
Then
$\log(1+16K_0X/\cos d)=\log X+\log\log X+O(1)$
and $2\pi d=\pi^2-2\pi/\log X$.
Substitution gives the stated asymptotic bounds.
\end{proof}

Theorem~\ref{s6:joint} applies along any parameter family with
$X\to\infty$. In particular, for fixed $\kappa\geq1$,
\[
 C^*(\kappa;\Gamma)\sim C^*_{\mathbb Z}(\kappa;\Gamma)
 \sim\pi^{-2}\log\frac{\kappa}{\Gamma-\kappa},
 \qquad \Gamma\downarrow\kappa.
\]
For any fixed $D>1$,
\[
 C^*(\kappa;D\kappa)\sim C^*_{\mathbb Z}(\kappa;D\kappa)
 \sim\pi^{-2}\log\kappa,
 \qquad \kappa\to\infty.
\]
These asymptotic relations do not imply equality of the real and
integer infima for every fixed pair $(\kappa,\Gamma)$.

\subsection{A strict gap between optimal counting coefficients}
\label{s6-strict-section}

Theorem~\ref{s6:joint} gives asymptotic bounds as $X\to\infty$.
We now compare the optimal counting coefficients with and without
a prescribed bound on the condition number after sampling, for
fixed $\kappa\geq1$ and $\Gamma>\kappa$.
We first consider the limit $\Gamma\to\infty$ with $\kappa$ fixed.

For each $\mu\in\mathcal U_w(\kappa)$, there are constants
$0<\alpha\leq\beta<\infty$ such that
$\alpha I\leq G_\mu\leq\beta I$ for every system with
$I\leq S\leq\kappa I$.
Since rescaling preserves condition numbers,
$\mu\in\mathcal U_w(\kappa;\Gamma)$ for every
$\Gamma>\max\{\kappa,\beta/\alpha\}$.
Thus
$\mathcal U_w(\kappa)
=\bigcup_{\kappa<\Gamma<\infty}\mathcal U_w(\kappa;\Gamma)$.

Taking infima of $D(\mu)$ over this increasing union gives
\[
 \lim_{\Gamma\to\infty}C^*(\kappa;\Gamma)=C^*(\kappa),
 \qquad
 \lim_{\Gamma\to\infty}C^*_{\mathbb Z}(\kappa;\Gamma)
 =C^*_{\mathbb Z}(\kappa).
\]

For $1<\kappa<\Gamma<\infty$, we next prove
$C^*(\kappa;\Gamma)>C^*(\kappa)$ and the corresponding inequality
for integer sampling times.
In Theorem~\ref{s5-power-dilution}, the exponent in
$t\mapsto t^\alpha$ may depend on the design.
Here we rescale the weights so that
$\sup_{0\leq z<1}\phi_\mu(z)=1$, shift each sampling time
from $t$ to $t+1$, and add an atom of weight one at zero.
These operations allow us to choose a single exponent
$p=p(\kappa,\Gamma)>1$ for all designs in
$\mathcal U_w(\kappa;\Gamma)$.

\begin{theorem}
\label{s6-strict-gap}
For every $1<\kappa<\Gamma<\infty$, there is
$p=p(\kappa,\Gamma)>1$ such that
\(
 C^*(\kappa;\Gamma)
 \geq pC^*(\kappa)>C^*(\kappa)\) and
 \(C^*_{\mathbb Z}(\kappa;\Gamma)
 \geq pC^*_{\mathbb Z}(\kappa)>C^*_{\mathbb Z}(\kappa).
\)
\end{theorem}

\begin{proof}
Fix $\mu\in\mathcal U_w(\kappa;\Gamma)$ and rescale its weights
so that $\sup_{0\leq z<1}\phi_\mu(z)=1$.
Set $\rho=\delta_0+\tau_1\mu$, where $\tau_1$ shifts each sampling
time from $t$ to $t+1$.
Then $D(\rho)=D(\mu)$, $\rho(\{0\})=1$, and
$\phi_\rho(z)=1-z+z\phi_\mu(z)\leq1$ for $0\leq z<1$.
For $r>0$,
\[
 M_\mu(r)
 \leq e\int_{[0,\infty)}e^{-t/r}\,d\mu(t)
 \leq\frac{e}{1-e^{-1/r}}
 \leq e(r+1).
\]
Thus $\rho((0,R])=M_\mu(R-1)\leq eR$ for $R>1$.
For $0<R<1$, this mass is zero, while
$\rho((0,1])=\mu(\{0\})=\phi_\mu(0)\leq1$.
Consequently,
$K_\rho:=\sup_{R>0}\rho((0,R])/R\leq2e$.
Proposition~\ref{prop:mass} gives
$G_\rho\leq b(1+2\sqrt2e)I$ for every system with
$aI\leq S\leq bI$.

For a system with $I\leq S\leq\kappa I$, take its direct sum
with the scalar system $A_r=r$, $B_r=\kappa(1-r^2)$, whose
integer Gramian is $\kappa$, where $0\leq r<1$.
The bound on the condition number of the sampled direct sum gives
$G_\mu\geq\kappa\phi_\mu(r^2)I/\Gamma$.
Taking the supremum over $r$ yields $G_\mu\geq\kappa I/\Gamma$.
Since $B\geq0$, $ASA\leq\kappa A^2$, and $\Gamma>1$,
\[
 G_\rho=B+AG_\mu A
 \geq B+\frac{\kappa}{\Gamma}A^2
 \geq\frac1\Gamma(B+ASA)
 =\frac S\Gamma.
\]
Rescaling gives $G_\rho\geq S/\Gamma$ whenever
$\operatorname{cond}(S)\leq\kappa$.
In particular, $\rho\in\mathcal U_w(\kappa)$.

Set $d=(\kappa-1)/(\kappa+1)$, $q=(1-d)/\Gamma$, and
$d'=d+dq/4$.
Then $0<d<d'<1$.
For every system with $S=I+E$ and $\|E\|\leq d$, the preceding
estimate gives $G_\rho\geq qI$.
Lemma~\ref{lem:gramian-interpolation}, with $\delta=d$,
$\delta'=d'$, $\ell=q$, and $Q=1$, gives
$G_\rho\geq qI/2$ whenever $\|S-I\|\leq d'$:
indeed, $\theta=d/d'$ satisfies
$1-\theta=q/(4+q)\leq q/4$.
Together with the uniform upper bound, this shows
$\rho\in\mathcal U_w(\kappa')$ for
$\kappa'=(1+d')/(1-d')>\kappa$.

Set
$\varepsilon_0=(\kappa+1)(d'-d)/4=d/(8\Gamma)\in(0,1)$.
If $(1-\varepsilon_0)I\leq S\leq(\kappa+\varepsilon_0)I$,
then
$\|2S/(\kappa+1)-I\|\leq(d+d')/2<d'$.
Rescaling the preceding lower bound therefore gives
$G_\rho\geq a_0I$, where
$a_0=(\kappa+1)q/4=1/(2\Gamma)$.
Both $\varepsilon_0$ and $a_0$ depend only on $\kappa,\Gamma$.

Apply the uniform assertion of Theorem~\ref{s5-power-dilution}
to these measures $\rho$ with $\eta=\varepsilon_0$,
$a_0=1/(2\Gamma)$, and $K=2e$.
These constants are independent of $\mu$, so a single
$p=p(\kappa,\Gamma)>1$ gives a design
$\sigma\in\mathcal U_w(\kappa)$ for each $\mu$, with
$D(\sigma)=D(\rho)/p=D(\mu)/p$.

For finite $D(\mu)$, the definition of $C^*(\kappa)$ gives
$D(\mu)\geq pC^*(\kappa)$; this inequality also holds when
$D(\mu)=\infty$.
Taking the infimum proves the inequality for real sampling times.
Corollary~\ref{cor:positive-lower} gives $C^*(\kappa)>0$,
so the inequality is strict.
Finally,
$C^*_{\mathbb Z}(\kappa;\Gamma)\geq C^*(\kappa;\Gamma)$ and
$C^*_{\mathbb Z}(\kappa)=C^*(\kappa)$ by
Theorem~\ref{s4-integer}.
Thus the same $p$ gives the inequality for integer sampling times.
\end{proof}

\subsection{Tight frames and local density bounds}
\label{s7-exact}
\label{s7-local-subsection}

For tight frames, we first determine $C^*(1;\Gamma)$ for each
prescribed bound $\Gamma>1$ on the sampled frame condition number
and construct a design supported on $\mathbb N_0$ that attains
this value. We then obtain uniform local density bounds and
extend them to systems with a prescribed bound on the original
frame condition number.

If the original frame is tight, then $S=aI$ for some $a>0$,
and the Stein identity gives $B=a(I-A^2)$.
Hence $G_\mu=a\phi_\mu(A^2)$, where
$\phi_\mu(z)=(1-z)\int_{[0,\infty)}z^t\,d\mu(t)$
for $0\leq z<1$, with $0^0=1$.
Thus, for $\Gamma\geq1$, a sampling design $\mu$ belongs to
$\mathcal U_w(1;\Gamma)$ if and only if
\begin{equation}\label{s7-scalar}
 0<\inf_{0\leq z<1}\phi_\mu(z)
 \leq\sup_{0\leq z<1}\phi_\mu(z)<\infty,
 \qquad
 \frac{\sup_{0\leq z<1}\phi_\mu(z)}
      {\inf_{0\leq z<1}\phi_\mu(z)}\leq\Gamma.
\end{equation}
Sufficiency follows from Proposition~\ref{prop:tight} and the
spectral theorem; necessity is Lemma~\ref{s6:scalar} with $\kappa=1$.

To determine $C^*(1;\Gamma)$ and $C^*_{\mathbb Z}(1;\Gamma)$,
define
\begin{equation}\label{eq:tight-periodic-profile}
 k(s)=\exp(-s-e^{-s}),\qquad
 P_h(s)=h\sum_{n\in\mathbb Z}k(s-nh),
 \qquad s\in\mathbb R,\quad h>0.
\end{equation}
The function $k$ is positive and smooth, with
$\int_{\mathbb R}k(s)\,ds=1$.
For each $h>0$, the series defining $P_h$ and all its termwise
derivatives converge uniformly on compact subsets of $\mathbb R$.
Thus $P_h$ is positive and smooth. It has period $h$ and satisfies
$h^{-1}\int_0^hP_h(s)\,ds=1$.

Set $m_h=\min_{s\in\mathbb R}P_h(s)$,
$M_h=\max_{s\in\mathbb R}P_h(s)$, and
$\mathcal R(h)=M_h/m_h$.
Lemma~\ref{s7-monotonicity} in Appendix~\ref{app:periodic-tools}
shows that $\mathcal R$ is a continuous strictly increasing
bijection from $(0,\infty)$ onto $(1,\infty)$.

\begin{theorem}\label{s7-main}
For $\Gamma>1$, let $h_\Gamma=\mathcal R^{-1}(\Gamma)$. Then
\(
 C^*(1;\Gamma)=C^*_{\mathbb Z}(1;\Gamma)=1/{h_\Gamma}.
\)
Both infima are attained by the same measure supported on $\mathbb N_0$.
\end{theorem}

\begin{proof}
Put $a=1/\Gamma$. For the lower bound, let
$\mu\in\mathcal U_w(1;\Gamma)$ with $D=D(\mu)<\infty$.
After multiplying all weights by a positive constant, we may assume
$a\leq\phi_\mu(z)\leq1$ for $0\leq z<1$.
For sufficiently small $0<\epsilon<a$,
Lemma~\ref{s7-periodization} gives periodic atomic measures with
periods $L\to\infty$ and $n_L$ distinct nodes per period.
Together with Lemma~\ref{s7-periodic-optimum}, this yields
\[
 \mathcal R(L/n_L)\leq\frac{1+\epsilon}{a-\epsilon},
 \qquad
 \limsup_{L\to\infty}\frac{n_L}{L}\leq D+\epsilon.
\]
Since $\mathcal R$ is strictly increasing,
$n_L/L\geq1/h_{(1+\epsilon)/(a-\epsilon)}$.
Hence $D+\epsilon\geq1/h_{(1+\epsilon)/(a-\epsilon)}$.
Letting $\epsilon\downarrow0$ and using the continuity of
$\mathcal R^{-1}$ gives $D\geq1/h_\Gamma$.

We now construct a measure supported on $\mathbb N_0$ that attains
this bound. Set $h=h_\Gamma$, so that $m_h/M_h=a$, and define
$\lambda_h=(h/M_h)\sum_{n\in\mathbb Z}e^{nh}\delta_{e^{nh}}$.
This auxiliary measure has finite mass on bounded intervals,
although its support accumulates at zero. For $u>0$,
\[
 a\leq u\int_{[0,\infty)}e^{-ut}\,d\lambda_h(t)
 =\frac{P_h(-\log u)}{M_h}\leq1.
\]
Replacing each time $t$ by $\lceil t\rceil$ gives
\[
 \sigma_h=\frac{h}{M_h(1-e^{-h})}\delta_1
 +\frac{h}{M_h}\sum_{n=1}^\infty
 e^{nh}\delta_{\lceil e^{nh}\rceil},
\]
where weights at coincident nodes are added.
Its support is locally finite. For $z=e^{-u}\in(0,1)$,
the inequalities $t\leq\lceil t\rceil<t+1$ and
$z\leq(1-z)/(-\log z)\leq1$, together with the preceding
Laplace transform bounds, give
$az^2\leq\phi_{\sigma_h}(z)\leq1$.

Set $J=\lceil2/(\Gamma-1)\rceil\geq1$ and
\(
 \rho_\Gamma=\sum_{j=0}^{J-1}\delta_j+\tau_J\sigma_h,
\)
where $\tau_J$ shifts each sampling time $t$ to $t+J$.
Then
$\phi_{\rho_\Gamma}(z)=1-z^J+z^J\phi_{\sigma_h}(z)\leq1$.
For the lower bound, use $\phi_{\sigma_h}(z)\geq az^2$ and
$J\geq2a/(1-a)$ to obtain
\[
 \phi_{\rho_\Gamma}(z)-a
 \geq (1-a)(1-z^J)-az^J(1-z^2)
 \geq \bigl(J(1-a)-2a\bigr)z^J(1-z)\geq0.
\]
At $z=0$, the multiplier equals one. Thus
$a\leq\phi_{\rho_\Gamma}(z)\leq1$ for $0\leq z<1$,
and \eqref{s7-scalar} gives
$\rho_\Gamma\in\mathcal U_w(1;\Gamma)$.

The sequence $\lceil e^{nh}\rceil$ is eventually strictly increasing.
Its counting function is $(\log R)/h+O_\Gamma(1)$, and translation
and a finite prefix change this estimate by only a bounded term.
Therefore
$N_{\rho_\Gamma}(R)=(\log R)/h_\Gamma+O_\Gamma(1)$.
Together with the lower bound, this proves that $\rho_\Gamma$
attains both infima and that their common value is $1/h_\Gamma$.
By Lemma~\ref{s7-monotonicity}, this value is continuous and
strictly decreasing in $\Gamma>1$.
\end{proof}

For each $\Gamma>1$, the integer design $\rho_\Gamma$ in
Theorem~\ref{s7-main} attains the common positive optimum
$C^*(1;\Gamma)=C^*_{\mathbb Z}(1;\Gamma)$.
This value tends to zero as $\Gamma\to\infty$.
Without a prescribed bound on the condition number after sampling,
the infimum is $C^*(1)=0$ and is not attained.
Appendix~\ref{app:tight-endpoint-expansions} gives asymptotic
expansions as $\Gamma\downarrow1$ and $\Gamma\to\infty$.

We next show that, for $\Gamma>1$, the optimal counting coefficient
$1/h_\Gamma$ from Theorem~\ref{s7-main} also gives a lower bound for
$\underline D_{\log}(\operatorname{supp}\mu)$ for every
$\mu\in\mathcal U_w(1;\Gamma)$, even when $D(\mu)=\infty$.
The lower logarithmic density, defined in
\eqref{eq:lower-log-density}, counts sampling times in intervals
$(R,e^\ell R]$ of logarithmic length $\ell$.
The infimum over $R\geq1$ requires a lower bound that holds
uniformly in $R$ as $\ell\to\infty$.

\begin{theorem}
\label{s7-local-counting}
Fix $\Gamma>1$. For every $\gamma>\Gamma$, there is a constant
$C_{\Gamma,\gamma}<\infty$ such that every
$\mu\in\mathcal U_w(1;\Gamma)$, with support $T$, satisfies
\begin{equation}\label{s7-local-window}
 \#\bigl(T\cap(R,e^\ell R]\bigr)
 \geq\frac{\ell}{h_\gamma}-C_{\Gamma,\gamma},
 \qquad R\geq1,\quad \ell\geq0.
\end{equation}
Consequently,
\begin{equation}\label{s7-local-sharp-density}
 \underline D_{\log}(T)\geq\frac1{h_\Gamma}.
\end{equation}
\end{theorem}

The proof is given in Appendix~\ref{app:periodic-tools}.
In logarithmic coordinates, it modifies the measure near the interval endpoints and extends it periodically. The resulting bounds on the convolution with $k$ and on the number of added nodes are uniform over the sampling design and the interval location. For an arbitrary design, subtracting global asymptotic counting estimates does not by itself give this uniform local bound.

\begin{remark}
For every $\Gamma>1$, the infimum of each quantity
\[
 \underline D_{\log}(\operatorname{supp}\mu),\qquad
 \liminf_{R\to\infty}\frac{N_\mu(R)}{\log R},\qquad D(\mu)
\]
over $\mu\in\mathcal U_w(1;\Gamma)$ is $1/h_\Gamma$.
The integer design $\rho_\Gamma$ in Theorem~\ref{s7-main}
attains all three infima. Their values therefore remain unchanged
when the support is restricted to $\mathbb N_0$.
\end{remark}

\begin{proof}
Set $T=\operatorname{supp}\mu$.
Taking $R=1$ in the infimum in \eqref{eq:lower-log-density}
and using the finiteness of $T\cap[0,1]$ gives
$\underline D_{\log}(T)
\leq\liminf_{r\to\infty}N_\mu(r)/\log r\leq D(\mu)$.
Theorem~\ref{s7-local-counting} therefore bounds all three
quantities below by $1/h_\Gamma$.

For $\rho_\Gamma$, the proof of Theorem~\ref{s7-main} gives
$N_{\rho_\Gamma}(r)=h_\Gamma^{-1}\log r+O_\Gamma(1)$.
By local finiteness, the error remains bounded on the remaining
bounded interval, so this estimate holds uniformly for $r\geq1$.
Subtracting its values at $r=R$ and $r=e^\ell R$ yields
\[
 \#\bigl(\operatorname{supp}\rho_\Gamma\cap(R,e^\ell R]\bigr)
 =\frac{\ell}{h_\Gamma}+O_\Gamma(1),
 \qquad R\geq1,\quad \ell\geq0,
\]
where the error is uniform in both $R$ and $\ell$.
Thus $\underline D_{\log}(\operatorname{supp}\rho_\Gamma)=1/h_\Gamma$,
and the global counting estimate gives the same value for the other
two quantities. Since $\rho_\Gamma$ is supported on $\mathbb N_0$,
it attains the infima in both support classes.
\end{proof}

\begin{remark}\label{s7-general-bound}
Fix $\Gamma>\kappa\geq1$. For every $\gamma>\Gamma/\kappa$,
there is a constant $C_{\Gamma/\kappa,\gamma}<\infty$,
depending only on $\Gamma/\kappa$ and $\gamma$, such that every
$\mu\in\mathcal U_w(\kappa;\Gamma)$, with support $T$, satisfies
\[
 \#\bigl(T\cap(R,e^\ell R]\bigr)
 \geq\frac{\ell}{h_\gamma}-C_{\Gamma/\kappa,\gamma},
 \qquad R\geq1,\quad \ell\geq0.
\]
Consequently, $\underline D_{\log}(T)\geq1/h_{\Gamma/\kappa}$ and
\(
 C^*_{\mathbb Z}(\kappa;\Gamma)
 \geq C^*(\kappa;\Gamma)
 \geq\max\left\{C^*(\kappa),1/h_{\Gamma/\kappa}\right\}>0.
\)
For $\kappa>1$, the term $C^*(\kappa)$ in this maximum can be
replaced by $p(\kappa,\Gamma)C^*(\kappa)$, where
$p(\kappa,\Gamma)>1$ is given by Theorem~\ref{s6-strict-gap}.
\end{remark}

\begin{proof}
Lemma~\ref{s6:scalar} gives
$\sup_{[0,1)}\phi_\mu/\inf_{[0,1)}\phi_\mu\leq\Gamma/\kappa$.
The characterization in \eqref{s7-scalar} therefore implies
$\mu\in\mathcal U_w(1;\Gamma/\kappa)$.
Theorem~\ref{s7-local-counting} gives the local counting estimate
and the lower logarithmic density bound.

Since $\underline D_{\log}(T)\leq D(\mu)$, taking infima gives
$C^*(\kappa;\Gamma)\geq1/h_{\Gamma/\kappa}$.
The inclusion
$\mathcal U_w(\kappa;\Gamma)\subseteq\mathcal U_w(\kappa)$
gives the other term in the maximum.
Restricting the support to $\mathbb N_0$ preserves these lower bounds.
For $\kappa>1$, Theorem~\ref{s6-strict-gap} gives the stated
improvement.
\end{proof}

For $\kappa>1$, the operators $A$ and $S$ need not commute.
Remark~\ref{s7-general-bound} gives necessary density bounds,
but does not determine the exact value of $C^*(\kappa;\Gamma)$.

Together, Theorems~\ref{s6-strict-gap} and \ref{s7-main} show that,
for each finite $\kappa\geq1$, a sequence
$\mu_n\in\mathcal U_w(\kappa)$ with $D(\mu_n)\to C^*(\kappa)$
cannot retain a common finite bound on the condition number after
sampling over all systems whose original frame condition number
is at most $\kappa$.

\begin{corollary}
\label{s6-strict-conditioning}
Fix $1\leq\kappa<\infty$. For $\mu\in\mathcal U_w(\kappa)$, define
\[
 \Omega_\kappa(\mu)=
 \sup_{\substack{\text{positive systems }(A,C)\\
                  \operatorname{cond}(S)\leq\kappa}}
       \operatorname{cond}(G_\mu).
\]
If $\mu_n\in\mathcal U_w(\kappa)$ and
$D(\mu_n)\to C^*(\kappa)$, then
$\Omega_\kappa(\mu_n)\to\infty$.
The same conclusion holds for designs supported on $\mathbb N_0$
with $D(\mu_n)\to C^*_{\mathbb Z}(\kappa)$.
\end{corollary}

\begin{proof}
If the conclusion fails, a subsequence satisfies
$\Omega_\kappa(\mu_n)\leq\Gamma$ for some finite $\Gamma>\kappa$.
Every measure in this subsequence then belongs to
$\mathcal U_w(\kappa;\Gamma)$, so
$D(\mu_n)\geq C^*(\kappa;\Gamma)$.

For $\kappa>1$, Theorem~\ref{s6-strict-gap} gives
$C^*(\kappa;\Gamma)>C^*(\kappa)$, contradicting
$D(\mu_n)\to C^*(\kappa)$.
For $\kappa=1$, Theorems~\ref{thm:tight-endpoint} and
\ref{s7-main} give $C^*(1)=0<C^*(1;\Gamma)$,
which yields the same contradiction.
The integer assertion follows from
$C^*_{\mathbb Z}(\kappa)=C^*(\kappa)$ in
Theorem~\ref{s4-integer}.
\end{proof}

\section{Normal operators and logarithmic sectors}
\label{nrm-section}

The preceding sections study optimal counting coefficients for
positive operators, with and without a prescribed bound on the
condition number after sampling. We now consider normal operators.
For tight original frames, the sampled Gramian retains the scalar
representation in Proposition~\ref{nrm-tight}.
For $\kappa>1$, Theorem~\ref{nrm-obstruction} shows that a common
design for all normal systems with original frame condition number
at most $\kappa$ must satisfy $N_\mu(R)\gtrsim R$.
Under the logarithmic sector condition \eqref{nrm-logsector},
Theorem~\ref{nrm-relative} gives designs with
$N_\mu(R)=O(\log R)$.
Theorem~\ref{nrm-sharp} determines the leading asymptotic constant
for the optimal counting coefficient as $\kappa\to\infty$,
with the sector parameter $M$ fixed.

For a fixed vector $f$, Philipp \cite{PhilippBessel} characterizes
the Bessel property of the orbit $(A^n f)_{n\geq0}$ of a normal
operator $A$ in terms of the scalar measure $\|E_A(\cdot)f\|^2$,
where $E_A$ is the spectral measure of $A$.
This concerns the upper frame inequality for the given orbit.
Here we assume both frame inequalities for the original integer
observations and seek weighted sampling measures that preserve
them uniformly over a specified class of normal operators.

Throughout this section, $A\in\mathcal B(H)$ is normal, and we assume
\begin{equation}\label{nrm-input}
 S=\sum_{n=0}^{\infty}(A^n)^*BA^n,\qquad
 B=C^*C,\qquad aI\leq S\leq bI,\qquad b/a\leq\kappa,
\end{equation}
where $0<a\leq b<\infty$.
All sums of positive operators are understood in the strong
operator topology.
Universality means that a fixed measure $\mu$ gives bounds
$\alpha I\leq G_\mu\leq\beta I$ for every system in the specified
spectral class with $I\leq S\leq\kappa I$, where
$0<\alpha\leq\beta<\infty$ are independent of the system and its
dimension. The general case reduces to this normalization by
replacing $C$ with $C/\sqrt a$.

The Stein identity and the tail identity remain valid:
\begin{equation}\label{nrm-stein}
 B=S-A^*SA,\qquad
 (A^n)^*SA^n=\sum_{k=n}^{\infty}(A^k)^*BA^k.
\end{equation}
Since the tails converge strongly to zero and $S\geq aI$,
we have $A^nx\to0$ for every $x$ and
$\sup_n\|A^n\|\leq\sqrt{b/a}$.
Normality then gives $\|A\|\leq1$, and strong stability implies
$E_A(\{\lambda:|\lambda|=1\})=0$.

For $t>0$, define $A^t$ by the bounded Borel functional calculus,
using $\lambda^t=|\lambda|^t e^{it\arg\lambda}$ for $\lambda\ne0$,
with $-\pi<\arg\lambda\leq\pi$ and $0^t=0$.
Set $A^0=I$. Thus
\(
 G_\mu=\sum_{t\in T}w_t(A^t)^*BA^t
\)
is well defined under this convention even when zero or points
of the negative real axis belong to the spectrum.
The adjoint $(A^t)^*$ is always taken after forming $A^t$.

\subsection{Tight frames and the obstruction to logarithmic sampling}

For tight original frames, the universal sampling classes are the
same for normal and positive operators. For $\kappa>1$, we show
that a common design for all normal systems with
$\operatorname{cond}(S)\leq\kappa$ must satisfy
$N_\mu(R)\gtrsim R$.

\begin{proposition}\label{nrm-tight}
Suppose $A$ is normal and $S=aI$. Then $B=a(I-A^*A)$.
If $\phi_\mu$ is bounded on $[0,1)$, the sampled sum converges
strongly and
\(
 G_\mu=a\phi_\mu(|A|^2).
\)
The universal weighted sampling classes for tight original frames
coincide for normal and positive operators, with or without a
prescribed bound on the condition number after sampling.
\end{proposition}

\begin{proof}
Write $|A|=(A^*A)^{1/2}$ and recall that
$\phi_\mu(u)=(1-u)\sum_{t\in T}w_tu^t$ for $0\leq u<1$,
with $0^0=1$.
The Stein identity \eqref{nrm-stein} gives $B=a(I-|A|^2)$.
Since $A$ is normal, $B$ commutes with every $A^t$, and
$(A^t)^*BA^t=a(I-|A|^2)|A|^{2t}$ for $t\geq0$.
Apply the spectral theorem to finite sampling sums.
When $\phi_\mu$ is bounded, monotone convergence gives strong
convergence and the formula for $G_\mu$.
The value assigned to $\phi_\mu(1)$ is irrelevant because the
spectral projection of $|A|$ at $1$ is zero.

By Proposition~\ref{prop:tight}, universality for tight positive
systems is equivalent to positive uniform lower and upper bounds
for $\phi_\mu$ on $[0,1)$.
The formula for $G_\mu$ makes these bounds sufficient for tight
normal systems. Necessity follows because positive operators
are normal.

For a prescribed condition number bound $\Gamma\geq1$,
positive diagonal systems with two spectral values test every ratio
$\phi_\mu(u)/\phi_\mu(v)$, where $u,v\in[0,1)$.
These ratios are at most $\Gamma$ if and only if
$\operatorname{cond}(G_\mu)\leq\Gamma$ for every tight normal
system. Thus the classes also coincide under this constraint.
Consequently, the conclusions for tight frames in
Proposition~\ref{prop:tight} and
Theorems~\ref{thm:tight-endpoint} and \ref{s7-main}
apply to normal operators.
\end{proof}

\begin{theorem}\label{nrm-obstruction}
Fix $\kappa>1$ and $0<c\leq\pi$. Consider all normal systems
satisfying $I\leq S\leq\kappa I$ whose spectra lie in a prescribed
closed angular sector of the unit disk with opening $2c$.
If $\mu$ is a common weighted design for this class with bounds
$\alpha I\leq G_\mu\leq\beta I$, then
\begin{equation}\label{nrm-weight-bound}
 \sup_{t\in T}w_t\leq\frac{\pi\beta}{c(\kappa-1)}.
\end{equation}
Moreover, $M_\mu(R)\asymp R$ and $N_\mu(R)\gtrsim R$
as $R\to\infty$. The counting order is attained by
$\mu=\sum_{n\geq0}\delta_n$.
\end{theorem}

\begin{proof}
For $N\geq2$, set $q_N=\kappa^{-1/(N-1)}$ and
$\omega_N=e^{2\pi i/N}$, and define
\[
 A_N=\sqrt{q_N}\operatorname{diag}
      (1,\omega_N,\ldots,\omega_N^{N-1}),\qquad
 C_N=\sqrt{\frac{\kappa-q_N}{N}}\,(1,\ldots,1).
\]
In the discrete Fourier basis, the integer Gramian $S_N$ has
eigenvalues $(\kappa-q_N)q_N^\ell/(1-q_N^N)=\kappa q_N^\ell$
for $0\leq\ell<N$.
Since $q_N^{N-1}=\kappa^{-1}$, the smallest and largest
eigenvalues are $1$ and $\kappa$.

For sufficiently large $N$, restrict to the $m_N$ spectral
coordinates in the prescribed sector, and denote the resulting
operators by $\widetilde A_N$ and $\widetilde C_N$.
The coordinate projection commutes with $A_N$, so the new integer
Gramian is a compression of $S_N$ and still lies between
$I$ and $\kappa I$.
Also, $m_N/N\to c/\pi$ and
$\|\widetilde C_N^*\widetilde C_N\|
=(m_N/N)(\kappa-q_N)$.

For each fixed $t\in T$, the operator $\widetilde A_N^t$ is
$q_N^{t/2}$ times a unitary matrix.
Positivity of the sampled sum and its upper bound give
$w_tq_N^t(m_N/N)(\kappa-q_N)\leq\beta$.
Letting $N\to\infty$ with $t$ fixed proves
\eqref{nrm-weight-bound}.

Choose a fixed ray in the sector and consider the scalar systems
$A=re^{i\vartheta}$ and $C=\sqrt{1-r^2}$ for $0\leq r<1$.
They have $S=1$, and the sampling bounds give
$\alpha\leq\phi_\mu(r^2)\leq\beta$.
Proposition~\ref{prop:tight} therefore yields
$M_\mu(R)\asymp R$.
Combining this with \eqref{nrm-weight-bound} gives
$N_\mu(R)\gtrsim R$.

Finally, $\mu=\sum_{n\geq0}\delta_n$ gives $G_\mu=S$ and
$N_\mu(R)=\lfloor R\rfloor+1$, so the counting order is sharp.
\end{proof}

Proposition~\ref{nrm-tight} gives the same scalar sampling criteria for tight original frames generated by normal and positive operators. For $\kappa>1$, Theorem~\ref{nrm-obstruction} requires $N_\mu(R)\gtrsim R$ for any common design over all normal systems with $I\leq S\leq\kappa I$.
This lower bound persists when the spectra are restricted to a
fixed angular sector of any positive opening.

\subsection{Logarithmic sectors and optimal density}

On $H_+=(\ker A)^\perp$, define $L=-\operatorname{Log}A$
using the argument $-\pi<\arg\lambda\leq\pi$, and put
$R=\operatorname{Re}L$. The operator $L$ may be unbounded.
For a fixed $M\geq0$, impose the logarithmic sector condition
\begin{equation}\label{nrm-logsector}
 |\operatorname{Im}L|\leq MR.
\end{equation}
This inequality is understood in the spectral sense: the spectral
measure of $L$ is concentrated on
$\{z:\operatorname{Re}z>0,\,
|\operatorname{Im}z|\leq M\operatorname{Re}z\}$.
Equivalently, the spectral measure of $A$ on $H_+$ is concentrated
where $0<|\lambda|<1$ and
$|\arg\lambda|\leq M(-\log|\lambda|)$.
The condition places no restriction on $\ker A$.

The Stolz condition $|1-\lambda|\leq A_0(1-|\lambda|)$,
with a fixed $A_0>0$, implies \eqref{nrm-logsector}
with $M=\pi(A_0+1)/2$.
Indeed, write $\lambda=re^{i\varphi}$ with $0<r<1$ and
$\varphi=\arg\lambda$. The triangle inequality gives
$|e^{i\varphi}-1|\leq|1-\lambda|+1-r$, so
\[
 |\varphi|
 \leq\frac{\pi}{2}|e^{i\varphi}-1|
 \leq\frac{\pi}{2}(A_0+1)(1-r)
 \leq\frac{\pi}{2}(A_0+1)(-\log r).
\]

For the estimates below, set $a_M=\sqrt{1+M^2}$ and
$d_0=\pi/2-\arctan M$.
For $|\theta|<d_0$, write
$c_\theta=\cos\theta-M|\sin\theta|>0$.
We retain $K_0=4(3+2\sqrt2)$ from
Lemma~\ref{lem:complex-energy}.
The exponentials $e^{-te^{i\theta}L}$ used below, with
$t>0$ and $|\theta|<d_0$, are extended by zero on $\ker A$.

\begin{lemma}\label{nrm-complex-energy}
Under \eqref{nrm-input} and \eqref{nrm-logsector}, for every
$x\in H$ and $|\theta|<d_0$,
\[
 \int_0^\infty\|Ce^{-te^{i\theta}L}x\|^2\,dt
 \leq\frac{a_MK_0b}{c_\theta}\|x\|^2.
\]
\end{lemma}

\begin{proof}
Apply the block argument of Lemma~\ref{lem:complex-energy} to
$P_j=E_R([2^j,2^{j+1}))$ on $H_+$.
Write $X_{ij}=P_iX|_{P_jH_+}$ and $A_j=A|_{P_jH_+}$.
The compressed Stein identity is now
$B_{jj}=(I-A_j^*)S_{jj}+A_j^*S_{jj}(I-A_j)$.
Since $|1-e^{-z}|\leq|z|$ for $\operatorname{Re}z\geq0$
and $|z|\leq a_M\operatorname{Re}z$ on the spectrum of $L$,
we have $\|I-A_j\|\leq a_M2^{j+1}$.
Thus $\|B_{jj}\|\leq4a_Mb2^j$, and positivity gives
$\|B_{ij}\|\leq4a_Mb2^{(i+j)/2}$.

The other change is the decay estimate
$\|e^{-te^{i\theta}L}P_j\|\leq e^{-tc_\theta2^j}$,
which follows from
$\operatorname{Re}(e^{i\theta}z)\geq c_\theta\operatorname{Re}z$.
The same block Schur estimate therefore applies with $b$ replaced
by $a_Mb$ and $\cos\theta$ by $c_\theta$.
It gives $a_MK_0b/c_\theta$ as the bound on finite spectral sums.
The density and Fatou argument in Lemma~\ref{lem:complex-energy},
followed by extension by zero on $\ker A$, proves the result.
\end{proof}

\begin{lemma}\label{nrm-rounding-lemma}
Assume \eqref{nrm-input} and \eqref{nrm-logsector}.
Let $\tau=\sum_{t\in T}w_t\delta_t$ be a positive atomic measure
supported on $[J,\infty)$, where $J\geq1$, and suppose
$\tau([0,s])\leq Ks$ for every $s>0$.
Define $T_\tau x=(\sqrt{w_t}CA^tx)_{t\in T}$ and
$\widehat T_\tau x=(\sqrt{w_t}CA^{\lceil t\rceil}x)_{t\in T}$,
retaining the original indices after rounding.
If $T_\tau$ is bounded, then $\widehat T_\tau$ is bounded and
\begin{equation}\label{nrm-rounding}
 \|\widehat T_\tau-T_\tau\|^2
 \leq\frac{a_M^2bK}{J}.
\end{equation}
\end{lemma}

\begin{proof}
Since all sampling times are positive, both maps vanish on
$\ker A$. It therefore suffices to work on $H_+$.
The spectral measure of $L$ is concentrated where
$r=\operatorname{Re}z>0$ and $|z|\leq a_Mr$.
For such $z$ and $t\geq J$, the inequality
$0\leq\lceil t\rceil-t<1$ gives
$|e^{-\lceil t\rceil z}-e^{-tz}|
\leq|z|e^{-tr}\leq a_Mre^{-tr}$.

Using the support and mass assumptions, integration by parts gives
\[
 r^2\int_{[J,\infty)}e^{-2rt}\,d\tau(t)
 \leq Kr^2\left(J+\frac1{2r}\right)e^{-2rJ}
 \leq\frac KJ.
\]
The last inequality follows from
$(v^2+v/2)e^{-2v}\leq1$ for $v\geq0$, with $v=rJ$.

Since $B\leq S\leq bI$, the spectral theorem and these estimates
give, for every finite $F\subset T$,
\[
 \sum_{t\in F}w_t
 \|C(A^{\lceil t\rceil}-A^t)x\|^2
 \leq\frac{a_M^2bK}{J}\|x\|^2.
\]
Taking the supremum over finite $F$ shows that
$\Delta_\tau x=
(\sqrt{w_t}C(A^{\lceil t\rceil}-A^t)x)_{t\in T}$
defines a bounded operator.
If $T_\tau$ is bounded, then
$\widehat T_\tau=T_\tau+\Delta_\tau$ is bounded, and the
estimate for $\Delta_\tau$ proves \eqref{nrm-rounding}.
\end{proof}

We now use Lemmas~\ref{nrm-complex-energy}
and~\ref{nrm-rounding-lemma} to extend the quadrature construction
of Proposition~\ref{prop:quadrature} to the present spectral class.

\begin{theorem}\label{nrm-relative}
Fix $M\geq0$, $\kappa\geq1$, and $0<\varepsilon<1$.
For $0<d<d_0$, put
$h=2\pi d/\log(1+16a_MK_0\kappa/(\varepsilon c_d))$.
There is a positive measure $\mu$ supported on $\mathbb N_0$,
depending only on $M,\kappa,\varepsilon,d$, such that every normal
system satisfying \eqref{nrm-input} and \eqref{nrm-logsector} has
\[
 (1-\varepsilon)S\leq G_\mu\leq(1+\varepsilon)S,
 \qquad
 N_\mu(R)=\frac{\log R}{h}+O(1),
 \quad R\to\infty.
\]
\end{theorem}

\begin{proof}
Set $\sigma=\varepsilon/(8\kappa)$ and
$K_h=he^h/(e^h-1)$.
Choose an integer
\[
 J\geq\max\left\{1,\frac{8a_M\kappa}{\varepsilon},
                \frac{64a_M^2\kappa K_h}{\varepsilon^2}\right\},
\]
and define
\(
 \mu=\sum_{n=0}^{J-1}\delta_n+
       \sum_{j=0}^\infty h\sigma e^{jh}
                   \delta_{J+\lceil\sigma e^{jh}\rceil},
\)
adding weights at coincident nodes.

Tonelli's theorem and the semigroup identity give
\[
 G_c=\int_0^\infty(A^t)^*BA^t\,dt
     =\int_0^1(A^s)^*SA^s\,ds\leq bI.
\]
The value at time zero does not affect these integrals,
and $G_c$ vanishes on $\ker A$.
For $x\in H_+$ with spectral support in a compact subinterval of
$(0,\infty)$ with respect to $R$, apply the quadrature argument of
Proposition~\ref{prop:quadrature} to
$f_x(z)=e^z\langle e^{-e^zL^*}B e^{-e^zL}x,x\rangle$.
This function is analytic in $|\operatorname{Im}z|<d_0$;
both exponentials use $e^z$.
With $t=e^s$, Cauchy--Schwarz and
Lemma~\ref{nrm-complex-energy} give
\[
 \int_{\mathbb R}|f_x(s+i\theta)|\,ds
 \leq
 \left(\int_0^\infty\|Ce^{-te^{i\theta}L}x\|^2\,dt\right)^{1/2}
 \left(\int_0^\infty\|Ce^{-te^{-i\theta}L}x\|^2\,dt\right)^{1/2}
 \leq\frac{a_MK_0b}{c_\theta}\|x\|^2.
\]
For $|\theta|\leq d$, we have $c_\theta\geq c_d>0$, and the
spectral restriction gives uniform decay as $s\to\pm\infty$.
Thus the trapezoidal estimate and the positive operator convergence
argument in Proposition~\ref{prop:quadrature}, applied to
$f_x(\,\cdot+\log\sigma)$, show that
\[
 Q_{h,\sigma}
 =\sum_{j\in\mathbb Z}h\sigma e^{jh}
       (A^{\sigma e^{jh}})^*BA^{\sigma e^{jh}},
 \qquad
 \|Q_{h,\sigma}-G_c\|\leq bE_M(h,d),
\]
where the series converges strongly and
$E_M(h,d)=2a_MK_0/[c_d(e^{2\pi d/h}-1)]$.
Both operators vanish on $\ker A$, so the estimate holds on $H$.
The choice of $h$ gives $E_M(h,d)=\varepsilon/(8\kappa)$.

For $0\leq s\leq1$, the spectral theorem gives
$\|A^{J+s}-A^J\|\leq sa_M/(eJ)$, since
$\sup_{r>0}re^{-Jr}=1/(eJ)$.
Using the formula for $G_c$ and $\|A^t\|\leq1$, we obtain
\[
 \|(A^J)^*(G_c-S)A^J\|
 \leq2b\int_0^1\|A^{J+s}-A^J\|\,ds
 \leq\frac{ba_M}{eJ}.
\]

Let $Q_+$ denote the part of $Q_{h,\sigma}$ with $j\geq0$.
The omitted weights have total mass
$h\sigma/(e^h-1)\leq\sigma$, so
$\|Q_{h,\sigma}-Q_+\|\leq b\sigma$.
Define the unrounded tail
$\tau=\sum_{j\geq0}h\sigma e^{jh}\delta_{J+\sigma e^{jh}}$
and put $\mu_J=\sum_{n=0}^{J-1}\delta_n+\tau$.
The tail identity \eqref{nrm-stein} gives
$G_{\mu_J}-S=(A^J)^*(Q_+-S)A^J$ on all of $H$.
Consequently,
\[
 \|G_{\mu_J}-S\|
 \leq b\bigl(a_M/J+\sigma+E_M(h,d)\bigr)
 \leq\frac{3\varepsilon b}{8\kappa}
 \leq\frac{\varepsilon a}{2}.
\]
Thus $(1-\varepsilon/2)S\leq G_{\mu_J}\leq(1+\varepsilon/2)S$.

Summing the geometric weights gives
$\tau([0,s])\leq K_hs$ for every $s>0$.
Also, $G_\tau\leq G_{\mu_J}$, so $T_\tau$ is bounded.
Lemma~\ref{nrm-rounding-lemma} and the choice of $J$ yield
\[
 \|(\widehat T_\tau-T_\tau)x\|
 \leq\frac{\varepsilon}{8}\sqrt a\,\|x\|
 \leq\frac{\varepsilon}{8}\|S^{1/2}x\|.
\]
Adjoining the same integer prefix to both analysis operators
and applying the triangle inequality gives
\[
 (\sqrt{1-\varepsilon/2}-\varepsilon/8)^2S
 \leq G_\mu
 \leq(\sqrt{1+\varepsilon/2}+\varepsilon/8)^2S.
\]
For $0<\varepsilon<1$, the lower coefficient is at least
$1-\varepsilon$ and the upper coefficient is at most
$1+\varepsilon$. This proves the required frame bounds.

Finally, the gaps
$\sigma e^{jh}(e^h-1)$ tend to infinity, so only finitely many
rounded nodes coincide. Rounding changes the counting function
by $O(1)$, and the finite prefix does not affect its leading term.
Hence $N_\mu(R)=(\log R)/h+O(1)$.
\end{proof}

For fixed $M\geq0$ and $\varepsilon\in(0,1)$, choose
$d=d_0-1/\log\kappa$ for sufficiently large $\kappa$.
Since $c_d=a_M\sin(1/\log\kappa)$, the construction in
Theorem~\ref{nrm-relative} gives
\begin{equation}\label{nrm-upper-asymptotic}
 D(\mu)=\frac1h
 =\frac{\log\kappa+\log\log\kappa+O_{M,\varepsilon}(1)}
        {\pi(\pi-2\arctan M)},
 \qquad \kappa\to\infty.
\end{equation}

For $M\geq0$ and $\kappa\geq1$, let $\mathcal U_{w,M}(\kappa)$
consist of the positive atomic measures with locally finite support
for which there are constants $0<\alpha\leq\beta<\infty$ such that
$\alpha I\leq G_\mu\leq\beta I$ for every normal system satisfying
\eqref{nrm-logsector} and $I\leq S\leq\kappa I$.
The constants may depend on $\mu,M,\kappa$, but not on the system
or its dimension.
Scaling $C$ by $a^{-1/2}$ shows that this is equivalent to requiring
$a\alpha I\leq G_\mu\leq a\beta I$ whenever
$aI\leq S\leq bI$ and $b/a\leq\kappa$.

Define
\[
 C_M^*(\kappa)=\inf_{\mu\in\mathcal U_{w,M}(\kappa)}D(\mu),
 \qquad
 C_{M,\mathbb Z}^*(\kappa)=
 \inf_{\substack{\mu\in\mathcal U_{w,M}(\kappa)\\
                  \operatorname{supp}\mu\subset\mathbb N_0}}D(\mu).
\]
No bound on the condition number after sampling is prescribed.
For $M=0$, the sector condition restricts $A$ to positive operators,
so $\mathcal U_{w,0}(\kappa)=\mathcal U_w(\kappa)$.

\begin{theorem}\label{nrm-integer}
For $M\geq0$ and $\kappa\geq1$, every
$\mu\in\mathcal U_{w,M}(\kappa)$ admits a measure
$\widehat\nu\in\mathcal U_{w,M}(\kappa)$ supported on $\mathbb N_0$
with $D(\widehat\nu)\leq D(\mu)$.
Consequently, $C_{M,\mathbb Z}^*(\kappa)=C_M^*(\kappa)$.
\end{theorem}

\begin{proof}
Fix $\mu\in\mathcal U_{w,M}(\kappa)$ with uniform bounds
$\alpha I\leq G_\mu\leq\beta I$ on the normalized class.
Since this class contains every positive system with $S=I$,
Proposition~\ref{prop:tight} gives
$\mu([0,R])\leq K(1+R)$ for some $K>0$ and all $R\geq0$.
Apply the construction from Theorem~\ref{s4-integer}: translate
the measure, add observations at finitely many initial integer
times, and round the remaining sampling times.
More precisely, choose $\eta=\max\{1,\kappa/\alpha\}$ and an
integer $J\geq\max\{1,4a_M^2\eta\kappa K\}$, and set
\[
 \nu_J=\sum_{n=0}^{J-1}\delta_n+\eta\tau_J\mu,
 \qquad
 \widehat\nu_J=\sum_{n=0}^{J-1}\delta_n+
 \eta\sum_{t\in\operatorname{supp}\mu}
              w_t\delta_{J+\lceil t\rceil}.
\]
Here $\tau_J$ denotes translation by $J$.
The identity used there becomes
$G_{\nu_J}=S+(A^J)^*(\eta G_\mu-S)A^J$, so
$I\leq G_{\nu_J}\leq(\kappa+\eta\beta)I$.
The translated tail has cumulative mass at most $\eta Ks$
for $s\geq J$, and its analysis operator is bounded by the upper
bound for $G_{\nu_J}$.
Lemma~\ref{nrm-rounding-lemma} replaces Lemma~\ref{s4-rounding}
and bounds the squared difference of the two tail analysis
operators by $a_M^2\eta\kappa K/J\leq1/4$.
The triangle inequality then gives
\[
 \tfrac14 I\leq G_{\widehat\nu_J}
 \leq\left(\sqrt{\kappa+\eta\beta}+\tfrac12\right)^2I.
\]
The parameters are independent of the system, and combining
coincident atoms preserves the Gramian.
Thus $\widehat\nu_J\in\mathcal U_{w,M}(\kappa)$.
As in Theorem~\ref{s4-integer},
$N_{\widehat\nu_J}(R)\leq J+N_\mu(R)$ for $R\geq J$, so
$D(\widehat\nu_J)\leq D(\mu)$.
Taking infima, and using the reverse inequality from restriction
to integer supports, proves the equality.
\end{proof}

We next prove that the coefficient of $\log\kappa$ in
\eqref{nrm-upper-asymptotic} is optimal for each fixed $M$.
The lower bound uses eigenvalues obtained from conjugate points
$w=re^{\pm i\theta}$ by the Cayley map $w\mapsto(1-w)/(1+w)$.

In Lemma~\ref{nrm-two-rays}, we construct such a system with a
scalar observation operator $C:H\to\mathbb C$ and estimate its
integer Gramian.

\begin{lemma}\label{nrm-two-rays}
Let $M>0$, $0<\theta<\arctan M$, and $h>0$.
There exist an invertible, strongly stable diagonal normal operator
$A$ satisfying \eqref{nrm-logsector} and a bounded operator
$C:H\to\mathbb C$ whose integer Gramian $S_h$ is bounded and
invertible. If
\begin{equation}\label{nrm-step-range}
 0<h\leq\min\{\pi(\pi-2\theta),\,2\pi\theta\},
\end{equation}
then $\operatorname{cond}(S_h)\leq64\exp\bigl(2\pi(\pi-2\theta)/h\bigr)$.
\end{lemma}

\begin{proof}
Choose $0<r_0\leq1/16$, to be fixed below.
For $j\geq1$ and $\sigma\in\{+1,-1\}$, set
$r_j=r_0e^{-jh}$, $w_{j,\sigma}=r_je^{i\sigma\theta}$, and
$\lambda_{j,\sigma}=(1-w_{j,\sigma})/(1+w_{j,\sigma})$.
On $H=\ell^2(\mathbb N\times\{+1,-1\})$, define
\begin{equation}\label{nrm-model-observation}
 Ae_{j,\sigma}=\lambda_{j,\sigma}e_{j,\sigma},
 \qquad
 Ce_{j,\sigma}=\frac{2\sqrt{r_j\cos\theta}}{1+w_{j,\sigma}}.
\end{equation}
The coefficients of $C$ are square summable, so $C$ is bounded.
Since $(1-r_0)/(1+r_0)\leq|\lambda_{j,\sigma}|<1$,
the operator $A$ is invertible, and strong stability follows by
dominated convergence.

For $|w|\leq r_0$, the principal logarithm satisfies
$-\operatorname{Log}((1-w)/(1+w))=2\operatorname{atanh}w$ and
$|\operatorname{atanh}w-w|\leq|w|^3/[3(1-|w|^2)]$.
Set $\varepsilon_0=r_0^2/[3(1-r_0^2)]$ and choose $r_0$ so that
\begin{equation}\label{nrm-model-sector-choice}
 \varepsilon_0\leq
 \frac{M\cos\theta-\sin\theta}{M+1}.
\end{equation}
The right side is positive and smaller than $\cos\theta$.
Hence
\[
 \frac{|\operatorname{Im}(-\operatorname{Log}\lambda_{j,\sigma})|}
      {\operatorname{Re}(-\operatorname{Log}\lambda_{j,\sigma})}
 \leq\frac{\sin\theta+\varepsilon_0}
           {\cos\theta-\varepsilon_0}\leq M,
\]
which proves \eqref{nrm-logsector}.

Write $i,k$ for the pairs indexing the basis vectors.
On finitely supported vectors, summing the geometric series gives
the Gramian entries
\[
 (S_h)_{ik}
 =\frac{\overline{Ce_i}\,Ce_k}
        {1-\overline{\lambda_i}\lambda_k}
 =\frac{2\sqrt{\operatorname{Re}w_i\operatorname{Re}w_k}}
        {\overline{w_i}+w_k}.
\]
Group the coordinates in the order $+1,-1$ for each $j$.
The block with indices $j,k$ is $K_{j-k}$, where
\[
 K_n=\cos\theta
 \begin{pmatrix}
 \operatorname{sech}(nh/2+i\theta)&
 e^{i\theta}\operatorname{sech}(nh/2)\\
 e^{-i\theta}\operatorname{sech}(nh/2)&
 \operatorname{sech}(nh/2-i\theta)
 \end{pmatrix}.
\]
These blocks are independent of $r_0$.

Put $U=\operatorname{diag}(e^{i\theta/2},e^{-i\theta/2})$.
The conjugation $U^*K_nU$ removes the phases from the entries
outside the diagonal.
Let $K(t)$ be obtained from $U^*K_nU$ by replacing $nh$ with $t$.
With the Fourier convention
$\widehat f(\xi)=\int_{\mathbb R}f(t)e^{-it\xi}\,dt$, its transform is
\[
 \begin{aligned}
 W_\theta(\xi)
 =2\pi\cos\theta\,\operatorname{sech}(\pi\xi)
 \begin{pmatrix}
 e^{-2\theta\xi}&1\\
 1&e^{2\theta\xi}
 \end{pmatrix}
 =2\pi\cos\theta\,\operatorname{sech}(\pi\xi)
 v(\xi)v(\xi)^*,
 \qquad
 v(\xi)=\binom{e^{-\theta\xi}}{e^{\theta\xi}}.
 \end{aligned}
\]
Indeed, the transform of $\operatorname{sech}(t/2)$ is
$2\pi\operatorname{sech}(\pi\xi)$.
Shifting the integration path by $2i\theta$ and $-2i\theta$
gives the respective factors $e^{-2\theta\xi}$ and $e^{2\theta\xi}$.
No pole is crossed because $\theta<\pi/2$.

Poisson summation gives
\[
 P_h(\omega)
 =\sum_{n\in\mathbb Z}U^*K_nU\,e^{-in\omega}
 =\frac{2\pi\cos\theta}{h}F_h(\omega),
\]
where
$F_h(\omega)=\sum_{k\in\mathbb Z}
\operatorname{sech}(\pi\xi_k)v(\xi_k)v(\xi_k)^*$
and $\xi_k=(\omega+2\pi k)/h$.
The series defining $F_h$ converges absolutely and uniformly
for $\omega\in[-\pi,\pi]$.

Put $a=\pi-2\theta>0$.
For $|\omega|\leq\pi$, set $\xi=\omega/h$ and choose
$\eta=\xi-2\pi/h$ if $\omega\geq0$, and
$\eta=\xi+2\pi/h$ otherwise.
Then $|\xi|+|\eta|=|\xi-\eta|=2\pi/h$.
Let $H_\omega$ be the sum of the two corresponding terms in $F_h$.
Its determinant is
\(
 \det H_\omega
 =4\operatorname{sech}(\pi\xi)\operatorname{sech}(\pi\eta)
       \sinh^2(2\pi\theta/h).
\)
Each term has trace
$2\cosh(2\theta\xi)/\cosh(\pi\xi)\leq2$.
Using $\operatorname{sech}u\geq e^{-|u|}$, we obtain
\[
 \lambda_{\min}F_h(\omega)
 \geq\frac{\det H_\omega}{\operatorname{tr}H_\omega}
 \geq e^{-2\pi^2/h}\sinh^2(2\pi\theta/h)>0.
\]
For the upper bound,
\(
 \operatorname{tr}F_h(\omega)
 \leq4\sum_{k\in\mathbb Z}e^{-a|\omega+2\pi k|/h}
 \leq4\coth(\pi a/h).
\)
The last inequality follows because the periodic exponential sum
is largest at $\omega=0$.

For a finitely supported vector $x=(x_j)_{j\geq1}$ with
$x_j\in\mathbb C^2$, set
$p_x(\omega)=\sum_{j\geq1}U^*x_j e^{-ij\omega}$.
The Gramian form satisfies
\[
 \sum_{m=0}^\infty\|CA^mx\|^2
 =\frac1{2\pi}\int_{-\pi}^{\pi}
 p_x(\omega)^*P_h(\omega)p_x(\omega)\,d\omega.
\]
The bounds for $F_h$ and Parseval's identity bound this expression
above and below by positive multiples of $\|x\|^2$.
Thus every finite positive time sum has a common upper bound,
first on finitely supported vectors and then on $H$ by density.
These sums converge strongly to a bounded positive block Toeplitz
operator $S_h$. The lower bound proves its invertibility.

Under \eqref{nrm-step-range}, both $\pi a/h\geq1$ and
$2\pi\theta/h\geq1$.
Using $\coth(\pi a/h)<2$ and
$\sinh u\geq(1-e^{-2})e^u/2$ for $u\geq1$, the preceding
symbol bounds give
\[
 \operatorname{cond}(S_h)
 \leq\frac{32}{(1-e^{-2})^2}e^{2\pi a/h}
 <64e^{2\pi a/h}.
\]
\end{proof}

We next bound the number of sampling times in each interval
$(R,e^\ell R]$, where $R\geq1$ and $\ell>0$.
The proof uses the two coordinates indexed by $\sigma=\pm1$
at each scale $r_j$ in the system \eqref{nrm-model-observation}.

\begin{lemma}\label{nrm-two-ray-counting}
If a positive locally finite atomic measure $\mu$ gives a frame for
the system \eqref{nrm-model-observation}, then there is a constant
$C_{\mu,h,\theta,r_0}<\infty$ such that
\[
 \#\bigl(\operatorname{supp}\mu\cap(R,e^LR]\bigr)
 \geq\frac{2L}{h}-C_{\mu,h,\theta,r_0},
 \qquad R\geq1,\quad L>0.
\]
In particular, $D(\mu)\geq2/h$.
\end{lemma}

\begin{proof}
Set $f_j=|Ce_{j,\sigma}|$ and
$s_j=-\log|\lambda_{j,\sigma}|$, both independent of $\sigma$.
Since $r_0\leq1/16$, we have
$f_j^2=4r_j\cos\theta/|1+w_{j,\sigma}|^2<1/2$ and
$s_j=-\log(1-f_j^2)/2$.
Thus there are constants $c_f,C_f,c_s,C_s>0$, depending only on
$\theta,r_0$, such that
\begin{equation}\label{nrm-ray-comparisons}
 c_fr_j\leq f_j^2\leq C_fr_j,\qquad
 c_sr_j\leq s_j\leq C_sr_j,\qquad
 f_j^2\leq2s_j.
\end{equation}

Let $\alpha_\mu,\beta_\mu>0$ be the sampled frame bounds.
For sufficiently large $R$, choose $j$ with $1\leq r_jR<e^h$.
Testing the upper bound on $e_{j,+1}$ gives
\[
 \beta_\mu
 \geq f_j^2e^{-2s_jR}\mu([0,R])
 \geq\frac{c_f}{R}e^{-2C_se^h}\mu([0,R]).
\]
By local finiteness, there is therefore a constant $K>0$ such that
$\mu([0,R])\leq KR$ for every $R\geq1$.
Integration by parts and $f_j^2\leq2s_j$ then give
\begin{equation}\label{nrm-ray-tail}
 f_j^2\int_{(Q,\infty)}e^{-2s_jt}\,d\mu(t)
 \leq K(1+2s_jQ)e^{-2s_jQ},
 \qquad Q\geq1.
\end{equation}

Choose $D>h-\log r_0$, with its final value fixed below.
For $R\geq1$, let $j_0$ be the first index satisfying
$r_{j_0}\leq e^{-D}/R$.
Then $j_0>1$ and
$e^{-D-h}/R<r_{j_0}\leq e^{-D}/R$.
For $L\geq2D+h$, put $m=\lfloor(L-2D)/h\rfloor$ and define
\(
 E=\operatorname{span}\{e_{j,\sigma}:
       j_0\leq j<j_0+m,\ \sigma=\pm1\}.
\)
If $(R,e^LR]$ contains fewer than $2m=\dim E$ sampling times,
there is a unit vector $x\in E$ such that $CA^tx=0$ at all
these times.

By Cauchy--Schwarz and \eqref{nrm-ray-comparisons},
\[
 \int_{[0,R]}|CA^tx|^2\,d\mu(t)
 \leq\mu([0,R])\sum_{j\geq j_0,\ \sigma=\pm1}f_j^2
 \leq\frac{2KC_fe^{-D}}{1-e^{-h}}.
\]
Set $Q=e^LR$. The choice of $m$ gives
$r_{j_0+m-1}Q>e^D$.
Using Cauchy--Schwarz, \eqref{nrm-ray-tail}, and
$(1+v)e^{-v}\leq2e^{-v/2}$ for $v\geq0$, we obtain
\[
 \int_{(Q,\infty)}|CA^tx|^2\,d\mu(t)
 \leq4K\sum_{r=0}^\infty
        \exp(-c_se^De^{rh})
 \leq\frac{4Ke^{-c_se^D}}{1-e^{-c_she^D}}.
\]
The last inequality uses $e^{rh}\geq1+rh$.

Choose $D$, independently of $R$ and $L$, so that the sum of
these two bounds is less than $\alpha_\mu$.
The sampled energy on $(R,Q]$ is zero, so this contradicts the
lower frame bound.
Thus the window contains at least
$2m\geq2L/h-4D/h-2$ sampling times.
For $L<2D+h$, the same lower bound is nonpositive and holds
automatically.
This proves the window estimate with
$C_{\mu,h,\theta,r_0}=4D/h+2$.
Taking $R=1$ and letting $L\to\infty$ gives $D(\mu)\geq2/h$.
\end{proof}

Combining the local counting bound in Lemma~\ref{nrm-two-ray-counting} with the integer construction in Theorem~\ref{nrm-relative} determines the leading coefficient of $\log\kappa$ in
$C_M^*(\kappa)$ and $C_{M,\mathbb Z}^*(\kappa)$ as
$\kappa\to\infty$.

\begin{theorem}\label{nrm-sharp}
For every fixed $M\geq0$,
\[
 C_M^*(\kappa)=C_{M,\mathbb Z}^*(\kappa)
 \sim\frac{\log\kappa}{\pi(\pi-2\arctan M)},
 \qquad \kappa\to\infty.
\]
More precisely, if $M>0$ and $a_*=\pi-2\arctan M$, then
\[
 \frac{\log\kappa-\log64}{\pi a_*}
 \leq C_M^*(\kappa)
 \leq\frac{\log\kappa+\log\log\kappa+O_M(1)}{\pi a_*}
\]
for all sufficiently large $\kappa$.
\end{theorem}

\begin{proof}
For $M=0$, the conclusion follows from
Theorem~\ref{thm:large-lower}, Corollary~\ref{cor:large-upper},
and Theorem~\ref{s4-integer}.

Suppose $M>0$, and put $\Theta=\arctan M$ and
$q=\log\kappa-\log64$.
Choose $\kappa$ sufficiently large that
$q>\max\{2,\pi/\Theta-2\}$.
For $\pi/(q+2)<\theta<\Theta$, set
$h=2\pi(\pi-2\theta)/q$.
Then \eqref{nrm-step-range} holds.
Choose $r_0$ for each $\theta$ according to
\eqref{nrm-model-sector-choice}.
Lemma~\ref{nrm-two-rays} gives a system with
$\operatorname{cond}(S_h)\leq\kappa$.
Rescaling $C$ gives $I\leq S_h\leq\kappa I$.

Every $\mu\in\mathcal U_{w,M}(\kappa)$ gives a frame for this system.
Lemma~\ref{nrm-two-ray-counting} therefore implies
$D(\mu)\geq2/h$.
Taking the infimum over $\mu$ gives
$C_M^*(\kappa)\geq q/[\pi(\pi-2\theta)]$.
Letting $\theta\uparrow\Theta$ proves
$C_M^*(\kappa)\geq(\log\kappa-\log64)/(\pi a_*)$.

For the upper bound, take $\varepsilon=1/2$ in
Theorem~\ref{nrm-relative}.
The resulting integer design satisfies
\eqref{nrm-upper-asymptotic}, so
\[
 C_{M,\mathbb Z}^*(\kappa)
 \leq\frac{\log\kappa+\log\log\kappa+O_M(1)}{\pi a_*}.
\]
Theorem~\ref{nrm-integer} gives equality of the two infima,
and the bounds prove the stated asymptotic formula.
\end{proof}

\begin{remark}
Fix an integer $q\geq1$, $M\geq0$, $\kappa\geq1$, and
$0<\varepsilon<1$.
There is a common positive measure $\nu$ supported on $\mathbb N_0$
such that
\[
 (1-\varepsilon)S\leq G_\nu\leq(1+\varepsilon)S,
 \qquad N_\nu(R)=O(\log R),
\]
for every normal system satisfying \eqref{nrm-input} whose power
$A^q$ satisfies the logarithmic sector condition
\eqref{nrm-logsector} with constant $M$.
In particular, this applies to selfadjoint systems with $q=2$
and $M=0$.
\end{remark}

\begin{proof}
Define $D_qx=(Cx,CAx,\ldots,CA^{q-1}x)$ with the direct sum norm.
The integer Gramian of $(A^q,D_q)$ is
\[
 \sum_{n\geq0}\sum_{r=0}^{q-1}
 (A^{qn+r})^*BA^{qn+r}=S.
\]
Let $\mu=\sum_{n\in T}w_n\delta_n$, with $T\subset\mathbb N_0$,
be the design supplied by Theorem~\ref{nrm-relative}
for this system, and set
\[
 \nu=\sum_{n\in T}w_n\sum_{r=0}^{q-1}\delta_{qn+r}.
\]
Since
$\sum_{r=0}^{q-1}\|CA^{qn+r}x\|^2=\|D_q(A^q)^nx\|^2$,
the sampled energies agree, so $\nu$ has the required relative bounds.
Moreover, $N_\nu(R)\leq qN_\mu(R/q)$ for every $R\geq0$.
Thus $N_\nu(R)=O(\log R)$ and $D(\nu)\leq qD(\mu)$.
For selfadjoint $A$, the operator $A^2$ is positive and satisfies
the sector condition with $M=0$.
\end{proof}

Theorem~\ref{nrm-sharp} determines the leading coefficient of
$\log\kappa$ in $C_M^*(\kappa)$ for each fixed $M$.
The exact value for finite $\kappa>1$ and the behavior as
$\kappa\downarrow1$ are not determined here.
We also do not extend to normal operators in logarithmic sectors
the nonattainment result of Theorem~\ref{s5-nonattainment}
or the strict gap under a prescribed sampled condition number
in Theorem~\ref{s6-strict-gap}.

At $\kappa=1$, the scalar reduction in Proposition~\ref{nrm-tight}
applies, but it does not resolve these questions for $\kappa>1$.

\section{Spectral geometry and optimal sampling growth}
\label{sec:spectral-geometry}

A fixed logarithmic sector gives logarithmic sampling and a sharp
leading coefficient as $\kappa\to\infty$, but it does not exhaust the spectral sets
admitting sparse common designs. We first characterize logarithmic
sampling through cumulative angular covering. The criterion is
sufficient for arbitrary Borel sets and necessary under relative radial
density. Under the latter hypothesis, we then determine optimal counting
growth with a regularity assumption on the covering numbers, and the
minimum upper counting exponent without that additional assumption.
Optimal leading coefficients for general fixed spectral sets remain
a separate question.
Throughout this section, $\mathbb D=\{\lambda\in\mathbb C:|\lambda|<1\}$
and $\Omega\subset\mathbb D$ is Borel.
An allowed normal system satisfies
\begin{equation}\label{geo:system}
 E_A(\mathbb C\setminus\Omega)=0,\qquad
 B=C^*C,\qquad
 S=\sum_{n\ge0}(A^n)^*BA^n,\qquad I\le S\le\kappa I.
\end{equation}
The condition on the spectral measure permits the closed spectrum to have
limit points on the unit circle, provided they carry no spectral mass.
Normalization loses no generality: replacing $C$ by $C/\sqrt a$
reduces $aI\le S\le bI$, $b/a\le\kappa$, to \eqref{geo:system}.
A common design has frame bounds uniform over all such systems,
including all dimensions, spectral multiplicities and observation
ranks. Its support and positive weights may depend on
$\Omega,\kappa$ and a prescribed accuracy.
Real powers have the convention of Section~\ref{nrm-section};
all sufficient constructions below use integer times.

Put $s_j=2^{-j}$ and
\begin{equation}\label{geo:annuli}
 \Omega_j=\{\lambda\in\Omega\setminus\{0\}:
              s_j\le-\log|\lambda|<2s_j\},\qquad j\ge1.
\end{equation}
Let $m_j(\Omega)$ be the least number of circular arcs of length
$s_j/10$ covering the arguments of $\Omega_j$; it is zero for an
empty band. The finitely deep part of the disk is treated separately
and does not affect the asymptotic criterion
\begin{equation}\label{geo:cover-condition}
 \sum_{j=1}^Jm_j(\Omega)=O(J)\qquad J\to\infty.
\end{equation}
Changing the fixed constant in the arc length changes these covering numbers
by at most a fixed factor. The arguments are considered on the circle,
so no choice of branch enters their covering numbers.

Put $\mathcal A_\Omega(J)=1+\sum_{j=1}^Jm_j(\Omega)$ for
$J\in\mathbb N_0$. For $\kappa\ge1$, let
$\mathcal U_{w,\Omega}(\kappa)$ be the class of positive atomic
measures $\mu$ on $[0,\infty)$ with locally finite support for
which there are constants $0<\alpha\le\beta<\infty$ such that
$\alpha I\le G_\mu\le\beta I$ for every normal system satisfying
\eqref{geo:system}. Denote its subclass supported on $\mathbb N_0$
by $\mathcal U^{\mathbb Z}_{w,\Omega}(\kappa)$.

\subsection{Angular covering and quantitative sampling bounds}
\label{geo:covering-subsection}

We show that the cumulative angular covering condition
\eqref{geo:cover-condition} gives common sampling at integer times,
with $N_\mu(R)=O(\log R)$ and arbitrarily small relative error.
The spectral arguments may vary with the radial scale, and the
covering numbers $m_j(\Omega)$ need not be uniformly bounded.

For a given finite frame, constructive weighted subsampling selects
a number of vectors proportional to the dimension while controlling
the frame bounds \cite{BSU2023}. Relative spectral approximation is
also available for finite sums of positive semidefinite matrices of
arbitrary rank \cite{dCHS2016}. Applying these results directly to an
observation Gramian generally gives a selection that depends on the
system. To obtain a common design, we represent the observations
using scalar functions determined by $\Omega$ and sparsify this
common family. We then transfer the estimates to every admissible
system through uniformly bounded maps from states to the
corresponding coefficients. These functions and maps are constructed
in the proof of Lemma~\ref{geo:quantitative-covering}.

We use the following consequence of the row sparsification theorem
of Cohen, Nelson and Woodruff, which gives a bound in terms of
stable rank \cite[Theorem~5 in the full version]{CNW2016}.
For a nonzero finite matrix $V$, its stable rank is
$\|V\|_{\mathrm F}^2/\|V\|^2$, where $\|\cdot\|_{\mathrm F}$
denotes the Frobenius norm.

\begin{lemma}\label{geo:ridge}
Let $\mathcal K$ be a complex Hilbert space, let $H$ be a bounded
positive injective operator on $\mathcal K$, and let
$v_1,\ldots,v_N\in\operatorname{Dom}(H^{-1/2})$. Set $F=\sum_{n=1}^N v_nv_n^*$ and
$d=\sum_{n=1}^N \|H^{-1/2}v_n\|^2$.
For every $0<\delta<1$, there are nonnegative weights
$w_1,\ldots,w_N$, at most $C_{\mathrm{sp}}(1+d)\delta^{-2}$
of which are nonzero, such that
\[
 -\delta(F+H)
 \le \sum_{n=1}^N w_nv_nv_n^*-F
 \le \delta(F+H),
\]
where $C_{\mathrm{sp}}$ is an absolute constant.
\end{lemma}

\begin{proof}
If $d=0$, take all weights to be zero. Otherwise, set
$q_n=H^{-1/2}v_n$, $T=\sum_{n=1}^N q_nq_n^*$, and
$u_n=(I+T)^{-1/2}q_n$.
Then $\sum_{n=1}^N u_nu_n^*=T(I+T)^{-1}\le I$, and
$\operatorname{tr}(T(I+T)^{-1})\le d$.

Let $E=\operatorname{span}\{u_1,\ldots,u_N\}$, and let $M=X+iY$
be the matrix with rows $u_n^*$ in an orthonormal basis of $E$,
where $X$ and $Y$ are real. The preceding bounds give
$\|M\|\le1$ and $\|M\|_{\mathrm F}^2\le d$.
Thus the real matrix $A=[X,-Y]$ satisfies
$\|A\|\le\|M\|\le1$ and
$\|A\|_{\mathrm F}^2=\|M\|_{\mathrm F}^2\le d$.
By \cite[Theorem~5 in the full version]{CNW2016}, applied with
$k=1+d$ and accuracy $\delta/2$, there is a real diagonal matrix
$D$ with at most $C_{\mathrm{sp}}(1+d)\delta^{-2}$ nonzero
entries such that
$\|A^{\mathsf T}(D^2-I)A\|\le\delta/2$.

For $z=a+ib$, we have
$A(a,b)=\operatorname{Re}(Mz)$ and
$A(b,-a)=\operatorname{Im}(Mz)$.
Adding the quadratic form estimates at $(a,b)$ and $(b,-a)$
gives $|z^*M^*(D^2-I)Mz|\le\delta\|z\|^2$.
Hence, with $w_n=D_{nn}^2$,
\[
 -\delta I
 \le \sum_{n=1}^N w_nu_nu_n^*
       -\sum_{n=1}^N u_nu_n^*
 \le \delta I
\]
on $E$. Both sums vanish on $E^\perp$, so the inequality holds
on all of $\mathcal K$.

Finally, the bounded operator $L=H^{1/2}(I+T)^{1/2}$ satisfies
$Lu_n=v_n$ and $LL^*=F+H$.
Multiplying the last inequality on the left by $L$ and on the
right by $L^*$ proves the claim.
\end{proof}

We first prove a quantitative upper bound for general covering
numbers. It requires no radial density assumption.
In \eqref{geo:counting-envelope}, the tail sums over $j>J+1$,
where $J=\lfloor\log_2R\rfloor$. These annuli contain only
spectral points with $-\log|\lambda|<1/R$.

\begin{lemma}
\label{geo:quantitative-covering}
Let $\Omega\subset\mathbb D$ be Borel. Suppose that, for some
$0<\theta<1$,
\begin{equation}\label{geo:weighted-cover-summability}
 \sum_{j\ge1}m_j(\Omega)2^{-(1-\theta)j}<\infty.
\end{equation}
For every $\kappa\ge1$ and $0<\varepsilon<1$, there is a sampling
measure $\mu$ supported on $\mathbb N_0$, depending only on
$\Omega,\kappa,\varepsilon,\theta$, such that
$(1-\varepsilon)S\le G_\mu\le(1+\varepsilon)S$
for every normal system satisfying \eqref{geo:system}.
For $R\ge1$ and $J=\lfloor\log_2R\rfloor$, it also satisfies
\begin{equation}\label{geo:counting-envelope}
 N_\mu(R)\le C_{\kappa,\varepsilon,\theta}
 \left(
 J+1+\mathcal A_\Omega(J+1)
 +\sum_{j>J+1}m_j(\Omega)
       2^{-(1-\theta)(j-J-1)}
 \right).
\end{equation}
\end{lemma}

\begin{proof}
Write $m_j=m_j(\Omega)$ and set $h_j=2s_j=2^{1-j}$.
Choose a covering of the arguments of each $\Omega_j$ by $m_j$
arcs and partition $\Omega_j$ into at most $m_j$ disjoint Borel
pieces whose arguments lie in the corresponding arcs.
These choices depend only on $\Omega$.
Let $\zeta_{jp}\in\partial\mathbb D$ be the corresponding arc
centers.

Fix a system satisfying \eqref{geo:system}, with state space $H$
and observation operator $C:H\to Y$.
Write $P_{jp}$ for the spectral projections of the chosen pieces.
On each piece, choose Borel functions $r$ and $\phi$ such that
\(
 \lambda=\zeta_{jp}e^{-r+i\phi},
 h_j/2\le r<h_j, |\phi|\le h_j/40.
\)
Let $P_{\mathrm f}$ be the spectral projection of
$\{\lambda\in\Omega:|\lambda|\le e^{-1}\}$.
Then $\|AP_{\mathrm f}\|\le e^{-1}$.

The identity $S-A^*SA=C^*C$, compressed to $P_{jp}H$, gives
\[
 \|CP_{jp}\|^2
 \le 2\kappa
      \bigl\|I-\overline{\zeta_{jp}}A|_{P_{jp}H}\bigr\|
 \le \frac{11}{5}\kappa h_j,
\]
where the last inequality follows from
$|1-\overline{\zeta_{jp}}\lambda|
 \le r+|\phi|\le(11/10)h_j$.
By spectral calculus, the function
$\lambda\mapsto r-3h_j/4-i\phi$
defines a bounded normal operator $Z_{jp}$ on $P_{jp}H$.
It satisfies
\[
 \|Z_{jp}\|<\frac3{10}h_j,\qquad
 A^nP_{jp}
 =\zeta_{jp}^{\,n}e^{-(3/4)nh_j}e^{-nZ_{jp}}P_{jp},
 \qquad n\ge1.
\]

For integers $n\ge1$ and $\ell\ge0$, define the scalar functions
$a_{jp,\ell}$ and coefficients $y_{jp,\ell}(x)\in Y$ by
\begin{align*}
 a_{jp,\ell}(n)
 &=\sqrt{h_j}\,\zeta_{jp}^{\,n}e^{-(3/4)nh_j}
   \frac{(nh_j/2)^\ell}{\ell!},
 \\
 y_{jp,\ell}(x)
 &=h_j^{-1/2}C
   \left(-\frac{2Z_{jp}}{h_j}\right)^\ell P_{jp}x,
 \qquad x\in H.
\end{align*}
The preceding estimates imply
\[
 \sum_{\ell\ge0}\|y_{jp,\ell}(x)\|^2
 \le \frac{(11/5)\kappa}{1-(3/5)^2}\|P_{jp}x\|^2
 <4\kappa\|P_{jp}x\|^2.
\]
For the remaining spectral subspace, put $\rho=e^{-1/2}$ and define
$a_{\mathrm f,\ell}(n)=\rho^n{\bf1}_{\{n=\ell\}}$ and
$y_{\mathrm f,\ell}(x)=\rho^{-\ell}CA^\ell P_{\mathrm f}x$
for $\ell\ge1$.
Since $\|C\|^2\le\kappa$ and $\|AP_{\mathrm f}\|\le e^{-1}$,
\[
 \sum_{\ell\ge1}\|y_{\mathrm f,\ell}(x)\|^2
 \le \frac{\kappa}{1-e^{-1}}\|P_{\mathrm f}x\|^2
 <2\kappa\|P_{\mathrm f}x\|^2.
\]

Let $\mathcal I$ be the set of indices $(j,p,\ell)$ and
$(\mathrm f,\ell)$ introduced above, and set
$\mathcal K=\ell^2(\mathcal I)$.
The coefficient map $V_{A,C}:H\to\mathcal K\otimes Y$, defined by
$V_{A,C}x=(y_\alpha(x))_{\alpha\in\mathcal I}$, satisfies
\begin{equation}\label{geo:coefficient-map}
 \|V_{A,C}\|^2\le4\kappa
\end{equation}
by orthogonality of the spectral projections.

The scalar functions satisfy
\begin{equation}\label{geo:feature-decay}
 \sum_{\ell\ge0}|a_{jp,\ell}(n)|^2
 \le h_je^{-(3/2)nh_j}e^{nh_j}
 =h_je^{-nh_j/2}.
\end{equation}
Assumption \eqref{geo:weighted-cover-summability} gives
$\sum_{j\ge1}m_jh_j<\infty$.
Thus $(a_\alpha(n))_{\alpha\in\mathcal I}\in\mathcal K$
for every $n\ge1$.
Together with \eqref{geo:coefficient-map}, the Cauchy--Schwarz
inequality shows that the series in
\begin{equation*}
 CA^nx=\sum_{\alpha\in\mathcal I}a_\alpha(n)y_\alpha(x),
 \qquad n\ge1,
\end{equation*}
converges absolutely in $Y$.
The identity follows by expanding $e^{-nZ_{jp}}$, first on
finitely many spectral pieces and then passing to the limit.

Set $v_n=(\overline{a_\alpha(n)})_{\alpha\in\mathcal I}\in\mathcal K$ for $n\ge1$.
For $k\ge0$, set $I_k=\{2^k,\ldots,2^{k+1}-1\}$ and
$F_k=\sum_{n\in I_k}v_nv_n^*$.
Let $H_k^{(\theta)}$ be the diagonal operator on $\mathcal K$
with entry $2^{-\theta|k-j+1|}$ at each coordinate $(j,p,\ell)$
and entry $2^{-k}$ at each coordinate $(\mathrm f,\ell)$.
These operators are bounded, positive and injective.
Summing their entries gives
\begin{equation*}
 \sum_{k\ge0}H_k^{(\theta)}\le b_\theta I,\qquad
 b_\theta=1+\frac{2}{2^\theta-1}.
\end{equation*}

By \eqref{geo:feature-decay},
\begin{equation*}
 \sum_{n\in I_k}\|(H_k^{(\theta)})^{-1/2}v_n\|^2
 \le E_k^{(\theta)}
 :=\sum_{j\ge1}m_jc_{k-j+1}^{(\theta)}+e_k,
\end{equation*}
where $c_d^{(\theta)}=2^{d+\theta|d|}e^{-2^{d-1}}$
for $d\in\mathbb Z$, and $e_k=2^k\sum_{n\in I_k}e^{-n}$.
The sequence $c^{(\theta)}=(c_d^{(\theta)})_{d\in\mathbb Z}$
belongs to $\ell^1(\mathbb Z)$, and
$c_d^{(\theta)}\le2^{(1-\theta)d}$ for $d\le0$.
Consequently, \eqref{geo:weighted-cover-summability} gives
$E_k^{(\theta)}<\infty$.
In particular,
$v_n\in\operatorname{Dom}((H_k^{(\theta)})^{-1/2})$
for every $n\in I_k$.

Set $\delta=\varepsilon/(1+4\kappa b_\theta)$.
Applying Lemma~\ref{geo:ridge} to each finite family
$\{v_n:n\in I_k\}$ gives nonnegative weights $w_{k,n}$,
at most $C_{\mathrm{sp}}\delta^{-2}(1+E_k^{(\theta)})$
of which are nonzero, such that
\begin{equation*}
 -\delta(F_k+H_k^{(\theta)})
 \le\widehat F_k-F_k
 \le\delta(F_k+H_k^{(\theta)}),
\end{equation*}
where $\widehat F_k=\sum_{n\in I_k}w_{k,n}v_nv_n^*$.
Define
$\mu=\delta_0+\sum_{k\ge0}\sum_{n\in I_k}w_{k,n}\delta_n$.

Tensor these inequalities with $I_Y$, evaluate at $V_{A,C}x$,
and sum over finitely many blocks.
The sums involving $F_k$ are bounded by
$\langle(S-B)x,x\rangle$, and those involving $H_k^{(\theta)}$
by $4\kappa b_\theta\|x\|^2$.
The sampled Gramian partial sums are therefore increasing and
uniformly bounded, so they converge strongly.
Passing to the limit gives
\[
 |\langle(G_\mu-S)x,x\rangle|
 \le\delta\langle(S-B)x,x\rangle
      +4\delta\kappa b_\theta\|x\|^2
 \le\varepsilon\langle Sx,x\rangle.
\]
The functions and weights are chosen independently of the system,
so the same measure satisfies this estimate for every system
in \eqref{geo:system}.

For $R\ge1$, set $J=\lfloor\log_2R\rfloor$.
The number of selected times satisfies
$N_\mu(R)\le1+C_{\mathrm{sp}}\delta^{-2}
\sum_{k=0}^J(1+E_k^{(\theta)})$.
For $j\le J+1$, we have
$\sum_{k=0}^Jc_{k-j+1}^{(\theta)}
\le\|c^{(\theta)}\|_{\ell^1}$.
For $j>J+1$, the bound for negative indices gives
\[
 \sum_{k=0}^Jc_{k-j+1}^{(\theta)}
 \le
 \frac{2^{-(1-\theta)(j-J-1)}}{1-2^{-(1-\theta)}}.
\]
Summing these estimates, using
$\sum_{j=1}^{J+1}m_j=\mathcal A_\Omega(J+1)-1$
and $\sum_{k\ge0}e_k<\infty$, proves
\eqref{geo:counting-envelope}.
\end{proof}

\begin{theorem}\label{geo:sufficient}
Let $\Omega\subset\mathbb D$ be Borel, and let $m_j=m_j(\Omega)$
be the angular covering numbers defined after \eqref{geo:annuli}.
Suppose that $\sum_{j=1}^Jm_j\le K_\Omega(J+1)$ for all $J\ge1$
and some finite $K_\Omega$.
For every $\kappa\ge1$ and $0<\varepsilon<1$, there is a positive
measure $\mu$ supported on $\mathbb N_0$, depending only on
$\Omega,\kappa,\varepsilon$, such that
\(
 (1-\varepsilon)S\le G_\mu\le(1+\varepsilon)S
\)
for every normal system satisfying \eqref{geo:system}.
Moreover, $N_\mu(R)=O_{K_\Omega,\kappa,\varepsilon}(\log R)$
as $R\to\infty$.
\end{theorem}

\begin{proof}
The assumed bound gives $m_j\le K_\Omega(j+1)$, so
\eqref{geo:weighted-cover-summability} holds with $\theta=1/2$.
Apply Lemma~\ref{geo:quantitative-covering} with this value of
$\theta$. It remains to estimate its counting bound.
For $J=\lfloor\log_2R\rfloor$, we have
$\mathcal A_\Omega(J+1)=O_{K_\Omega}(J+1)$ and
\[
 \sum_{j>J+1}m_j2^{-(j-J-1)/2}
 \le K_\Omega\sum_{\ell\ge1}(J+\ell+2)2^{-\ell/2}
 =O_{K_\Omega}(J+1).
\]
Substitution into \eqref{geo:counting-envelope} proves the claim.
\end{proof}

\subsection{Necessity under relative radial density}

For a Borel set $\Omega\subset\mathbb D$, set
$\mathcal R_\Omega
=\{-\log|\lambda|:\lambda\in\Omega\setminus\{0\}\}$.
We say that $\Omega$ is \emph{radially relatively dense near the
unit circle} if there are $c\in(0,1)$ and $s_*>0$ such that
\begin{equation}\label{geo:radial-density}
 \mathcal R_\Omega\cap[cs,s]\ne\varnothing,
 \qquad 0<s<s_*.
\end{equation}
Retain the annuli $\Omega_j$ from \eqref{geo:annuli} and their
angular covering numbers $m_j=m_j(\Omega)$.
Changing the initial index or multiplying the arc length by a
fixed positive constant does not affect the cumulative condition
\eqref{geo:cover-condition}.

Under \eqref{geo:radial-density}, we will show that common
logarithmic sampling for some fixed upper bound $\kappa>1$ on
the original frame condition number is equivalent to
\eqref{geo:cover-condition}.
Both are equivalent to common logarithmic sampling at nonnegative
integer times with arbitrarily small relative error for every
finite $\kappa\ge1$; see Theorem~\ref{geo:classification}.
The necessity proof constructs finite normal systems from spectral
points with separated arguments in each annulus, while keeping
$I\le S\le\kappa I$; see Lemma~\ref{geo:angular-packet}.
We first establish the matrix estimate needed for this construction.

\begin{lemma}\label{geo:angular-matrix}
Suppose that $\theta_1,\ldots,\theta_N\in\mathbb R$ satisfy
$\min_{k\in\mathbb Z}|\theta_i-\theta_j+2\pi k|\ge\delta$
for $i\ne j$, where $0<\delta\le\pi$.
The Hermitian matrix $H=(H_{ij})_{i,j=1}^N$ defined by
\[
 H_{ij}=
 \begin{cases}
 (1-e^{i(\theta_j-\theta_i)})^{-1},&i\ne j,\\
 0,&i=j
 \end{cases}
\]
satisfies $\|H\|\le30/\delta$.
\end{lemma}

\begin{proof}
The case $N=1$ is immediate.
Write $\mathbb T=\mathbb R/(2\pi\mathbb Z)$ and define the Hilbert
transform on $L^2(\mathbb T,d\theta)$ by
$\mathcal H(e^{in\theta})=-i\operatorname{sgn}(n)e^{in\theta}$,
where $\operatorname{sgn}(0)=0$.
Parseval's identity gives $\|\mathcal H\|\le1$.

For $0<r<1$, define
$\mathcal H_r(e^{in\theta})
=-ir^{|n|}\operatorname{sgn}(n)e^{in\theta}$.
Its convolution kernel, with normalization $1/(2\pi)$, is
\[
 K_r(t)=2\sum_{n\ge1}r^n\sin(nt)
       =\frac{2r\sin t}{1-2r\cos t+r^2}.
\]
As $r\uparrow1$, the operators $\mathcal H_r$ converge strongly
to $\mathcal H$, and $K_r(t)$ converges uniformly to $\cot(t/2)$
on sets at positive distance from $2\pi\mathbb Z$.

Let $I_i$ be the arc of length $\ell=\delta/2$ centered at
$\theta_i$, and set $u_i=\ell^{-1/2}{\bf1}_{I_i}$.
The arcs are disjoint, so the vectors $u_i$ are orthonormal.
Thus the matrix $M$ with entries
$M_{ij}=\langle\mathcal H u_j,u_i\rangle$ satisfies $\|M\|\le1$.
Since each $K_r$ is odd, $M_{ii}=0$.
For $i\ne j$, the kernel convergence gives
\[
 M_{ij}=\frac1{2\pi\ell}
        \int_{I_i}\int_{I_j}\cot\frac{x-y}{2}\,dy\,dx.
\]

Set $Q_{ij}=\cot((\theta_i-\theta_j)/2)$ for $i\ne j$
and $Q_{ii}=0$.
Let $d_{ij}\in[\delta,\pi]$ denote the circular distance between
$\theta_i$ and $\theta_j$.
For $x\in I_i$ and $y\in I_j$, choose representatives such that
$|(x-y)-(\theta_i-\theta_j)|\le\ell$.
Along this interval, the distance from $2\pi\mathbb Z$ is at least
$d_{ij}-\ell\ge d_{ij}/2$.
The derivative of $t\mapsto\cot(t/2)$ therefore has magnitude
at most $2\pi^2/d_{ij}^2$.
The mean value theorem yields
$\left|M_{ij}-(\ell/(2\pi))Q_{ij}\right|
\le\pi\ell^2/d_{ij}^2$.
Ordering the centers on each side of $\theta_i$ gives
$\sum_{j\ne i}d_{ij}^{-2}
\le(2/\delta^2)\sum_{n\ge1}n^{-2}
=\pi^2/(3\delta^2)$.
The absolute row and column sums of
$M-(\ell/(2\pi))Q$ are therefore at most $\pi^3/12$.
Schur's test gives
$\|M-(\ell/(2\pi))Q\|\le\pi^3/12$, and hence
\[
 \|Q\|
 \le\frac{2\pi}{\ell}(1+\pi^3/12)
 =\frac{4\pi}{\delta}(1+\pi^3/12).
\]

Finally, the identity
$(1-e^{it})^{-1}=1/2+(i/2)\cot(t/2)$ gives
$H=(1/2)(\mathbf1\mathbf1^*-I)-(i/2)Q$,
where $\mathbf1=(1,\ldots,1)^{\mathsf T}$.
Since $N\delta\le2\pi$,
\[
 \|H\|
 \le\frac{N-1}{2}+\frac12\|Q\|
 \le\frac{3\pi+\pi^4/6}{\delta}
 <\frac{30}{\delta}.
\]
\end{proof}

\begin{lemma}\label{geo:angular-packet}
Fix $\kappa>1$, set $\rho=(\kappa-1)/(\kappa+1)$, and let
$0<s\le\pi\rho/1000$.
Let $\Lambda\subset\{\lambda:e^{-2s}\le|\lambda|\le e^{-s}\}$,
and let $m$ be the least number of circular arcs of length $s/10$
covering its arguments.
Then there are distinct points $\lambda_1,\ldots,\lambda_N\in\Lambda$
and a linear map $C:\mathbb C^N\to\mathbb C$ such that, with
$A=\operatorname{diag}(\lambda_1,\ldots,\lambda_N)$, the Gramian
$S=\sum_{n\ge0}(A^n)^*C^*CA^n$ satisfies $I\le S\le\kappa I$ and
\begin{equation}\label{geo:packet-size}
 N\ge\frac{\rho}{20001}m,\qquad
 |Ce_i|^2\le4\kappa s,\qquad 1\le i\le N,
\end{equation}
where $e_1,\ldots,e_N$ are the standard coordinate vectors.
\end{lemma}

\begin{proof}
If $\Lambda=\varnothing$, take $N=0$ and the system on $\mathbb C^0$.
Otherwise, set $\delta=1000s/\rho\le\pi$ and choose a maximal
set of arguments $\theta_1,\ldots,\theta_N$ with pairwise circular
distance at least $\delta$.
This set is finite, and the arcs of radius $\delta$ centered at
these angles cover all arguments of $\Lambda$.
Each such arc can be covered by at most
$20\delta/s+1\le20001/\rho$ arcs of length $s/10$.
Thus $m\le20001N/\rho$, proving the bound for $N$ in
\eqref{geo:packet-size}.

Choose corresponding points
$\lambda_i=e^{-r_i+i\theta_i}\in\Lambda$, with $s\le r_i\le2s$.
Set $A=\operatorname{diag}(\lambda_1,\ldots,\lambda_N)$ and define
$C_0:\mathbb C^N\to\mathbb C$ by
$C_0e_i=f_i=(1-e^{-2r_i})^{1/2}$.
The corresponding Gramian $S_0$ has entries
\[
 (S_0)_{ij}
 =\frac{f_if_j}
 {1-e^{-(r_i+r_j)}e^{i(\theta_j-\theta_i)}},
\]
and $(S_0)_{ii}=1$.
Let $F=\operatorname{diag}(f_1,\ldots,f_N)$ and let $H$ be the
matrix in Lemma~\ref{geo:angular-matrix}.
Write $S_0-I=FHF+E$, where $E_{ii}=0$.
Since $f_i^2\le4s$, that lemma gives
$\|FHF\|\le120s/\delta$.

For $i\ne j$, set $q=e^{-(r_i+r_j)}$ and
$z=e^{i(\theta_j-\theta_i)}$.
Using $|1-qz|^2=(1-q)^2+q|1-z|^2$, we obtain
\[
 |E_{ij}|
 =\frac{f_if_j(1-q)}{|1-qz|\,|1-z|}
 \le\frac{16e^2s^2}{|1-z|^2},
\]
where we used $q\ge e^{-4}$, $1-q\le4s$ and $f_if_j\le4s$.
Let $d_{ij}$ be the circular distance between $\theta_i$ and
$\theta_j$.
Then $|1-z|\ge2d_{ij}/\pi$ and
$\sum_{j\ne i}d_{ij}^{-2}\le\pi^2/(3\delta^2)$,
as in the proof of Lemma~\ref{geo:angular-matrix}.
Schur's test therefore gives
\(
 \|E\|
 \le4e^2\pi^4s^2/(3\delta^2)
 <1000s^2/\delta^2.
\)
Consequently,
\[
 \|S_0-I\|
 \le120\frac{s}{\delta}+1000\frac{s^2}{\delta^2}
 =0.12\rho+0.001\rho^2
 <\rho.
\]

Set $C=(1-\rho)^{-1/2}C_0$.
Its Gramian is $S=(1-\rho)^{-1}S_0$, so
$I\le S\le((1+\rho)/(1-\rho))I=\kappa I$.
Finally, $|Ce_i|^2\le4s/(1-\rho)\le4\kappa s$,
which proves the remaining bound.
\end{proof}

To obtain a counting bound from these systems, we first use
relative radial density to control their sampled energy outside
a suitable time interval.

\begin{lemma}\label{geo:packet-window}
Suppose that $\Omega$ satisfies \eqref{geo:radial-density},
and fix $\kappa>1$.
Let $\mu$ be a sampling measure such that
$\alpha I\le G_\mu\le\beta I$ for every normal system satisfying
\eqref{geo:system}, with fixed $0<\alpha\le\beta<\infty$.
Then there are constants $0<a<b<\infty$ and $s_1>0$,
independent of $N$, such that every system of dimension $N$
constructed in Lemma~\ref{geo:angular-packet}, with spectrum
in $\Omega$ and $0<s<s_1$, satisfies
\(
 \#\bigl(\operatorname{supp}\mu\cap[a/s,b/s]\bigr)\ge N/2.
\)
\end{lemma}

\begin{proof}
Write $M_\mu(R)=\mu([0,R])$.
For $r\in\mathcal R_\Omega$, choose $\lambda\in\Omega$ with
$|\lambda|=e^{-r}$.
The scalar system $A=\lambda$, $C=(1-e^{-2r})^{1/2}$ has $S=1$.
The common upper bound therefore gives
$(1-e^{-2r})\int_{[0,\infty)}e^{-2rt}\,d\mu(t)\le\beta$.
For sufficiently large $R$, condition \eqref{geo:radial-density}
allows us to choose $r\in[c/R,1/R]$.
Then $r\le1/2$, $1-e^{-2r}\ge r$ and $e^{-2rR}\ge e^{-2}$,
so $M_\mu(R)\le\beta e^2R/c$.
By local finiteness, there is $K>0$ such that
$M_\mu(R)\le K(1+R)$ for all $R\ge0$.

Fix a system from Lemma~\ref{geo:angular-packet}, with
$\lambda_i=e^{-r_i+i\theta_i}$ and $r_i\in[s,2s]$.
For every real $t\ge0$,
$|CA^te_i|^2=|Ce_i|^2e^{-2r_it}$.
Thus \eqref{geo:packet-size} and the mass bound give
\[
 \int_{[0,a/s)}|Ce_i|^2e^{-2r_it}\,d\mu(t)
 \le4\kappa sM_\mu(a/s)
 \le4\kappa K(a+s).
\]
For the other tail, use $r_i\ge s$ and integrate by parts.
The lower endpoint term is nonpositive, so
\[
 \int_{(b/s,\infty)}|Ce_i|^2e^{-2r_it}\,d\mu(t)
 \le8\kappa s^2
       \int_{b/s}^{\infty}e^{-2st}M_\mu(t)\,dt
 \le4\kappa K(b+s+\tfrac12)e^{-2b}.
\]

Choose $0<a\le\min\{1,\alpha/(64\kappa K)\}$ and then $b>1$
such that $4\kappa K(b+3/2)e^{-2b}\le\alpha/8$.
Set $s_1=a$.
For $0<s<s_1$, each of the two tail bounds is at most $\alpha/8$.
Let $G_{\mathrm{out}}$ be the Gramian obtained by restricting
$\mu$ to $[0,a/s)\cup(b/s,\infty)$.
Summing over the coordinate vectors gives
$\operatorname{Tr}G_{\mathrm{out}}\le\alpha N/4$.

If $[a/s,b/s]$ contained fewer than $N/2$ sampling times,
the scalar functionals $CA^t$ at those times would have a common
kernel $V$ with $\dim V>N/2$.
On $V$, the quadratic forms of $G_\mu$ and $G_{\mathrm{out}}$
agree.
Positivity and the lower frame bound therefore imply
\(
 \alpha\dim V
 \le\operatorname{Tr}G_{\mathrm{out}}
 \le\alpha N/4,
\)
a contradiction.
\end{proof}

Combining Theorem~\ref{geo:sufficient} with
Lemmas~\ref{geo:angular-packet} and~\ref{geo:packet-window}
gives the following characterization.

\begin{theorem}
\label{geo:classification}
Suppose that the Borel set $\Omega\subset\mathbb D$ satisfies
\eqref{geo:radial-density}. The following are equivalent.
\begin{enumerate}
\item For some $\kappa>1$, there exist a sampling measure $\mu$
on $[0,\infty)$ and constants $0<\alpha\le\beta<\infty$ such that
$N_\mu(R)=O(\log R)$ and $\alpha I\le G_\mu\le\beta I$
for every normal system satisfying \eqref{geo:system}.

\item The angular covering numbers $m_j=m_j(\Omega)$, defined
after \eqref{geo:annuli}, satisfy
\begin{equation}\label{geo:cover-criterion}
 \sum_{j=1}^Jm_j=O(J),\qquad J\to\infty.
\end{equation}

\item For every finite $\kappa\ge1$ and $0<\varepsilon<1$,
there is a sampling measure $\mu$ supported on $\mathbb N_0$
with $N_\mu(R)=O(\log R)$ such that
$(1-\varepsilon)S\le G_\mu\le(1+\varepsilon)S$
for every normal system satisfying \eqref{geo:system}.
\end{enumerate}
Real powers are defined as in Section~\ref{nrm-section}.
\end{theorem}

\begin{proof}
Theorem~\ref{geo:sufficient} gives (2)$\Rightarrow$(3),
and (3)$\Rightarrow$(1) is immediate.
It remains to prove (1)$\Rightarrow$(2).

Assume (1), and set $T=\operatorname{supp}\mu$ and
$\rho=(\kappa-1)/(\kappa+1)$.
Apply Lemmas~\ref{geo:angular-packet} and
\ref{geo:packet-window} with $\Lambda=\Omega_j$ and $s=2^{-j}$.
There are constants $0<a<b<\infty$ and an integer $j_1\ge1$
such that
\(
 \#\bigl(T\cap[a2^j,b2^j]\bigr)
 \ge(\rho m_j)/{40002}\) for  \(j\ge j_1.\)
Every positive time belongs to at most
$L=\lceil\log_2(b/a)\rceil+2$ of these intervals.
Summing over $j_1\le j\le J$ therefore gives
\begin{equation}\label{geo:cumulative-lower}
 N_\mu(b2^J)
 \ge\frac{\rho}{40002L}\sum_{j=j_1}^Jm_j,
 \qquad J\ge j_1.
\end{equation}
The derivation of \eqref{geo:cumulative-lower} uses only the
common frame bounds and \eqref{geo:radial-density}.

By (1), $N_\mu(b2^J)=O(J)$, so
$\sum_{j=j_1}^Jm_j=O(J)$.
The earlier covering numbers contribute only a finite constant,
which proves \eqref{geo:cover-criterion}.
\end{proof}

\begin{remark}
The condition $\kappa>1$ in item~(1) of
Theorem~\ref{geo:classification} is essential.
When $\kappa=1$, we have $S=I$, and
Proposition~\ref{nrm-tight} reduces the sampling problem to the
spectral moduli $|\lambda|$.
Thus common logarithmic sampling can exist even when
\eqref{geo:cover-condition} fails.
The theorem characterizes the existence of common sampling
measures but does not determine the optimal constants in their
logarithmic counting bounds.
\end{remark}

\subsection{Counting growth and the optimal exponent}
\label{geo:growth-subsection}

We combine the cumulative lower bound \eqref{geo:cumulative-lower}
with Lemma~\ref{geo:quantitative-covering} to determine when the
upper and lower bounds have the same order.

\begin{theorem}
\label{geo:matching-growth}
Suppose that $\Omega$ satisfies \eqref{geo:radial-density}
and that there are $1\le K_g<\infty$, $0<\nu<1$ and
$J_0\in\mathbb N$ such that
\begin{equation}\label{geo:regular-growth}
 \mathcal A_\Omega(J+\ell)
 \le K_g2^{\nu\ell}\mathcal A_\Omega(J),
 \qquad J\ge J_0,\quad \ell\in\mathbb N_0.
\end{equation}
Fix $\kappa>1$.
For every $\sigma\in\mathcal U_{w,\Omega}(\kappa)$,
there are $c_\sigma>0$ and $1\le R_\sigma<\infty$ such that
\(
 N_\sigma(R)\ge c_\sigma
 \mathcal A_\Omega(\lfloor\log_2R\rfloor)\) for \(R\ge R_\sigma\).
For every $0<\varepsilon<1$, there is
$\mu\in\mathcal U^{\mathbb Z}_{w,\Omega}(\kappa)$ such that
$(1-\varepsilon)S\le G_\mu\le(1+\varepsilon)S$
for every normal system satisfying \eqref{geo:system}, and
\(
 N_\mu(R)\le C_{\Omega,\kappa,\varepsilon}
 \mathcal A_\Omega(\lfloor\log_2R\rfloor)
\)
for all sufficiently large $R$.
Thus $\mathcal A_\Omega(\lfloor\log_2R\rfloor)$ is the optimal
counting order for both real and integer sampling times.
\end{theorem}

\begin{proof}
Let $c$ be the constant in \eqref{geo:radial-density},
and choose $p\in\mathbb N$ with $2^{-p}\le c$.
For all sufficiently large $j$, relative radial density ensures
that at least one of $\Omega_j,\ldots,\Omega_{j+p}$ is nonempty.
Hence $\mathcal A_\Omega(J)\ge c_0(J+1)$ for some $c_0>0$
and all sufficiently large $J$.

Fix $\sigma\in\mathcal U_{w,\Omega}(\kappa)$.
By \eqref{geo:cumulative-lower}, there are $b,c_1>0$ and
$j_1\ge1$ such that
$N_\sigma(b2^L)\ge c_1\sum_{j=j_1}^Lm_j$
for all sufficiently large integers $L$.
Since $\mathcal A_\Omega(L)\to\infty$, the omitted finite sum
gives $\sum_{j=j_1}^Lm_j\ge\mathcal A_\Omega(L)/2$ for large $L$.
Choose an integer $q\ge\max\{0,\log_2b\}$ and set
$J=\lfloor\log_2R\rfloor$ and $L=J-q$.
For sufficiently large $R$, we have $L\ge J_0$ and $b2^L\le R$.
Condition \eqref{geo:regular-growth} gives
$\mathcal A_\Omega(J)\le K_g2^{\nu q}\mathcal A_\Omega(L)$.
Combining these inequalities proves the lower bound.

For the upper bound, choose $0<\theta<1-\nu$.
Taking $J=J_0$ in \eqref{geo:regular-growth} gives
$m_j\le\mathcal A_\Omega(j)=O(2^{\nu j})$.
Thus \eqref{geo:weighted-cover-summability} holds.
Apply Lemma~\ref{geo:quantitative-covering} to obtain a measure
with the stated relative approximation.
For $J\ge J_0$, condition \eqref{geo:regular-growth} bounds
the tail in \eqref{geo:counting-envelope} by
\[
 \sum_{\ell\ge1}m_{J+1+\ell}2^{-(1-\theta)\ell}
 \le
 \frac{K_g}{2^{1-\theta-\nu}-1}\,
 \mathcal A_\Omega(J+1).
\]
Finally,
$\mathcal A_\Omega(J+1)\le K_g2^\nu\mathcal A_\Omega(J)$
and $J+1\le c_0^{-1}\mathcal A_\Omega(J)$ for all large $J$.
Substituting these estimates into \eqref{geo:counting-envelope}
proves the upper bound.
\end{proof}

We next determine the optimal counting exponent without assuming
\eqref{geo:regular-growth}. Define
\begin{equation}\label{geo:counting-exponents}
 d_\Omega=\limsup_{J\to\infty}
          \frac{\log\mathcal A_\Omega(J)}{J\log2},
 \qquad
 p(\mu)=\limsup_{R\to\infty}
          \frac{\log(1+N_\mu(R))}{\log R}.
\end{equation}
The whole circle can be covered by $O(2^j)$ arcs of length
$2^{-j}/10$. Thus $m_j(\Omega)=O(2^j)$ and
$0\le d_\Omega\le1$.

\begin{theorem}
\label{geo:optimal-exponent}
Suppose that $\Omega$ satisfies \eqref{geo:radial-density},
and let $d_\Omega$ and $p(\mu)$ be as in
\eqref{geo:counting-exponents}.
For every fixed $\kappa>1$,
\begin{equation}\label{geo:exponent-optimum}
 \min_{\mu\in\mathcal U_{w,\Omega}(\kappa)}p(\mu)
 =\min_{\mu\in\mathcal U^{\mathbb Z}_{w,\Omega}(\kappa)}p(\mu)
 =d_\Omega.
\end{equation}
For every $0<\varepsilon<1$, there is
$\mu\in\mathcal U^{\mathbb Z}_{w,\Omega}(\kappa)$ attaining
both minima and satisfying
$(1-\varepsilon)S\le G_\mu\le(1+\varepsilon)S$
for every normal system satisfying \eqref{geo:system}.
\end{theorem}

\begin{proof}
Relative radial density gives $\mathcal A_\Omega(J)\to\infty$.
Thus \eqref{geo:cumulative-lower} and the definitions in
\eqref{geo:counting-exponents} imply $p(\mu)\ge d_\Omega$
for every $\mu\in\mathcal U_{w,\Omega}(\kappa)$.
A fixed multiplicative change in time does not affect this exponent.

Fix $0<\varepsilon<1$.
If $d_\Omega=1$, take $\mu=\sum_{n\ge0}\delta_n$.
Then $G_\mu=S$ and $p(\mu)=1$.

Suppose that $d_\Omega<1$, and fix $0<\theta<1-d_\Omega$.
For every $a$ with $d_\Omega<a<1-\theta$, the definition of
$d_\Omega$ gives a finite $K_a>0$ such that
$m_j(\Omega)\le\mathcal A_\Omega(j)\le K_a2^{aj}$
for all $j\ge1$.
In particular, \eqref{geo:weighted-cover-summability} holds.
Apply Lemma~\ref{geo:quantitative-covering} once, with the fixed
parameters $\theta,\kappa,\varepsilon$, to obtain $\mu$.

For each such $a$, we have
$\mathcal A_\Omega(J+1)\le K_a2^{a(J+1)}$ and
\[
 \sum_{\ell\ge1}m_{J+1+\ell}(\Omega)2^{-(1-\theta)\ell}
 \le
 \frac{K_a2^{a(J+1)}}{2^{1-\theta-a}-1}.
\]
Since $a>0$, we also have $J+1=O_a(2^{aJ})$.
Therefore \eqref{geo:counting-envelope} gives
$N_\mu(R)=O_a(R^a)$ for the same measure $\mu$ and every
$d_\Omega<a<1-\theta$.
Letting $a\downarrow d_\Omega$ gives $p(\mu)\le d_\Omega$.
The lower bound proves equality, so this measure attains both minima.
\end{proof}

For $d_\Omega<1$, attainment in \eqref{geo:exponent-optimum}
need not give $N_\mu(R)=O(R^{d_\Omega})$.
The value $p(\mu)=0$ does not imply $N_\mu(R)=O(\log R)$,
and $p(\mu)=1$ does not imply $N_\mu(R)\asymp R$.
Under \eqref{geo:regular-growth},
Theorem~\ref{geo:matching-growth} determines the counting order.
Neither theorem determines optimal leading coefficients
for a general spectral set.

\subsection{Consequences and limits of the covering criterion}
\label{geo:radial-gaps}

We apply the covering criterion to specific spectral sets and
examine the role of relative radial density.

\begin{corollary}
\label{geo:positive-completion}
Let $\Omega\subset\mathbb D$ be Borel, and set
$\Gamma=\Omega\cup[0,1)$.
The covering numbers of $\Omega$ satisfy
\eqref{geo:cover-criterion} if and only if there are $\kappa>1$
and $\mu\in\mathcal U_{w,\Gamma}(\kappa)$ with
$N_\mu(R)=O(\log R)$.
Under this condition, item~(3) of
Theorem~\ref{geo:classification} holds with $\Omega$
replaced by $\Gamma$.
\end{corollary}

\begin{proof}
Since $\mathcal R_\Gamma=(0,\infty)$, the set $\Gamma$ satisfies
\eqref{geo:radial-density}.
The added interval contributes only the argument zero, so
$m_j(\Omega)\le m_j(\Gamma)\le m_j(\Omega)+1$.
Thus $\Omega$ and $\Gamma$ satisfy the cumulative covering
condition simultaneously.
The result follows from Theorem~\ref{geo:classification}.
\end{proof}

\begin{corollary}\label{geo:curves}
Let $s_0>0$ and let $\varphi:(0,s_0]\to\mathbb R$ be continuous.
Set $\Omega_\varphi=\{e^{-s+i\varphi(s)}:0<s\le s_0\}$.
Choose $j_0\in\mathbb N$ with $2^{1-j_0}\le s_0$.
For $j\ge j_0$, define
$\omega_j=\sup_{s,t\in[2^{-j},2^{1-j}]}
|\varphi(s)-\varphi(t)|$
and $u_j=2^j\min\{\omega_j,1\}$.

There are $\kappa>1$ and
$\mu\in\mathcal U_{w,\Omega_\varphi}(\kappa)$ with
$N_\mu(R)=O(\log R)$ if and only if
\begin{equation}\label{geo:curve-criterion}
 \sum_{j=j_0}^Ju_j=O(J-j_0+1),\qquad J\to\infty.
\end{equation}
Under this condition, item~(3) of
Theorem~\ref{geo:classification} holds for $\Omega_\varphi$.
The condition also implies that $\varphi(s)$ converges to some
$\theta_0\in\mathbb R$, with
$|\varphi(s)-\theta_0|=O(s\log(1/s))$ as $s\downarrow0$.
\end{corollary}

\begin{proof}
The set $\Omega_\varphi$ is a countable union of compact curve
segments, hence Borel.
Since $\mathcal R_{\Omega_\varphi}=(0,s_0]$,
condition \eqref{geo:radial-density} holds.

On $[2^{-j},2^{1-j}]$, the image of $\varphi$ is an interval
of length $\omega_j$.
Reduction modulo $2\pi$ gives an arc of length $\omega_j$
when $\omega_j<2\pi$ and the whole circle otherwise.
The angular image of $[2^{-j},2^{1-j})$ has the same closure.
Covering these images by arcs of length $2^{-j}/10$ therefore gives
\(
 m_j(\Omega_\varphi)
 \asymp1+2^j\min\{\omega_j,1\}
 =1+u_j,
\)
with absolute constants.
The sampling assertions follow from
Theorem~\ref{geo:classification}.

Condition \eqref{geo:curve-criterion} gives $u_j=O(j)$.
Hence $\omega_j<1$ for all sufficiently large $j$, and
$\omega_j=2^{-j}u_j=O(j2^{-j})$.
In particular, $\sum_{j\ge j_0}\omega_j<\infty$.
The values $\varphi(2^{-j})$ therefore form a Cauchy sequence.
Let its limit be $\theta_0\in\mathbb R$.
For $s\in[2^{-j},2^{1-j}]$,
\[
 |\varphi(s)-\theta_0|
 \le\sum_{\ell\ge j}\omega_\ell
 =O(j2^{-j})
 =O(s\log(1/s)).
\]
\end{proof}

\begin{example}
Fix $\kappa>1$ and use the notation of
Corollary~\ref{geo:curves}.
Each counting order below is optimal in the sense of
Theorem~\ref{geo:matching-growth}: every common sampling measure
satisfies the corresponding lower bound, and a common integer
sampling measure attains the upper bound with any prescribed
relative Gramian error $0<\varepsilon<1$.

\emph{Power curves.}
Let $\varphi(s)=\theta_0+cs^\gamma$ on $(0,s_0]$, where
$\theta_0,c\in\mathbb R$, $c\ne0$ and $\gamma>0$.
The optimal counting order for $\Omega_\varphi$ is
\[
 \begin{cases}
 R^{1-\gamma},&0<\gamma<1,\\
 \log R,&\gamma\ge1.
 \end{cases}
\]
The oscillation on $[s,2s]$ is
$|c|(2^\gamma-1)s^\gamma$, so
$m_j(\Omega_\varphi)\asymp1+2^{j(1-\gamma)}$ for large $j$.
Hence $\mathcal A_{\Omega_\varphi}(J)\asymp2^{(1-\gamma)J}$
when $0<\gamma<1$, and
$\mathcal A_{\Omega_\varphi}(J)\asymp J$ when $\gamma\ge1$.
In both cases, \eqref{geo:regular-growth} holds, and
Theorem~\ref{geo:matching-growth} gives the stated order.

\smallskip
\noindent\emph{Logarithmic factors.}
Let $\varphi(s)=\theta_0+cs(\log(1/s))^p$ on $(0,s_0]$,
where $\theta_0,c\in\mathbb R$, $c\ne0$, $p>0$ and $0<s_0<1$.
The optimal counting order is $(\log R)^{p+1}$.
Indeed,
$\varphi'(s)=c(\log(1/s))^p(1-p/\log(1/s))$,
so the oscillation on $[s,2s]$ is comparable to
$s(\log(1/s))^p$ for small $s$.
Thus $m_j(\Omega_\varphi)\asymp j^p$ and
$\mathcal A_{\Omega_\varphi}(J)\asymp J^{p+1}$.
For each fixed $0<\nu<1$, the bound
$(1+\ell)^{p+1}\le C_{p,\nu}2^{\nu\ell}$ for
$\ell\in\mathbb N_0$ implies \eqref{geo:regular-growth}.
Theorem~\ref{geo:matching-growth} therefore applies.
In particular, $p=1$ gives the optimal order $(\log R)^2$.

\smallskip
\noindent\emph{Sparse angular excursions.}
Let $\chi:\mathbb R\to[0,1]$ be continuous, supported in
$[-1/4,1/4]$, with $\chi(0)=1$. Define
\[
 \varphi(s)=s\sum_{k\ge1}2^k
       \chi\bigl(\log(1/s)-2^k\bigr),
 \qquad 0<s\le1.
\]
The summand indexed by $k$ is supported where
$|\log(1/s)-2^k|\le1/4$ and is bounded by
$e^{1/4}2^ke^{-2^k}$.
The supports are disjoint and locally finite in $(0,1]$.
Consequently, $\varphi$ is continuous and tends to zero
as $s\downarrow0$.

Each support meets at most two dyadic intervals
$[2^{-j},2^{1-j}]$.
For these intervals, $j\log2=2^k+O(1)$ and
$u_j=2^j\min\{\omega_j,1\}=O(2^k)$.
Therefore,
\[
 \sum_{j=1}^J u_j
 \le C\sum_{2^k\le J\log2+C}2^k
 =O(J).
\]
The estimate $m_j(\Omega_\varphi)\asymp1+u_j$ from the proof
of Corollary~\ref{geo:curves} gives
$\mathcal A_{\Omega_\varphi}(J)\asymp J$.
Thus \eqref{geo:regular-growth} holds, and
Theorem~\ref{geo:matching-growth} gives optimal logarithmic
counting.

On the other hand,
$\varphi(e^{-2^k})/e^{-2^k}=2^k\to\infty$ as $k\to\infty$.
Since $\varphi(s)\to0$, no finite $M$ gives
$|\arg\lambda|\le M(-\log|\lambda|)$ for all
$\lambda\in\Omega_\varphi$ sufficiently close to $1$.
Thus this curve lies in no fixed logarithmic sector near $1$.
\end{example}

\begin{corollary}
Let $s_0>0$ and let $E\subset\partial\mathbb D$ be nonempty
and compact. Set
$\Omega_E=\{e^{-s}\zeta:0<s\le s_0,\ \zeta\in E\}$.
Write $\mathcal N(E,r)$ for the least number of circular arcs
of length $r$ covering $E$, and define its upper box dimension by
\[
 \overline{\dim}_{\mathrm B}E
 =\limsup_{r\downarrow0}
       \frac{\log\mathcal N(E,r)}{\log(1/r)}.
\]
For every fixed $\kappa>1$, the minimum of $p(\mu)$ over
$\mu\in\mathcal U_{w,\Omega_E}(\kappa)$ equals
$\overline{\dim}_{\mathrm B}E$.
For each $0<\varepsilon<1$, this minimum is attained by an
integer sampling measure satisfying
$(1-\varepsilon)S\le G_\mu\le(1+\varepsilon)S$
for every normal system satisfying \eqref{geo:system}
with $\Omega=\Omega_E$.

If $\mathcal N(E,r)\asymp r^{-d}$ as $r\downarrow0$ for some
$0<d<1$, the optimal counting order is $R^d$.
If $E$ is finite, the optimal order is $\log R$.
\end{corollary}

\begin{proof}
The set $\Omega_E$ is Borel, and
$\mathcal R_{\Omega_E}=(0,s_0]$, so
\eqref{geo:radial-density} holds.
Write $m_j=m_j(\Omega_E)$.
Each nonempty annulus has angular set $E$, and
$m_j=\mathcal N(E,2^{-j}/10)$ for all sufficiently large $j$.
Thus $m_j$ is nondecreasing, and
$m_J\le\mathcal A_{\Omega_E}(J)\le1+Jm_J$.
The factor $J$ does not affect the exponent.
Monotonicity of covering numbers also allows the upper box
dimension to be computed along $r=2^{-J}/10$.
Hence $d_{\Omega_E}=\overline{\dim}_{\mathrm B}E$, and
Theorem~\ref{geo:optimal-exponent} gives the minimum and its
attainment.

If $\mathcal N(E,r)\asymp r^{-d}$, then
$\mathcal A_{\Omega_E}(J)\asymp2^{dJ}$.
If $E$ is finite, then $\mathcal A_{\Omega_E}(J)\asymp J$.
Both estimates imply \eqref{geo:regular-growth}, so
Theorem~\ref{geo:matching-growth} gives the stated counting orders.
\end{proof}

The following example shows that, without relative radial density,
common logarithmic sampling need not imply
\eqref{geo:cover-condition}.

\begin{example}
\label{geo:gap-counterexample}
There is a countable set $\Omega_0\subset\mathbb D$ such that,
for every $1\le\kappa<\infty$, some
$\nu_\kappa\in\mathcal U^{\mathbb Z}_{w,\Omega_0}(\kappa)$
satisfies $D(\nu_\kappa)=0$, or equivalently
$N_{\nu_\kappa}(R)=o(\log R)$.
However, $\sum_{j=1}^Jm_j(\Omega_0)$ is not $O(J)$.
Thus the minimum counting coefficient for this spectral class
is zero and is attained for both real and integer sampling times.
\end{example}

The construction and its estimates are given in
Appendix~\ref{app:normal-radial-gaps}.
There, $\Omega_0=\bigcup_{k\ge1}\Omega_k$, where each
$\Omega_k$ is finite and lies on the circle
$|\lambda|=e^{-s_k}$, with $s_k\downarrow0$.
Large gaps between the radial parameters allow the observation
block associated with $\Omega_k$ to begin at a time $T_k$
satisfying $s_kT_k\to\infty$.
Large positive weights compensate for the exponential decay
at these times.

Proposition~\ref{geo:radial-thickness} considers
$\Omega_+=\Omega_0\cup\{e^{-2s_k}:k\ge1\}$,
adding one positive spectral point for each $\Omega_k$.
For every $\kappa>1$ and
$\mu\in\mathcal U_{w,\Omega_+}(\kappa)$, it gives
$N_\mu(R_k)\ge c_\mu\sqrt{R_k}$ for some $c_\mu>0$
and a sequence $R_k\to\infty$.
Thus this enlargement rules out common logarithmic sampling
for both real and integer times.

These examples explain the role of radial gaps.
They do not give a necessary and sufficient condition for
common logarithmic sampling on arbitrary spectral sets.

\appendix

\section{Full universality and sampling at prime times}
\label{app:arithmetic-designs}

This appendix gives a sufficient condition for full universality,
as defined in Section~\ref{sec:introduction}. The same sampling
times and weights preserve the frame property for every finite
bound on the original frame condition number. The application
to prime times shows how weights affect this property on a fixed
support.

The proof combines an estimate on spectral subspaces near $1$
with uniform approximation by the powers $z^t$, $t\in T$.
It does not require $A$ and $B=C^*C$ to commute.

\begin{proposition}
\label{prop:muntz-universal}
Let $\mu=\sum_{t\in T}w_t\delta_t$ be a sampling design with
$0\in T$. If $M_\mu(R)=dR+o(R)$ for some $0<d<\infty$ and
$\sum_{t\in T\setminus\{0\}}1/t=\infty$, then
$\mu\in\mathcal U_w$.
\end{proposition}

\begin{proof}
Fix a system satisfying \eqref{eq:frame}. Local finiteness
and the mass asymptotic give
$K_+=\sup_{R>0}\mu((0,R])/R<\infty$.
By Proposition~\ref{prop:mass}, the analysis operator
$\mathcal T_\mu:H\to\ell^2(T;Y)$, defined by
$\mathcal T_\mu x=(\sqrt{w_t}\,CA^tx)_{t\in T}$, is bounded.
We prove that it is bounded below.

\smallskip
\noindent\emph{Estimate near the spectral point $1$.}
For $0<r<1$, let $P_r=E_A([r,1])$ and $H_r=P_rH$.
Define $A_r=A|_{H_r}$, $S_r=P_rS|_{H_r}$ and
$B_r=P_rB|_{H_r}$, where $S$ is the Gramian in
\eqref{eq:gramian}. Lemma~\ref{lem:stein} gives
$B_r=S_r-A_rS_rA_r$, so $\|B_r\|\leq2b(1-r)$.
Also, $\|-\log A_r\|\leq-\log r$.

For $x\in H_r$, set $F_x(t)=\|CA^tx\|^2$ for $t\geq0$.
The preceding bounds imply
$F_x(0)\leq2b(1-r)\|x\|^2$ and
$|F_x'(t)|\leq4b(1-r)(-\log r)\|x\|^2$ for $t\geq0$.
Put $D_\mu(t)=\mu((0,t])-dt=o(t)$.
Stieltjes integration by parts gives
\[
 \langle(G_\mu-dG_c)x,x\rangle
 =w_0F_x(0)-\int_0^\infty D_\mu(t)F_x'(t)\,dt,
\]
where $G_c$ is defined in \eqref{eq:continuous}.
The boundary term at infinity vanishes because
$D_\mu(t)=O(t)$ and $tF_x(t)\to0$, as shown in the proof
of Proposition~\ref{prop:mass}. At zero, local finiteness
gives $D_\mu(t)=-dt$ for all sufficiently small $t>0$.

Given $\eta>0$, choose $T_\eta$ such that
$|D_\mu(t)|\leq\eta t$ for $t\geq T_\eta$, and set
$J_\eta=\int_0^{T_\eta}|D_\mu(t)|\,dt$.
The local derivative bound and Lemma~\ref{lem:energy} yield
\[
 \|P_r(G_\mu-dG_c)P_r\|
 \leq2bw_0(1-r)
      +4b(1-r)(-\log r)J_\eta+\sqrt2\,b\eta.
\]
Moreover, \eqref{eq:continuous} and
$\|(A^s-I)P_r\|\leq1-r$ for $0\leq s\leq1$ imply
$\|P_r(G_c-S)P_r\|\leq2b(1-r)$.
Letting $r\uparrow1$ and then $\eta\downarrow0$, we obtain
\begin{equation}\label{eq:muntz-spectral-limit}
 \|P_r(G_\mu-dS)P_r\|\longrightarrow0,
 \qquad r\uparrow1.
\end{equation}
This estimate uses only the mass asymptotic and remains valid
when $d=0$. For the present case $d>0$, choose $r<1$ such that
$P_rG_\mu P_r\geq(da/2)P_r$.

\smallskip
\noindent\emph{Estimate on the remaining spectrum.}
Local finiteness allows the positive elements of $T$ to be
listed as an increasing sequence tending to infinity.
The reciprocal divergence assumption and the M\"untz theorem
imply that $\operatorname{span}\{z^t:t\in T\}$ is dense in
$C[0,1]$ in the uniform norm; see
\cite[Theorem~2.1]{BorweinErdelyi1996}.

Suppose that $\mathcal T_\mu$ is not bounded below.
Choose unit vectors $x_n$ with $\|\mathcal T_\mu x_n\|\to0$.
Since each $w_t>0$, we have $CA^tx_n\to0$ for every $t\in T$.
Uniform approximation and continuous functional calculus
therefore give $Cg(A)x_n\to0$ for every $g\in C[0,1]$.

Choose $r<r_1<1$ and $\chi\in C[0,1]$ such that
$0\leq\chi\leq1$, $\chi=1$ on $[0,r]$ and
$\chi=0$ on $[r_1,1]$. Set $u_n=\chi(A)x_n$.
For each integer $k\geq0$, applying the preceding conclusion
to $g(z)=z^k\chi(z)$ gives $CA^ku_n\to0$.
Since $\|u_n\|\leq1$ and $u_n\in E_A([0,r_1])H$,
\[
 \sum_{k=N}^\infty\|CA^ku_n\|^2
 =\langle SA^Nu_n,A^Nu_n\rangle
 \leq b r_1^{2N},
 \qquad N\in\mathbb N_0.
\]
The original lower frame bound then gives
$\limsup_{n\to\infty}a\|u_n\|^2\leq b r_1^{2N}$
for every $N$. Letting $N\to\infty$, we obtain $u_n\to0$.

Now $v_n=x_n-u_n$ belongs to $P_rH$ and $\|v_n\|\to1$.
Boundedness of $\mathcal T_\mu$ gives
$\mathcal T_\mu v_n\to0$, contradicting
$P_rG_\mu P_r\geq(da/2)P_r$.
Thus sampling preserves the frame property for this system.

Uniform bounds over systems with fixed original frame bounds
$a,b$ follow from the orthogonal direct sum argument in
Section~\ref{sec:preliminaries}. Hence $\mu\in\mathcal U_w$.
\end{proof}

The reciprocal divergence condition is not necessary:
Theorem~\ref{s4-universal-design} constructs fully universal
designs for which the reciprocal sum converges.

\begin{example}
\label{ex:prime-sampling}
Let $\mathcal P$ be the set of prime numbers, and define
\(
 \mu_{\mathcal P}=\delta_0+\sum_{p\in\mathcal P}\delta_p,
 \
 \mu_{\log}=\delta_0+\sum_{p\in\mathcal P}(\log p)\delta_p.
\)
Then $\mu_{\log}\in\mathcal U_w$ and
$\mu_{\mathcal P}\notin\mathcal U_w(1)$.
For a system satisfying \eqref{eq:frame},
$G_{\mu_{\mathcal P}}$ has a positive lower bound
if and only if $\|A\|<1$.
\end{example}

\begin{proof}
Write $\pi(R)=\#\{p\in\mathcal P:p\leq R\}$, and let $p_n$
be the $n$th prime. The prime number theorem gives
$\pi(R)\sim R/\log R$ and $p_n\sim n\log n$; see
\cite[\S27.2(i)]{DLMF}. Partial summation yields
\[
 \sum_{p\leq R}\log p
 =\pi(R)\log R-\int_2^R\frac{\pi(t)}t\,dt
 \sim R,
 \qquad R\to\infty.
\]
Indeed, the first term is asymptotic to $R$, and the integral
is $o(R)$ because $\pi(t)/t\to0$.
Thus $M_{\mu_{\log}}(R)\sim R$.
Also, $p_n\sim n\log n$ implies $\sum_n1/p_n=\infty$.
Proposition~\ref{prop:muntz-universal} gives
$\mu_{\log}\in\mathcal U_w$.
Both designs have the same counting function,
$N_{\mu_{\mathcal P}}(R)=N_{\mu_{\log}}(R)
=1+\pi(R)\sim R/\log R$.

For the unweighted design, $G_{\mu_{\mathcal P}}\leq S$
because $\{0\}\cup\mathcal P\subset\mathbb N_0$.
Suppose first that $\|A\|=1$.
Since $\mu_{\mathcal P}((0,R])=\pi(R)=o(R)$,
the argument for \eqref{eq:muntz-spectral-limit}, with $d=0$,
gives $\|P_rG_{\mu_{\mathcal P}}P_r\|\to0$ as $r\uparrow1$.
Every $P_r$ with $0<r<1$ is nonzero, so
$G_{\mu_{\mathcal P}}$ has no positive lower bound.

Conversely, suppose $q=\|A\|<1$ and that the lower bound fails.
Choose unit vectors $x_n$ with
$\mathcal T_{\mu_{\mathcal P}}x_n\to0$.
The M\"untz approximation argument in the preceding proof
gives $CA^kx_n\to0$ for every fixed integer $k\geq0$.
For every integer $N\geq1$, the original frame inequality gives
\[
 a\leq\sum_{k=0}^{N-1}\|CA^kx_n\|^2
       +\langle SA^Nx_n,A^Nx_n\rangle
 \leq\sum_{k=0}^{N-1}\|CA^kx_n\|^2+bq^{2N}.
\]
Letting $n\to\infty$ and then $N\to\infty$ gives a
contradiction. This proves the equivalence.

To obtain a Parseval counterexample, take
$H=Y=\ell^2(\mathbb N)$ and define
$Ae_j=e^{-1/j}e_j$ on the standard basis, with
$C=(I-A^2)^{1/2}$.
Then $S=I$ and $\|A\|=1$.
The unweighted prime sampling family is therefore Bessel
but is not a frame, proving
$\mu_{\mathcal P}\notin\mathcal U_w(1)$.
\end{proof}

\begin{corollary}
Suppose that $H$ is infinite dimensional, $C$ has finite rank
and \eqref{eq:frame} holds. Then $\|A\|=1$.
The family sampled with $\mu_{\mathcal P}$ is not a frame,
whereas the family sampled with $\mu_{\log}$ is a frame.
\end{corollary}

\begin{proof}
If $q=\|A\|<1$, then $B=C^*C$ has finite rank and
$\sum_{n=N}^\infty\|A^nBA^n\|
\leq\|B\|q^{2N}/(1-q^2)\to0$ as $N\to\infty$.
Thus $S$ is a norm limit of finite rank operators and is compact.
This contradicts $S\geq aI$ on an infinite dimensional space.
Since $A$ is a contraction by Lemma~\ref{lem:stein},
we have $\|A\|=1$.
The sampling conclusions follow from
Example~\ref{ex:prime-sampling}.
\end{proof}

In particular, suppose that $H$ is infinite dimensional and
$(A^nf_j)_{n\in\mathbb N_0,\,1\leq j\leq m}$ is a frame.
Set $T=\{0\}\cup\mathcal P$, $w_0=1$ and $w_p=\log p$
for $p\in\mathcal P$.
Then $(A^tf_j)_{t\in T,\,1\leq j\leq m}$ is not a frame,
whereas $(\sqrt{w_t}\,A^tf_j)_{t\in T,\,1\leq j\leq m}$
is a frame. This follows by taking
$Cx=(\langle x,f_j\rangle)_{j=1}^m$ in the corollary.

\section{Optimal logarithmic density for a fixed geometric model}
\label{app:fixed-model}

We show that the coefficient $1/h$ in
Theorem~\ref{thm:local-density} is optimal for the fixed system
in Proposition~\ref{prop:cayley}.
For this pair $(A,f)$, define
\[
 \mathfrak d(h)
 =\inf\left\{
 D(\mu):
 (\sqrt{w_t}A^tf)_{t\in\operatorname{supp}\mu}
 \text{ is a frame for }\ell^2(\mathbb N)
 \right\}.
\]
Here the infimum is over sampling designs with positive weights
and locally finite support, and
$D(\mu)=\limsup_{R\to\infty}N_\mu(R)/\log R$.
Let $\mathfrak d_{\mathbb Z}(h)$ denote the same infimum with
$\operatorname{supp}\mu\subset\mathbb N_0$.
The sampling design and its frame bounds may depend on the
fixed pair $(A,f)$.

We first construct an invertible Toeplitz matrix.

\begin{lemma}
\label{lem:geometric-reference}
For every $q>1$, there exists $c>0$ such that
\(
 H_{q,c}
 =\bigl(q^{(k-j)/2}e^{-cq^{k-j}}\bigr)_{k,j\geq1}
\)
defines a bounded invertible operator on $\ell^2(\mathbb N)$.
\end{lemma}

\begin{proof}
For $c>0$, set
$h_c(z)=\sum_{n\in\mathbb Z}q^{n/2}e^{-cq^n}z^n$.
The coefficients are absolutely summable.
The negative coefficients decay geometrically, and the
positive coefficients decay faster than exponentially.
Thus $h_c$ is holomorphic on $|z|>q^{-1/2}$.

To choose $c$, let $g(u)=\exp(u/2-e^u)$, $a=\log q$
and $b=1/2$. Substitution of $v=e^u$ gives
$\widehat g(\xi)=\int_{\mathbb R}g(u)e^{-2\pi i\xi u}\,du
=\Gamma_{\!E}(b-2\pi i\xi)$,
where $\Gamma_{\!E}$ is the Euler Gamma function.
Its product formula, with
$\psi=\Gamma_{\!E}'/\Gamma_{\!E}$, gives
\[
 \widehat g(\xi)
 =\Gamma_{\!E}(b)e^{-2\pi i\psi(b)\xi}
  \prod_{n=0}^{\infty}
  \frac{e^{-2\pi i\xi/(b+n)}}{1-2\pi i\xi/(b+n)}.
\]
This is Schoenberg's product with parameters
$a_n=-(n+b)$ and $\sum_{n\geq0}a_n^{-2}<\infty$.
It has infinitely many factors and no Gaussian factor.
By Schoenberg's characterization \cite{Schoenberg1951},
$g$ is totally positive: for every $m\geq1$ and strictly
increasing real sequences $(x_i)_{i=1}^m$ and $(y_j)_{j=1}^m$,
we have $\det(g(x_i-y_j))_{i,j=1}^m\geq0$.
See also \cite[Section~1]{Kloos2015}.

Define the Zak transform by
\[
 Z_ag(x,\omega)
 =\sum_{k\in\mathbb Z}g(x-ka)e^{2\pi ika\omega}.
\]
The product form above is preserved by translation and positive
dilation. Hence \cite[Theorems~3.7 and~3.12]{Kloos2015},
after rescaling to step one, shows that $Z_ag$ has exactly one
zero in $[0,a)\times[0,1/a)$.
Choose $x_0\in[0,a)$ different from the first coordinate
of this zero. Periodicity in $\omega$ gives
$Z_ag(x_0,\omega)\neq0$ for every $\omega\in\mathbb R$.
Since
$h_c(e^{i\theta})
=c^{-1/2}Z_ag(\log c,-\theta/(2\pi a))$,
the choice $c_0=e^{x_0}$ makes $h_{c_0}$ nonzero on the
unit circle.

Let $\nu$ be the winding number of $h_{c_0}$ around zero
as the unit circle is traversed counterclockwise.
A change of index gives
$h_{qc}(z)=q^{-1/2}z^{-1}h_c(z)$.
Thus $c=q^\nu c_0$ makes $h_c$ nonzero on the unit circle
with winding number zero.
It is then nonzero on an annulus $r<|z|<R$ with $r<1<R$
and has a holomorphic logarithm there.

We finish with the standard factorization argument for scalar
Toeplitz operators; see \cite{BottcherSilbermann2006}.
Split the Laurent series of this logarithm into its strictly
negative and nonnegative parts, $\ell_-+\ell_+$, and set
$h_-=e^{\ell_-}$ and $h_+=e^{\ell_+}$.
The Laurent coefficients of $h_-$ and $h_-^{-1}$ are
absolutely summable and supported on nonpositive indices.
Those of $h_+$ and $h_+^{-1}$ are absolutely summable and
supported on nonnegative indices.

For a Laurent series $v(z)=\sum_{n\in\mathbb Z}v_nz^n$,
write $T(v)=(v_{k-j})_{k,j\geq1}$.
Multiplication of the absolutely convergent coefficient sums
gives
\(
 T(h_c)=T(h_-)T(h_+),
 \
 T(h_c)^{-1}=T(h_+^{-1})T(h_-^{-1}).
\)
Since $H_{q,c}=T(h_c)$, the result follows.
\end{proof}

\begin{theorem}
\label{thm:fixed-model-optimum}
For every $h>0$,
$\mathfrak d(h)=\mathfrak d_{\mathbb Z}(h)=1/h$.
Both infima are attained by sampling measures $\mu_h$
and $\widehat\mu_h$, respectively, with
$\operatorname{supp}\widehat\mu_h\subset\mathbb N_0$ and
\[
 N_{\mu_h}(R)=\frac{\log R}{h}+O_h(1),
 \qquad
 N_{\widehat\mu_h}(R)=\frac{\log R}{h}+O_h(1),
 \qquad R\to\infty.
\]
\end{theorem}

\begin{proof}
Theorem~\ref{thm:local-density} gives the lower bound $1/h$
for both infima. We construct measures attaining it.

Put $q=e^h$, and retain
$d_j=q^{-j}/2$,
$\lambda_j=(1-d_j)/(1+d_j)$ and
$f_j=\sqrt{1-\lambda_j^2}$ from
Proposition~\ref{prop:cayley}.
Set $s_j=-\log\lambda_j$ and $r_j=s_jq^j$.
Then $r_j\geq1$ and $r_j-1=O_h(q^{-2j})$.
The diagonal operator
$D=\operatorname{diag}(f_jq^{j/2})$ is bounded and invertible,
since $f_jq^{j/2}=\sqrt2/(1+q^{-j}/2)$.

\smallskip
\noindent\emph{Sampling at geometric times.}
Choose $c$ as in Lemma~\ref{lem:geometric-reference}.
Set $t_k=cq^k$ and
$\tau=\sum_{k\geq1}t_k\delta_{t_k}$.
The analysis operator, initially defined on finitely supported
vectors by
$(\mathcal T x)_k=\sqrt{t_k}\langle x,A^{t_k}f\rangle$,
has matrix entries
\(
 \mathcal T_{kj}
 =\sqrt c\,q^{(k-j)/2}e^{-cr_jq^{k-j}}D_{jj}.
\)
Compare it with the invertible operator
$\mathcal T_0=\sqrt c\,H_{q,c}D$.
The mean value theorem and the bounds on $r_j$ give
\(
 |\mathcal T_{kj}-(\mathcal T_0)_{kj}|
 \leq C_hq^{-2j}q^{3(k-j)/2}e^{-cq^{k-j}}\) for \(j,k\geq1\).
Consequently,
\[
 \sum_{j,k\geq1}
 |\mathcal T_{kj}-(\mathcal T_0)_{kj}|^2
 \leq C_h^2
 \sum_{j\geq1}q^{-4j}
 \sum_{n\in\mathbb Z}q^{3n}e^{-2cq^n}
 <\infty.
\]
Thus $\mathcal T-\mathcal T_0$ extends to a compact operator,
and $\mathcal T$ is bounded and Fredholm.
Here a bounded operator is Fredholm if its range is closed
and both its kernel and the orthogonal complement of its
range are finite dimensional.

It follows that $G_\tau=\mathcal T^*\mathcal T$ is a positive
Fredholm operator. Its quadratic form is the sampled energy:
the coordinate formula for $\mathcal T$ extends from finitely
supported vectors to all of $\ell^2(\mathbb N)$ by continuity.

\smallskip
\noindent\emph{Adding finitely many integer times.}
Let $G$ be any positive Fredholm sampling Gramian for this
system, and set $E=\ker G$ and $m=\dim E$.
Write $P_E$ for the orthogonal projection onto $E$.
Since the original integer orbit is a frame,
$\{P_EA^nf:n\geq0\}$ spans $E$.
Choose integers $n_1,\ldots,n_m$ whose projected vectors
form a basis of $E$.

Adding these observations with weight one gives a positive
operator $\widetilde G$ with quadratic form
\[
 \langle\widetilde Gx,x\rangle
 =\langle Gx,x\rangle
  +\sum_{\ell=1}^m|\langle x,A^{n_\ell}f\rangle|^2.
\]
The operator $\widetilde G$ has zero kernel and remains
Fredholm because the added operator has finite rank.
It is therefore boundedly invertible.
If $m=0$, no additional observations are needed.

Applying this argument to $G_\tau$ gives a measure
$\mu_h$ for which the sampled family is a frame.
Its support consists of $\{cq^k:k\geq1\}$ and finitely
many integer times. Hence
$N_{\mu_h}(R)=(\log R)/h+O_h(1)$.

\smallskip
\noindent\emph{Sampling at integer times.}
Keep the weights $t_k$ and replace $t_k$ by $\lceil t_k\rceil$.
Using the same coordinate index $k$, define
$(\widehat{\mathcal T}x)_k
=\sqrt{t_k}\langle x,A^{\lceil t_k\rceil}f\rangle$.

The geometric sum gives
$\tau([0,R])\leq qR/(q-1)$ for $R>0$.
Let $Q_M$ be the orthogonal projection in the target space
$\ell^2(\mathbb N)$ onto the coordinates with $t_k\geq M$.
Applying Lemma~\ref{s4-rounding} to the restriction of $\tau$
to $[M,\infty)$ gives
\[
 \|Q_M(\widehat{\mathcal T}-\mathcal T)\|^2
 \leq\frac{b_hq}{(q-1)M}\longrightarrow0,
 \qquad M\to\infty,
\]
where $b_h=\|S_h\|$ and $S_h$ is the original frame operator
from Proposition~\ref{prop:cayley}.
Only finitely many coordinates satisfy $t_k<M$.
Thus $\widehat{\mathcal T}-\mathcal T$ is compact, and
$\widehat{\mathcal T}^*\widehat{\mathcal T}$ is positive
and Fredholm.

Adding finitely many integer observations as above makes
this Gramian boundedly invertible.
At repeated rounded times, combine the weights; this leaves
the Gramian unchanged.
Since $\lceil cq^k\rceil$ is eventually strictly increasing,
the resulting measure $\widehat\mu_h$ satisfies
$N_{\widehat\mu_h}(R)=(\log R)/h+O_h(1)$.
Together with the lower bound, this proves the theorem.
\end{proof}

The designs in Theorem~\ref{thm:fixed-model-optimum} are chosen
for one fixed system. In contrast,
Theorem~\ref{s5-nonattainment} requires a common design for
every system with a prescribed bound on the original frame
condition number. The two attainment conclusions are therefore
consistent.

\section{Prescribed supports: obstructions and sparsification}
\label{app:prescribed-supports}

The main text chooses the sampling times and weights together.
Here we prescribe the sampling times and ask whether strictly
positive weights on every point give a measure in
$\mathcal U_w(\kappa)$.
For $\kappa=1$, bounded ratios of consecutive positive times
give a complete criterion. For $\kappa>1$, the local counting
bounds are necessary but not sufficient.
We also show that every support of a measure in
$\mathcal U_w(\kappa)$ contains a subset with
$t_{j+1}\geq e^\eta t_j$ for some $\eta>0$ that still supports
a measure in the same class.

\subsection{Support criteria and obstructions}
\label{sec:clustered-supports}

\begin{corollary}
\label{cor:tight-support}
Let $T=\{0,t_1,t_2,\ldots\}$, where
$0<t_1<t_2<\cdots$ and $t_j\to\infty$.
There exists $\mu\in\mathcal U_w(1)$ with
$\operatorname{supp}\mu=T$ if and only if
$\sup_{j\geq1}t_{j+1}/t_j<\infty$.
\end{corollary}

\begin{proof}
Suppose first that such a measure exists.
Proposition~\ref{prop:tight} gives
$cR\leq M_\mu(R)\leq CR$ for all sufficiently large $R$.
Choose $\Lambda>C/c$.
Every interval $(R,\Lambda R]$ with sufficiently large $R$
contains a support point. Otherwise,
$c\Lambda R\leq M_\mu(\Lambda R)=M_\mu(R)\leq CR$,
a contradiction.
Taking $R=t_j$ shows that $t_{j+1}\leq\Lambda t_j$
for all sufficiently large $j$.
The finitely many remaining ratios are also bounded.

Conversely, suppose $t_{j+1}/t_j\leq Q$ for every $j\geq1$.
Set $t_0=0$ and
$\mu=\delta_0+\sum_{j\geq1}(t_j-t_{j-1})\delta_{t_j}$.
All weights are strictly positive.
If $t_j\leq R<t_{j+1}$, then $M_\mu(R)=1+t_j$.
Hence $R/Q\leq M_\mu(R)\leq1+R$ for $R\geq t_1$.
Proposition~\ref{prop:tight} gives $\mu\in\mathcal U_w(1)$.
\end{proof}

For larger bounds on the original frame condition number,
the geometric model in Proposition~\ref{prop:cayley} gives
a necessary local counting condition.
We use the lower logarithmic density
$\underline D_{\log}$ defined in
\eqref{eq:lower-log-density}.

\begin{corollary}
\label{cor:local-support-density}
Suppose that $\operatorname{cond}(S_h)\leq\kappa$ for the
model in Proposition~\ref{prop:cayley}.
Every support $T$ of a measure in $\mathcal U_w(\kappa)$
satisfies $\underline D_{\log}(T)\geq1/h$.
For $\kappa>1$, this gives
$\underline D_{\log}(T)\geq1/h_\kappa$, where
$h_\kappa=2\log(1+4(\kappa+1)/(\kappa-1))$.
A support of a fully universal measure satisfies
$\underline D_{\log}(T)=\infty$.
\end{corollary}

\begin{proof}
Every measure in $\mathcal U_w(\kappa)$ must preserve the
frame property of the stated model.
Taking the infimum over $R$ in \eqref{eq:local-density},
dividing by $L$ and letting $L\to\infty$ gives the first claim.
The proof of Corollary~\ref{cor:positive-lower} shows that
$\operatorname{cond}(S_{h_\kappa})\leq\kappa$,
which gives the explicit bound.
A fully universal measure must preserve every such model,
so its support has lower logarithmic density at least $1/h$
for every $h>0$.
\end{proof}

\begin{theorem}
\label{thm:shrinking-clusters}
Fix $h>0$ and the system in Proposition~\ref{prop:cayley}.
Let $q>1$ satisfy $\log q>h$, and suppose
$r_n\geq0$ and $r_n/q^n\to0$.
If a locally finite set $T$, apart from finitely many points,
is contained in
$\bigcup_{n\geq1}[q^n,q^n+r_n]$,
then no choice of strictly positive weights on $T$
makes the sampled family a frame for this system.
\end{theorem}

\begin{proof}
Suppose that a measure $\mu$ supported on $T$ gives a frame
with lower bound $a_\mu>0$.
The upper frame bound, as used in the proof of
Theorem~\ref{thm:local-density}, gives
$M_\mu(R)\leq KR$ for $R\geq1$.

Choose $D>h+\log2$ so large that the sum of the bounds in
\eqref{eq:model-left-energy} and
\eqref{eq:model-right-energy} is less than $a_\mu/2$.
Put $H_q=\log q$.
Since $H_q>h$, we can choose an integer $M\geq1$ such that
\[
 L=MH_q\geq2D+h,
 \qquad
 m=\left\lfloor\frac{L-2D-h}{h}\right\rfloor+1>M+1.
\]

Set $R_n=q^n$.
Let $j_n$ be the first index with
$d_{j_n}\leq e^{-D}/R_n$, and put
$J_n=\{j_n,\ldots,j_n+m-1\}$ and
$E_n=\operatorname{span}\{e_j:j\in J_n\}$.
Since $\dim E_n=m>M+1$, there is a unit vector $x_n\in E_n$
such that, with $v_n(t)=\langle x_n,A^tf\rangle$,
\[
 v_n(q^k)=0,
 \qquad k=n,\ldots,n+M.
\]
The centers $q^k$ need not belong to $T$.

For sufficiently large $n$, every sample in
$(R_n,e^LR_n]$ belongs to one of the clusters with indices
$n,\ldots,n+M$.
Indeed, $r_k/q^k<q-1$ eventually, so earlier clusters end
before $R_n$, while later clusters start after
$e^LR_n=q^{n+M}$.
The finitely many initial clusters and exceptional points
also lie before $R_n$ for large $n$.
Set $\Delta_n=\max_{n\leq k\leq n+M}r_k$.
Since $M$ is fixed,
\[
 \frac{\Delta_n}{R_n}
 =\max_{0\leq\ell\leq M}
   q^\ell\frac{r_{n+\ell}}{q^{n+\ell}}
 \longrightarrow0,
 \qquad n\to\infty.
\]

Write $s_j=-\log\lambda_j$.
The estimates $f_j^2\leq4d_j$ and
$s_j\leq(8/3)d_j$ from the proof of
Theorem~\ref{thm:local-density}, together with the
Cauchy--Schwarz inequality, give
\[
 |v_n'(u)|^2
 \leq\sum_{j\in J_n}f_j^2s_j^2e^{-2s_ju}
 \leq\frac{256}{9}\sum_{j\in J_n}d_j^3
 \leq C_h\frac{e^{-3D}}{R_n^3},
 \qquad u\geq0,
\]
where $C_h=256/[9(1-e^{-3h})]$.
Since $v_n$ vanishes at each relevant cluster center,
the fundamental theorem of calculus gives
$|v_n(t)|^2\leq C_he^{-3D}\Delta_n^2/R_n^3$
at every sample in the window.
Therefore,
\[
 \int_{(R_n,e^LR_n]}|v_n(t)|^2\,d\mu(t)
 \leq C_hKe^{L-3D}
       \left(\frac{\Delta_n}{R_n}\right)^2
 \longrightarrow0,
 \qquad n\to\infty.
\]

The bounds in \eqref{eq:model-left-energy} and
\eqref{eq:model-right-energy} apply to every unit vector
in this coordinate subspace, independently of its vanishing
conditions. Thus the energy outside the window is less than
$a_\mu/2$.
For large $n$, the total sampled energy is less than
$a_\mu$, contradicting $\|x_n\|=1$.
\end{proof}

\begin{corollary}
For every $\kappa>1$, there is a set
$T=\{0,t_1,t_2,\ldots\}\subset\mathbb N_0$, with
$0<t_1<t_2<\cdots$, such that
\[
 N_T(R)\sim(\log R)^2,
 \qquad
 \sup_{j\geq1}\frac{t_{j+1}}{t_j}<\infty,
 \qquad
 \underline D_{\log}(T)=\infty,
\]
where $N_T(R)=\#(T\cap[0,R])$.
This set is the support of a measure in $\mathcal U_w(1)$
but of no measure in $\mathcal U_w(\kappa)$.
\end{corollary}

\begin{proof}
Choose an integer $q>\max\{16,e^{h_\kappa}\}$, where
$h_\kappa$ is defined in
Corollary~\ref{cor:local-support-density}.
Put $H_q=\log q$ and
\[
 m_n=\lfloor2H_q^2n\rfloor,
 \qquad
 T=\{0\}\cup
   \bigcup_{n\geq1}\{q^n,q^n+1,\ldots,q^n+m_n\}.
\]
Since $2(\log q)^2\leq q$ and $nq\leq q^n$,
we have $m_n\leq q^n$.
The clusters are therefore disjoint, and each ends before
$q^{n+1}$.
Consecutive positive sampling times have ratio at most $q$,
so Corollary~\ref{cor:tight-support} gives a measure in
$\mathcal U_w(1)$ with support $T$.

On the other hand, $m_n/q^n\to0$ and $\log q>h_\kappa$.
Theorem~\ref{thm:shrinking-clusters} rules out every weighting
that gives a frame for the model at step $h_\kappa$.
Its original frame condition number is at most $\kappa$.
Thus no measure in $\mathcal U_w(\kappa)$ has support $T$.

The first $n$ complete clusters contain
$\sum_{k=1}^n(m_k+1)=H_q^2n^2+O(n)$ points.
A partial cluster contributes only $O(n)$ points.
Since $\log R=H_qn+O(1)$ for $q^n\leq R<q^{n+1}$,
we obtain $N_T(R)\sim(\log R)^2$.

To verify the lower logarithmic density, fix $R\geq1$
and write $s=\log R$.
Set $n_0=\lfloor s/H_q\rfloor+1$ and
$p=\lfloor(L-\log2)/H_q\rfloor$.
For $L>\log2+H_q$, the $p$ complete clusters with indices
$n_0,\ldots,n_0+p-1$ lie in $(R,e^LR]$.
Their first center is greater than $R$, and the logarithm
of the last endpoint is at most
$s+pH_q+\log2\leq s+L$.
Since $m_k+1\geq2H_q^2k$ and $n_0\geq1$,
\[
 \#(T\cap(R,e^LR])
 \geq2H_q^2\sum_{k=1}^p k
 =H_q^2p(p+1).
\]
For $L\geq4H_q$, we have $p\geq L/(2H_q)$.
Hence
\[
 \inf_{R\geq1}
 \frac{\#(T\cap(R,e^LR])}{L}
 \geq\frac L4\longrightarrow\infty,
 \qquad L\to\infty.
\]
This proves $\underline D_{\log}(T)=\infty$.
\end{proof}

\subsection{Relative time perturbations and sparsification}
\label{s4-rt-section}

We now show that small relative changes of the sampling times
preserve membership in $\mathcal U_w(\kappa)$ for each fixed
finite $\kappa$.
This allows us to combine weights at nearby times and retain
only a subset of the original support.

As in Section~\ref{sec:preliminaries}, let
$L=-\log A$ on $H_+=(\ker A)^\perp$.
For $t>0$, expressions involving $L$ and $A^t$ are extended
by zero on $\ker A$.

\begin{lemma}
\label{s4-rt-derivative}
Let $\tau$ be a positive Borel measure on $(0,\infty)$
such that $\tau((0,r])\leq Kr$ for every $r>0$.
Every system satisfying \eqref{eq:frame} obeys
\[
 \int_{(0,\infty)}t^2\|CLA^tx\|^2\,d\tau(t)
 \leq(1+\sqrt3)bK\|x\|^2,
 \qquad x\in H.
\]
\end{lemma}

\begin{proof}
We first estimate the second derivative.
For $x\in H_+$, set
$x_m=E_L([m^{-1},m])x$.
By Lemma~\ref{lem:stein}, $E_A(\{1\})=0$, so $x_m\to x$.
The Stein identity and a change of variables give
\[
 \begin{aligned}
 \int_0^\infty t^4\|CL^2A^tx_m\|^2\,dt
 &=\int_0^\infty
   \bigl[t^4-(t-1)_+^4\bigr]
   \langle SL^2A^tx_m,L^2A^tx_m\rangle\,dt\\
 &\leq4b\int_0^\infty t^3\|L^2A^tx_m\|^2\,dt
 =\frac{3b}{2}\|x_m\|^2,
 \end{aligned}
\]
where $(t-1)_+=\max\{t-1,0\}$.
The last equality follows from the spectral theorem.
Fatou's lemma and the convention on $\ker A$ yield
$\int_0^\infty t^4\|CL^2A^tx\|^2\,dt
\leq(3b/2)\|x\|^2$ for every $x\in H$.

Set $H_x(t)=t^2\|CLA^tx\|^2$ for $t>0$.
Differentiation, the Cauchy--Schwarz inequality and
Lemma~\ref{lem:energy} give
\[
 \begin{aligned}
 \int_0^\infty t|H_x'(t)|\,dt
 &\leq2\int_0^\infty t^2\|CLA^tx\|^2\,dt\\
 &\quad
 +2\left(\int_0^\infty t^2\|CLA^tx\|^2\,dt\right)^{1/2}
   \left(\int_0^\infty t^4\|CL^2A^tx\|^2\,dt\right)^{1/2}\\
 &\leq(1+\sqrt3)b\|x\|^2.
 \end{aligned}
\]

The spectral theorem and dominated convergence give
$tLA^tx\to0$ as $t\to\infty$.
Thus $H_x(t)\to0$ and
$H_x(t)=-\int_t^\infty H_x'(s)\,ds$.
Writing $W(s)=\tau((0,s])\leq Ks$, we have
$\int_0^\infty W(s)|H_x'(s)|\,ds<\infty$.
Fubini's theorem therefore gives
\[
 \int_{(0,\infty)}H_x(t)\,d\tau(t)
 =-\int_0^\infty W(s)H_x'(s)\,ds
 \leq K\int_0^\infty s|H_x'(s)|\,ds.
\]
The preceding estimate proves the lemma.
\end{proof}

\begin{theorem}
\label{s4-rt-stability}
Let $\mu=\sum_{t\in T}w_t\delta_t$ be a sampling design
such that
$K_\mu=\sup_{u>0}\mu((0,u])/u<\infty$.
For $\eta\geq0$, choose $r(t)>0$ for every positive $t\in T$
with $|\log(r(t)/t)|\leq\eta$, and set $r(0)=0$
when $0\in T$.
Let $\mu_r=\sum_{t\in T}w_t\delta_{r(t)}$, with weights
at coincident times combined.

For every system satisfying \eqref{eq:frame}, the operators
into $\ell^2(T;Y)$ defined by
\(
 (\mathcal T_\mu x)_t=\sqrt{w_t}\,CA^tx,
 \
 (\mathcal T_r x)_t=\sqrt{w_t}\,CA^{r(t)}x
\)
are bounded and satisfy
\begin{equation}\label{s4-rt-perturbation-bound}
 \|\mathcal T_r-\mathcal T_\mu\|
 \leq\eta\sqrt{(1+\sqrt3)bK_\mu e^\eta}.
\end{equation}
If $\mu\in\mathcal U_w(\kappa)$ for a finite $\kappa\geq1$,
then $\mu_r\in\mathcal U_w(\kappa)$ whenever
$0\leq\eta<\eta_0$, for some $\eta_0>0$ depending only
on $\mu$ and $\kappa$.
\end{theorem}

\begin{proof}
Write $w_0=\mu(\{0\})$.
If $r(t)\leq R$, then $t\leq e^\eta R$.
Hence $\mu_r$ has locally finite support and finite positive
weights after coincident times are combined.
Also,
$\mu_r((0,u])\leq\mu((0,e^\eta u])\leq K_\mu e^\eta u$.
Proposition~\ref{prop:mass} gives
$\|\mathcal T_\mu\|^2\leq b(w_0+\sqrt2K_\mu)$ and
$\|\mathcal T_r\|^2\leq b(w_0+\sqrt2e^\eta K_\mu)$.

For $t>0$, put $d_t=\log(r(t)/t)$ and
$s_\theta(t)=te^{\theta d_t}$ for $0\leq\theta\leq1$.
Differentiation along this path and the Cauchy--Schwarz
inequality give
\[
 \|C(A^{r(t)}-A^t)x\|^2
 \leq\eta^2\int_0^1
 s_\theta(t)^2\|CLA^{s_\theta(t)}x\|^2\,d\theta.
\]
For each $\theta$, define
$\tau_\theta=\sum_{t\in T\setminus\{0\}}
w_t\delta_{s_\theta(t)}$.
Then
$\tau_\theta((0,u])
\leq\mu((0,e^\eta u])\leq K_\mu e^\eta u$.
Summing the preceding estimate with weights $w_t$ and
applying Tonelli's theorem and
Lemma~\ref{s4-rt-derivative} proves
\eqref{s4-rt-perturbation-bound}.
The observation at zero contributes no difference.

Now suppose $\mu\in\mathcal U_w(\kappa)$.
Proposition~\ref{prop:tight} ensures $K_\mu<\infty$.
For systems normalized by $I\leq S\leq\kappa I$,
choose uniform bounds
$\alpha I\leq G_\mu\leq\beta I$.
Set
$\varepsilon_\eta
=\eta\sqrt{(1+\sqrt3)\kappa K_\mu e^\eta}$,
and choose $\eta_0>0$ with
$\varepsilon_{\eta_0}<\sqrt\alpha$.
The triangle inequality and
\eqref{s4-rt-perturbation-bound} imply
\[
 (\sqrt\alpha-\varepsilon_\eta)^2I
 \leq G_{\mu_r}
 \leq(\sqrt\beta+\varepsilon_\eta)^2I,
 \qquad 0\leq\eta<\eta_0.
\]
Combining coincident weights preserves the squared analysis
norm, so these are the bounds for the merged measure.
They are uniform over the normalized systems.
For general bounds $aI\leq S\leq bI$ with $b/a\leq\kappa$,
apply the result to $C/\sqrt a$ and multiply the resulting
Gramian bounds by $a$.
Thus $\mu_r\in\mathcal U_w(\kappa)$.
\end{proof}

\begin{theorem}
\label{s4-rt-sparsification}
Let $1\leq\kappa<\infty$ and
$\mu\in\mathcal U_w(\kappa)$, with support $T$.
There exist $\eta>0$ and $\nu\in\mathcal U_w(\kappa)$ such that
\[
 \operatorname{supp}\nu=T'=\{0,t_1,t_2,\ldots\}\subseteq T,
 \qquad
 t_{j+1}\geq e^\eta t_j,
 \qquad j\geq1.
\]
Moreover, $D(\nu)\leq D(\mu)$ and
$N_\nu(R)\leq\eta^{-1}\log R+O_{\mu,\kappa}(1)$
as $R\to\infty$.
\end{theorem}

\begin{proof}
By Proposition~\ref{prop:tight},
$0\in T$ and $M_\mu(R)\asymp R$.
Thus $T$ is unbounded.
Local finiteness gives a first positive point $t_1$.

Choose $0<\eta<\eta_0$ as in
Theorem~\ref{s4-rt-stability}, and define
$t_{j+1}=\min\{t\in T:t\geq e^\eta t_j\}$.
Unboundedness and local finiteness ensure that this minimum
exists at every step.
By construction, every point of $T\cap[t_j,t_{j+1})$
lies in $[t_j,e^\eta t_j)$.
Set
\(
 \nu=w_0\delta_0+
 \sum_{j\geq1}\mu([t_j,t_{j+1}))\delta_{t_j}.
\)
Every new positive weight is a finite sum of original weights,
and the weight at zero is unchanged.
Moving each time in $T\cap[t_j,t_{j+1})$ to $t_j$
changes its logarithm by less than $\eta$.
Theorem~\ref{s4-rt-stability} therefore gives
$\nu\in\mathcal U_w(\kappa)$.

Since $T'\subseteq T$, we have
$N_\nu(R)\leq N_\mu(R)$ and hence $D(\nu)\leq D(\mu)$.
Finally, $t_j\geq t_1e^{(j-1)\eta}$ gives the stated
counting bound.
\end{proof}

\begin{remark}
Fix $1\leq\kappa<\infty$, and let
$T\subset[0,\infty)$ be locally finite with $0\in T$.
There exists a measure in $\mathcal U_w(\kappa)$ with
support exactly $T$ if and only if $T$ contains the support
$T'=\{0,t_1,t_2,\ldots\}$ of a measure in
$\mathcal U_w(\kappa)$ with
$\inf_{j\geq1}\log(t_{j+1}/t_j)>0$.
\end{remark}

\begin{proof}
The forward implication is
Theorem~\ref{s4-rt-sparsification}.
Conversely, let $\nu\in\mathcal U_w(\kappa)$ have support $T'$.
Assign strictly positive summable weights to
$T\setminus T'$, and let $\sigma$ be the resulting measure.
If its total mass is $c<\infty$, then
$0\leq G_\sigma\leq bcI$, since
$\|C\|^2\leq b$ and $\|A^t\|\leq1$ for every $t\geq0$.
Thus $\nu+\sigma$ has support exactly $T$, retains the
uniform lower bound of $\nu$ and has a uniform upper bound.
Hence $\nu+\sigma\in\mathcal U_w(\kappa)$.
\end{proof}

\section{Periodic comparison and periodization in logarithmic coordinates}
\label{app:periodic-tools}

This appendix proves the scalar estimates used in
Subsection~\ref{s7-exact}.
We first compare positive atomic measures on a circle and
prove strict monotonicity of $\mathcal R$.
We then construct periodic measures from finite logarithmic
intervals, with uniform error bounds that imply the local
counting theorem.

\subsection{Periodic comparison and monotonicity}

Use the kernel $k$ and the periodic function $P_h$ from
\eqref{eq:tight-periodic-profile}, and write
$m_h=\min P_h$, $M_h=\max P_h$ and
$\mathcal R(h)=M_h/m_h$.
For a measure $\nu$ on $\mathbb R$, we use the convention
$(k*\nu)(y)=\int_{\mathbb R}k(y-x)\,d\nu(x)$.

If $\mu$ has an atom of mass $w$ at $t=e^x>0$, assign mass
$w/t$ to $x$, and denote the resulting measure by $\nu$.
With $u=e^{-y}$, this gives
$u\int_{(0,\infty)}e^{-ut}\,d\mu(t)=(k*\nu)(y)$.
For $L>0$, set $k_L(s)=\sum_{n\in\mathbb Z}k(s+nL)$.
This series and its derivatives converge uniformly on compact
sets, and $\int_0^Lk_L(s)\,ds=1$.
Convolution on the circle $\mathbb R/L\mathbb Z$ is taken
over one period.

For a real continuous function on a circle, let $V_c$ be
the largest even number of strictly alternating positive and
negative values that can be selected in cyclic order.
This number may be infinite; it is zero when no sign change
occurs, including for the zero function.
For a finite signed atomic measure, combine atoms at the
same point, omit zero masses and count the cyclic sign
changes of the remaining masses.

\begin{lemma}
\label{s7-variation}
For every $L>0$, every real continuous function $f$ of period
$L$, and every finite signed atomic measure $\sigma$ on
$\mathbb R/L\mathbb Z$,
\[
 V_c(k_L*f)\leq V_c(f),
 \qquad
 V_c(k_L*\sigma)\leq V_c(\sigma).
\]
\end{lemma}

\begin{proof}
The kernel $k$ is integrable and totally positive:
$\det(k(x_i-y_j))_{i,j=1}^m\geq0$
for every $m\geq1$ and strictly increasing real sequences
$(x_i)_{i=1}^m$ and $(y_j)_{j=1}^m$;
see \cite[p.~212, (3.15)]{Schoenberg1953}.
Schoenberg's periodic convolution theorem
\cite[pp.~225--226, (4.25)--(4.27)]{Schoenberg1953},
after rescaling the period, proves the assertion for functions.

For an atomic measure, replace each nonzero atom by a smooth
bump with the same sign and mass.
Choose the supporting arcs disjoint and in the original
cyclic order.
The resulting functions $f_\epsilon$ satisfy
$V_c(f_\epsilon)\leq V_c(\sigma)$ and converge weakly
to $\sigma$.
Uniform continuity of $k_L$ gives
$k_L*f_\epsilon\to k_L*\sigma$ uniformly.
Every finite collection of strictly alternating signs in
the limit persists for sufficiently small $\epsilon$.
The assertion for measures follows.
\end{proof}

We apply this variation bound to compare an arbitrary positive
atomic measure with an equally spaced measure having the same
number of atoms.

\begin{lemma}
\label{s7-periodic-optimum}
Let $\nu$ be a finite positive atomic measure with exactly
$N\geq1$ distinct atoms on $\mathbb R/L\mathbb Z$.
Then
\[
 \frac{\max_s(k_L*\nu)(s)}{\min_s(k_L*\nu)(s)}
 \geq\mathcal R(L/N).
\]
Equally spaced atoms with equal masses attain equality.
\end{lemma}

\begin{proof}
Set $h=L/N$.
Translate the equally spaced set so that one point coincides
with an atom of $\nu$, and put
$\chi=h\sum_{j=0}^{N-1}\delta_{\theta+jh}$.
Then $k_L*\chi=P_h(\cdot-\theta)$.
For every $c>0$, the measure $\nu-c\chi$ has at most
$2N-1$ nonzero atoms.
Its number of cyclic sign changes is even, so
$V_c(\nu-c\chi)\leq2N-2$.

If $P_h$ is constant, the inequality is immediate.
Otherwise, suppose it fails, and let $m$ and $M$ be the
minimum and maximum of $k_L*\nu$.
Choose $c$ with $M/M_h<c<m/m_h$.
Then $g=k_L*(\nu-c\chi)$ is negative at each translated
maximum of $P_h$ and positive at each translated minimum.
Choosing one maximum and one minimum in each period gives
$2N$ points with alternating signs in cyclic order.
Thus $V_c(g)\geq2N$, contradicting
Lemma~\ref{s7-variation}.
For equally spaced atoms with equal masses, the convolution
is a positive multiple of a translate of $P_h$,
which proves equality.
\end{proof}

To study $\mathcal R$, use the Fourier convention
$\widehat k(\xi)=\int_{\mathbb R}k(s)e^{-i\xi s}\,ds$.
Substitution of $v=e^{-s}$ gives
$\widehat k(\xi)=\int_0^\infty e^{-v}v^{i\xi}\,dv
=\Gamma_{\!E}(1+i\xi)$,
where $\Gamma_{\!E}$ denotes the Euler Gamma function.

\begin{lemma}
\label{s7-monotonicity}
The function $\mathcal R$ is a continuous strictly increasing
bijection from $(0,\infty)$ onto $(1,\infty)$.
\end{lemma}

\begin{proof}
Set $q_h(t)=P_h(ht)$, which has period one.
Its Fourier coefficients are
\begin{equation}\label{s7-fourier-profile}
 \widehat q_h(m)
 =\Gamma_{\!E}\left(1+\frac{2\pi im}{h}\right),
 \qquad m\in\mathbb Z.
\end{equation}
The Gamma product formula gives
\begin{equation}\label{s7-gamma-modulus}
 |\Gamma_{\!E}(1+i\xi)|^2
 =\prod_{n=1}^\infty
  \left(1+\frac{\xi^2}{n^2}\right)^{-1}
 =\frac{\pi\xi}{\sinh(\pi\xi)},
 \qquad \xi\in\mathbb R,
\end{equation}
where the last expression is assigned the value $1$ at
$\xi=0$; see \cite[\S5.8]{DLMF}.
The first two factors imply
$|\widehat q_h(m)|\leq h^2/(2\pi^2m^2)$ for $m\neq0$.
Thus the Fourier series converges uniformly for $h$ in
compact subsets of $(0,\infty)$, and $q_h$ depends
continuously on $h$ in the uniform norm.
Consequently, $m_h$, $M_h$ and $\mathcal R(h)$ are continuous.
The same bound gives $q_h\to1$ uniformly as $h\downarrow0$.
Since $\widehat q_h(0)=1$ and $\widehat q_h(1)\neq0$,
we also have $m_h<1<M_h$.

To prove strict monotonicity, fix $0<h_1<h_2$ and set
$c=h_1/h_2$.
By \cite[Proposition~4]{Jurek2022}, with both parameters
equal to one, there is a real random variable $Z_c$ such that
\[
 \mathbb E e^{itZ_c}
 =\frac{\Gamma_{\!E}(1+it)}{\Gamma_{\!E}(1+ict)},
 \qquad t\in\mathbb R.
\]
Let $\sigma$ be the distribution of $V=-Z_c/h_1$ modulo one.
Comparison with \eqref{s7-fourier-profile} gives
$q_{h_1}=\sigma*q_{h_2}$.

The independent series in the same proposition gives
$V=BE/h_1+W$, where $B$, $E$ and $W$ are independent,
$B\in\{0,1\}$ with $\mathbb P(B=1)=1-c$, and $E$ has
the exponential distribution with rate one.
The distribution of $E/h_1$ modulo one has density
$r(t)=h_1e^{-h_1t}/(1-e^{-h_1})$ for $0\leq t<1$.
Its infimum is $h_1/(e^{h_1}-1)$.
Hence $\sigma\geq\epsilon\,dt$ on the unit circle, where
$\epsilon=(1-c)h_1/(e^{h_1}-1)\in(0,1)$.
Since $q_{h_2}$ has mean one,
\[
 m_{h_1}
 \geq\epsilon+(1-\epsilon)m_{h_2}>m_{h_2},\qquad
 M_{h_1}
 \leq\epsilon+(1-\epsilon)M_{h_2}<M_{h_2}.
\]
Therefore $\mathcal R(h_1)<\mathcal R(h_2)$.

The uniform convergence $q_h\to1$ gives
$\mathcal R(h)\to1$ as $h\downarrow0$.
Also, $P_h(0)\geq h/e$ and $m_h\leq1$, so
$\mathcal R(h)\geq h/e\to\infty$ as $h\to\infty$.
These limits complete the proof.
\end{proof}

\subsection{Periodization and local counting}

The next construction converts a finite logarithmic interval
into a periodic measure.
Its error estimates are uniform in the measure and the
location of the interval.

\begin{lemma}
\label{s7-periodization}
Let $\mu$ be a positive atomic measure with locally finite
support.
Suppose $0<a\leq\phi_\mu(z)\leq1$ for $0\leq z<1$,
where $\phi_\mu$ is defined in Subsection~\ref{s7-exact},
and let
$D=\limsup_{R\to\infty}N_\mu(R)/\log R<\infty$.
For every $0<\epsilon<a$, there are positive atomic measures
$\omega_L$ of period $L$, for all sufficiently large $L$,
with $n_L$ distinct atoms per period, such that
\[
 a-\epsilon\leq(k*\omega_L)(y)\leq1+\epsilon,
 \qquad y\in\mathbb R,
 \qquad
 \limsup_{L\to\infty}\frac{n_L}{L}\leq D+\epsilon.
\]
\end{lemma}

\begin{proof}
Set $Q_\mu(u)=\int_{[0,\infty)}e^{-ut}\,d\mu(t)$.
The scalar upper bound gives
\[
 M_\mu(R)
 \leq eQ_\mu(1/R)
 \leq\frac{e}{1-e^{-1/R}}
 \leq2eR,
 \qquad R\geq1.
\]
Remove the atoms in $[0,1]$ and define
$\nu=\sum_{t>1}(\mu(\{t\})/t)\delta_{\log t}$.
Every interval of length one has $\nu$ mass at most
$B_0=2e^2$.
Indeed, for $x\geq0$,
$\nu([x,x+1])\leq e^{-x}M_\mu(e^{x+1})\leq2e^2$;
the same bound holds for intervals meeting the negative
axis because $\operatorname{supp}\nu\subset(0,\infty)$.

Let $F=k*\nu$.
For $y\geq0$ and $u=e^{-y}$,
\[
 F(y)
 =\frac{u}{1-e^{-u}}\phi_\mu(e^{-u})
  -u\int_{[0,1]}e^{-ut}\,d\mu(t).
\]
Since $\mu([0,1])\leq2e$ and
$1\leq u/(1-e^{-u})\leq1+u$, we obtain
$a-2e\,e^{-y}\leq F(y)\leq1+e^{-y}$.
Thus, for every $\eta>0$, we can choose $Y_\eta\geq0$,
independently of $\mu$, so that
$a-\eta\leq F(y)\leq1+\eta$ for $y\geq Y_\eta$.

Fix $Y\geq Y_\eta$, $W\geq2$ and $L>4W$.
Let $\theta=\theta_L$ vanish on
$[0,W]\cup[L-W,L]$, equal one on $[2W,L-2W]$,
and be linear on the intervening intervals.
Extend it by zero outside $[0,L]$.
Then $0\leq\theta\leq1$ and its Lipschitz constant
is at most $1/W$.
Write $\nu_Y(E)=\nu(Y+E)$ and $c=(a+1)/2$.
Periodically extend the measure
\(
 \theta(x)\,d\nu_Y(x)+c(1-\theta(x))\,dx\) for \(0\leq x<L\),
and denote it by $\widetilde\omega_L$.
Its continuous part is confined to the four boundary
intervals of length $W$.

We claim that
\begin{equation}\label{s7-boundary-error}
 \left|
 (k*\widetilde\omega_L)(y)
 -\bigl[c+\theta(y)(F(Y+y)-c)\bigr]
 \right|
 \leq\frac{C}{W}+Ce^{-W},
 \qquad 0\leq y<L,
\end{equation}
where $C$ depends only on $k$ and $B_0$.

The convolution equals
\[
 c+\sum_{j\in\mathbb Z}\int_0^L
 k(y-x-jL)\theta(x)\,d(\nu_Y-c\,dx)(x).
\]
For $j=0$, subtracting $\theta(y)(F(Y+y)-c)$ leaves
\[
 \int_{\mathbb R}
 k(y-x)(\theta(x)-\theta(y))\,d(\nu_Y-c\,dx)(x).
\]
Its absolute value is at most
\[
 \frac1W\int_{\mathbb R}|y-x|k(y-x)\,
 d(\nu_Y+c\,dx)(x)\leq\frac CW.
\]
To obtain the last bound, partition the line into unit
intervals, use the uniform mass bound and note that
$\sum_{j\in\mathbb Z}
\sup_{s\in[j,j+1]}(1+|s|)k(s)<\infty$.

For $j\neq0$, the translated supports of $\theta$ are
disjoint and lie at distance at least $W$ from
$y\in[0,L]$.
Their measures retain a uniform mass bound on unit intervals.
The exponential decay of $k$ at $+\infty$ and its faster
decay at $-\infty$ therefore bound their total contribution
by $Ce^{-W}$.
This proves \eqref{s7-boundary-error}.
Since $c\in[a,1]$, the convolution lies between
$a-\eta-C/W-Ce^{-W}$ and $1+\eta+C/W+Ce^{-W}$.

We next replace the continuous boundary measure by atoms.
Partition each of its four boundary intervals into pieces
of length $\delta\leq1$, allowing a shorter final piece,
and move the mass of each piece to one point in that piece.
Repeat the construction periodically.
After omitting zero masses and combining coincident atoms,
the resulting measure $\omega_L$ is positive and atomic.
At most $4W/\delta+4$ atoms have been added.

The mean value theorem bounds the uniform convolution error
by
\[
 \begin{aligned}
 c\delta\int_{\mathbb R}
 \sup_{|v-s|\leq\delta}|k'(v)|\,ds
 &\leq
 c\delta\bigl(\|k'\|_{L^1}
             +2\delta\|k''\|_{L^1}\bigr).
 \end{aligned}
\]
The last inequality follows from the fundamental theorem
of calculus and Fubini's theorem.
Both derivatives are integrable, so this error is
$O(\delta)$, independently of $L$, $W$, $Y$ and $\mu$.

Choose $\eta$ small, then $W$ large and finally $\delta$
small, so that the total error is at most $\epsilon$.
Fix these parameters and $Y$ before letting $L\to\infty$.
The construction and its error estimates use only local
finiteness and the scalar bounds; the assumption $D<\infty$
is needed only for the following count:
\[
 n_L
 \leq N_\mu(e^{Y+L})+4W/\delta+4
 \leq(D+\epsilon)(Y+L)+4W/\delta+4
\]
for all sufficiently large $L$.
Dividing by $L$ proves the claim.
\end{proof}

\begin{proof}[Proof of Theorem~\ref{s7-local-counting}]
Fix $\Gamma>1$ and $\gamma>\Gamma$.
Multiplying all weights by a common constant preserves
the support.
By \eqref{s7-scalar}, we may therefore assume
$a\leq\phi_\mu(z)\leq1$ for $0\leq z<1$,
where $a=1/\Gamma$.

Use the measures $\nu$ and functions $F=k*\nu$ from the
proof of Lemma~\ref{s7-periodization}.
The estimates there give a mass bound $2e^2$ on every
unit interval and, for each $\eta>0$, a number
$Y_\eta\geq0$ independent of $\mu$ such that
$a-\eta\leq F(y)\leq1+\eta$ for $y\geq Y_\eta$.

Choose $0<\zeta<a$ with
$(1+\zeta)/(a-\zeta)\leq\gamma$.
For each interval $[Y,Y+\ell]$ with
$Y\geq Y_\eta$ and $\ell>4W$, apply the construction in
the proof of Lemma~\ref{s7-periodization}.
By \eqref{s7-boundary-error} and the boundary discretization
estimate, its total error is at most
\[
 \eta+\frac CW+Ce^{-W}
 +c\delta\bigl(\|k'\|_{L^1}
              +2\delta\|k''\|_{L^1}\bigr),
 \qquad c=(a+1)/2.
\]
Choose $\eta$ small, then $W$ large and finally $\delta$
small, so that this is at most $\zeta$.
All parameters, including $Y_\eta$, now depend only on
$\Gamma$ and $\gamma$.

The resulting positive periodic atomic measure has
convolution between $a-\zeta$ and $1+\zeta$.
If it has $n$ distinct atoms per period, then
\(
 n\leq\#\bigl(T\cap(e^Y,e^{Y+\ell}]\bigr)+B,
 \ B=\lceil4W/\delta\rceil+4.
\)
At most $B$ atoms come from the boundary discretization.
Every other atom comes from the stated window, and the
cutoff removes its endpoints.
Local finiteness ensures that this number is finite;
no global bound on $N_\mu$ is needed.

Lemmas~\ref{s7-periodic-optimum} and
\ref{s7-monotonicity} give
$\mathcal R(\ell/n)\leq(1+\zeta)/(a-\zeta)\leq\gamma$,
and hence $n\geq\ell/h_\gamma$.
Thus \eqref{s7-local-window} holds with error $B$
when $\log R\geq Y_\eta$ and $\ell>4W$.

For general $R\geq1$, remove the portion of
$[\log R,\log R+\ell]$ below $Y_\eta$.
Its length is at most $Y_\eta$.
If the remaining interval has length greater than $4W$,
apply the preceding estimate to it.
Otherwise, $\ell\leq Y_\eta+4W$, and nonnegativity
of the count suffices.
Therefore \eqref{s7-local-window} holds for all
$R\geq1$ and $\ell\geq0$ with
$C_{\Gamma,\gamma}=B+(Y_\eta+4W)/h_\gamma$.

Divide \eqref{s7-local-window} by $\ell$, take the infimum
over $R\geq1$ and let $\ell\to\infty$.
Finally, let $\gamma\downarrow\Gamma$.
The continuity of $\mathcal R^{-1}$, given by
Lemma~\ref{s7-monotonicity}, proves
\eqref{s7-local-sharp-density}.
\end{proof}

\section{Endpoint expansions for tight frames}
\label{app:tight-endpoint-expansions}

We refine the exact formula in Theorem~\ref{s7-main}
as the prescribed bound $\Gamma$ on the condition number
after sampling approaches $1$ or infinity.

\begin{theorem}
As $\Gamma\downarrow1$, with
$L=\log(1/(\Gamma-1))$,
\(
 \pi^2C^*(1;\Gamma)
 =L+(\log L)/2+\log8+o(1).
\)
As $\Gamma\to\infty$,
\[
 \frac1{C^*(1;\Gamma)}
 =\log\Gamma+\log\log\Gamma+1+o(1).
\]
The same expansions hold for $C^*_{\mathbb Z}$.
\end{theorem}

\begin{proof}
By Theorem~\ref{s7-main},
$C^*(1;\Gamma)=C^*_{\mathbb Z}(1;\Gamma)=1/h_\Gamma$,
where $\mathcal R(h_\Gamma)=\Gamma$ and
$\mathcal R(h)=M_h/m_h$.
We therefore estimate $M_h$ and $m_h$ at the two endpoints.

First let $h\downarrow0$.
Recall that $q_h(t)=P_h(ht)$ has period one.
By \eqref{s7-fourier-profile} and
\eqref{s7-gamma-modulus}, its first nonconstant Fourier
coefficient has modulus
\[
 A_h
 =\left|\Gamma_{\!E}\left(1+\frac{2\pi i}{h}\right)\right|
 =\frac{2\pi}{\sqrt h}e^{-\pi^2/h}
  (1-e^{-4\pi^2/h})^{-1/2}.
\]
The same identity gives
\[
 \sum_{m=2}^\infty
 \left|\Gamma_{\!E}\left(1+\frac{2\pi im}{h}\right)\right|
 =O(h^{-1/2}e^{-2\pi^2/h})=o(A_h).
\]
Thus
$q_h(t)=1+2A_h\cos(2\pi t+\vartheta_h)+r_h(t)$
for some $\vartheta_h\in\mathbb R$, with
$\|r_h\|_\infty=o(A_h)$.
It follows that $M_h=1+2A_h+o(A_h)$ and
$m_h=1-2A_h+o(A_h)$, so
$\mathcal R(h)-1
=8\pi h^{-1/2}e^{-\pi^2/h}(1+o(1))$.

At $h=h_\Gamma$, put $Y=\pi^2/h_\Gamma$.
Taking logarithms yields
$L=Y-(\log Y)/2-\log8+o(1)$.
Hence $Y/L\to1$ and then
$Y=L+(\log L)/2+\log8+o(1)$.
This proves the first expansion.

Now let $h\to\infty$ and put $q=e^h$.
For $0\leq s\leq h$, set $v=e^{h-s}\in[1,q]$.
With $f(v)=ve^{-v}$, the definition of $P_h$ gives
$P_h(s)/h=\sum_{m\in\mathbb Z}f(vq^m)$.
Uniformly for $1\leq v\leq q$,
\[
 \sum_{m\leq-2}f(vq^m)
 \leq\frac{v}{q(q-1)}
 \leq\frac1{q-1},\qquad
 \sum_{m\geq1}f(vq^m)
 \leq\sum_{m\geq1}q^me^{-q^m}
 =O(qe^{-q}).
\]
Therefore
$P_h(s)/h=f(v)+f(v/q)+O(q^{-1})$,
and all omitted terms are nonnegative.

The function $f$ is increasing on $[0,1]$ and decreasing
on $[1,\infty)$.
Splitting $[1,q]$ at $v=h$ gives
\[
 \frac hq e^{-h/q}
 \leq f(v)+f(v/q)
 \leq\frac1e+\frac hq,
 \qquad 1\leq v\leq q,\quad h\geq1.
\]
At $v=1$, the sum is at least $1/e$.
Together with the uniform remainder bound, this proves
$M_h=(h/e)(1+o(1))$.

For the minimum, the preceding lower bound gives
$m_h/h\geq(h/q)e^{-h/q}$.
For an upper bound, choose $v=h+2\log h$, which belongs
to $[1,q]$ for sufficiently large $h$.
Then $f(v)=O((qh)^{-1})$ and
$f(v/q)=(h/q)(1+o(1))$.
Since $q^{-1}=o(h/q)$, the uniform remainder estimate gives
$m_h/h\leq(h/q)(1+o(1))$.
Thus $m_h=h^2e^{-h}(1+o(1))$, and consequently
$\mathcal R(h)=e^h/(eh)(1+o(1))$.

At $h=h_\Gamma$, this yields
$\log\Gamma=h_\Gamma-\log h_\Gamma-1+o(1)$.
Hence $h_\Gamma/\log\Gamma\to1$ and then
$h_\Gamma=\log\Gamma+\log\log\Gamma+1+o(1)$.
The identity $1/C^*(1;\Gamma)=h_\Gamma$ proves
the second expansion.
\end{proof}

\section{Radial gaps and the effect of additional spectral points}
\label{app:normal-radial-gaps}

We prove Example~\ref{geo:gap-counterexample} and the counting bound
stated at the end of Section~\ref{sec:spectral-geometry}.
The spectral set is independent of the bound $\kappa$ on the
condition number of the original frame, whereas the sampling
measure may depend on $\kappa$.

\begin{proof}[Proof of Example~\ref{geo:gap-counterexample}]
We construct decreasing numbers $s_k>0$, positive integers $q_k,T_k$,
and weights $w_k$. Set
\[
 q_k=\lfloor s_k^{-1/2}\rfloor,\qquad
 r_k=e^{-s_k},\qquad
 d_k=1-e^{-2s_kq_k},\qquad
 \Omega_k=\{r_ke^{2\pi i\ell/q_k}:0\leq\ell<q_k\}.
\]
We require $s_1<1/16$ and $s_{k+1}<s_k/4$, and choose $T_k$ to be
a positive multiple of $q_k$. Define
\[
 w_k=\frac{e^{2s_kT_k}}{d_k},\qquad
 \mu_k=w_k\sum_{n=0}^{q_k-1}\delta_{T_k+n},\qquad
 \mu=\sum_{k\geq1}\mu_k,\qquad
 \Omega_0=\bigcup_{k\geq1}\Omega_k.
\]
Writing $Q_k=\sum_{i=1}^kq_i$, we also require
\begin{equation}\label{geo:gap-delays}
 T_k>T_{k-1}+q_{k-1},\qquad
 \log T_k\geq Q_k^2,
\end{equation}
where the first inequality applies only for $k\geq2$.

Consider any normal system satisfying \eqref{geo:system} with
spectral measure carried by $\Omega_0$ and $I\leq S\leq\kappa I$.
Let $P_k=E_A(\Omega_k)$ and $H_k=P_kH$. Then
$H=\bigoplus_{k\geq1}H_k$.
Write $A_k=A|_{H_k}$, $S_{kk}=P_kS|_{H_k}$, and
$B_{kk}=P_kB|_{H_k}$, where $B=C^*C$.
Since $A_k^{q_k}=r_k^{q_k}I$, grouping the integer Gramian by
residues modulo $q_k$ gives
\[
 S_{kk}=\frac1{d_k}\sum_{n=0}^{q_k-1}
              (A_k^n)^*B_{kk}A_k^n,\qquad
 B_{kk}\leq d_kS_{kk}\leq\kappa d_kI.
\]
Thus $\mu_k$ reproduces the original Gramian on $H_k$:
\begin{equation}\label{geo:gap-exact}
 P_kG_{\mu_k}P_k|_{H_k}
 =w_kr_k^{2T_k}\sum_{n=0}^{q_k-1}
              (A_k^n)^*B_{kk}A_k^n
 =S_{kk}.
\end{equation}
For $l\ne k$, the bound $B_{ll}\leq\kappa d_lI$ gives
\begin{equation}\label{geo:gap-leakage}
 \|P_lG_{\mu_k}P_l|_{H_l}\|
 \leq\kappa e_{kl},\qquad
 e_{kl}:=q_kd_lw_ke^{-2s_lT_k}
       =\frac{q_kd_l}{d_k}e^{-2(s_l-s_k)T_k}.
\end{equation}

We choose the parameters independently of $\kappa$ so that
\begin{equation}\label{geo:gap-small-leakage}
 e_{kl}\leq2^{-2(k+l)},\qquad k\ne l.
\end{equation}
Suppose the parameters with indices less than $k$ have been chosen.
Since $d_k\leq2s_kq_k\leq2\sqrt{s_k}\to0$ as $s_k\downarrow0$,
we may choose $s_k<s_{k-1}/4$ so that
$q_iw_id_k\leq2^{-2(i+k)}$ for every $i<k$.
This implies $e_{ik}\leq2^{-2(i+k)}$.
Next choose a sufficiently large multiple $T_k$ of $q_k$ satisfying
\eqref{geo:gap-delays}.
For each $l<k$, the last expression in \eqref{geo:gap-leakage}
tends to zero as $T_k\to\infty$, because $s_l>s_k$.
Increasing $T_k$ therefore ensures the remaining inequalities
in \eqref{geo:gap-small-leakage}.
For $k=1$, choose any $s_1<1/16$ and omit the conditions involving
earlier indices. This constructs $\Omega_0$ and a locally finite
measure $\mu$ supported on the positive integers.

Fix $\kappa\geq1$. Define
\(
 \mathcal T_kx=
 \bigl(w_k^{1/2}CA^{T_k+n}x\bigr)_{0\leq n<q_k}
 \in Y^{q_k}.
\)
For $J\geq2$, let
$\mathcal T_{\geq J}x=(\mathcal T_kx)_{k\geq J}$,
$\mathcal Dx=(\mathcal T_kP_kx)_{k\geq J}$, and
$\mathcal E=\mathcal T_{\geq J}-\mathcal D$.
Initially, take $x$ to be a finite sum of vectors in the $H_l$.
Set $P=\sum_{k\geq J}P_k$ and $Q=I-P$.
By \eqref{geo:gap-exact},
$\|\mathcal D\|\leq\sqrt\kappa$ and
$\|\mathcal Dx\|\geq\|x\|$ for $x\in PH$.

Equations \eqref{geo:gap-leakage} and
\eqref{geo:gap-small-leakage} give
$\|\mathcal T_kP_l\|\leq\sqrt\kappa\,2^{-(k+l)}$ for $k\ne l$.
Applying the Cauchy--Schwarz inequality to
$\sum_l2^{-l}\|P_lx\|$ yields
\(
 \|\mathcal E\|\leq{\sqrt\kappa}/{3},\
 \|\mathcal EP\|
 \leq\sqrt\kappa\sum_{k\geq J}4^{-k}.
\)
These estimates extend the operators to all of $H$ by density.
Choose $J$ so that
$\sqrt\kappa\sum_{k\geq J}4^{-k}\leq1/2$.
For $\mu_{\geq J}=\sum_{k\geq J}\mu_k$, we obtain
\begin{equation}\label{geo:gap-tail}
 PG_{\mu_{\geq J}}P\geq\tfrac14P,\qquad
 0\leq G_{\mu_{\geq J}}\leq4\kappa I.
\end{equation}
Applying $\|u+v\|^2\geq(1/2)\|u\|^2-\|v\|^2$ to
$\mathcal T_{\geq J}Px$ and $\mathcal T_{\geq J}Qx$ gives
\(
 G_{\mu_{\geq J}}\geq(1/8)P-4\kappa Q.
\)

We now obtain a lower bound on the remaining subspace $QH$.
Let $r_Q=\max_{k<J}r_k=e^{-s_{J-1}}<1$.
Then $\|A|_{QH}\|\leq r_Q$, regardless of the spectral
multiplicities.
For an integer $L\geq1$, let $\tau_L$ translate a measure by $L$,
and set
\(
 \nu_\kappa=\sum_{n=0}^{L-1}\delta_n
              +8\kappa\,\tau_L\mu_{\geq J}.
\)
Writing $|A|=(A^*A)^{1/2}$, the integer Gramian identity gives
\[
\begin{aligned}
 G_{\nu_\kappa}
 &=S-(A^L)^*SA^L
       +8\kappa(A^L)^*G_{\mu_{\geq J}}A^L
 \geq I-\kappa|A|^{2L}
       +\kappa|A|^{2L}P-32\kappa^2|A|^{2L}Q\\
 &\geq I-(\kappa+32\kappa^2)r_Q^{2L}Q.
\end{aligned}
\]
Choose $L$ so that $(\kappa+32\kappa^2)r_Q^{2L}\leq1/2$.
The upper bound in \eqref{geo:gap-tail} then gives
\(
(1/2)I\leq G_{\nu_\kappa}\leq(\kappa+32\kappa^2)I.
\)
The choices of $J$ and $L$ depend only on $\kappa$ and the fixed
construction. Hence
$\nu_\kappa\in\mathcal U^{\mathbb Z}_{w,\Omega_0}(\kappa)$.

If $T_k\leq R<T_{k+1}$, then \eqref{geo:gap-delays} implies
\[
 \frac{N_\mu(R)}{\log R}
 \leq\frac{Q_k}{Q_k^2}
 =\frac1{Q_k}\longrightarrow0.
\]
Since $N_{\nu_\kappa}(R)\leq L+N_\mu(R)$, we also have
$N_{\nu_\kappa}(R)=o(\log R)$ and $D(\nu_\kappa)=0$.

Finally, choose $j_k$ so that $2^{-j_k}\leq s_k<2^{1-j_k}$.
The sets $\Omega_k$ lie in distinct annuli from
\eqref{geo:annuli}.
Their arguments have circular separation $2\pi/q_k$, whereas the
arcs defining $m_{j_k}(\Omega_0)$ have length $2^{-j_k}/10$.
Since $q_k\sim s_k^{-1/2}$, each such arc contains at most one
argument from $\Omega_k$ for large $k$. Therefore
\[
 m_{j_k}(\Omega_0)=q_k,\qquad
 \frac{m_{j_k}(\Omega_0)}{j_k}
 \asymp\frac{s_k^{-1/2}}{\log(1/s_k)}
 \longrightarrow\infty.
\]
Thus $\Omega_0$ does not satisfy \eqref{geo:cover-condition}.
\end{proof}

For the next proposition, consider the system with state space
$H=\mathbb C^{q_k}$ and observation space $Y=\mathbb C$ defined by
\(
 Ae_\ell=r_ke^{2\pi i\ell/q_k}e_\ell,\
 Ce_\ell=\sqrt{1-r_k^2}\) for  \(0\leq\ell<q_k\),
where $(e_\ell)_{\ell=0}^{q_k-1}$ is the standard basis.
Grouping the orbit vectors by residues modulo $q_k$ gives the
discrete Fourier basis. The integer Gramian $S_k$ therefore has
eigenvalues
\begin{equation}\label{geo:gap-circulant}
 \frac{q_k(1-r_k^2)}{1-r_k^{2q_k}}r_k^{2\ell},
 \qquad 0\leq\ell<q_k.
\end{equation}
In particular,
$\operatorname{cond}(S_k)=e^{2s_k(q_k-1)}\to1$.

Adding the positive spectral value $e^{-2s_k}$ for each $k$
changes the counting conclusion as follows.

\begin{proposition}\label{geo:radial-thickness}
Let $\Omega_+=\Omega_0\cup\{e^{-2s_k}:k\geq1\}$.
For every $\kappa>1$ and
$\mu\in\mathcal U_{w,\Omega_+}(\kappa)$, there exist $c_\mu>0$
and $R_k\to\infty$ such that
$N_\mu(R_k)\geq c_\mu\sqrt{R_k}$.
\end{proposition}

\begin{proof}
Let $0<\alpha\leq\beta<\infty$ be common bounds for $\mu$, so that
$\alpha I\leq G_\mu\leq\beta I$ for every system in the class.
For sufficiently large $k$, write $s=s_k$, $q=q_k$, and $r=e^{-s}$.
By \eqref{geo:gap-circulant}, the system on $\mathbb C^q$ defined
above has Gramian condition number at most $\kappa$.
Its smallest Gramian eigenvalue $a$ satisfies $a\leq1$, since
every diagonal entry of the Gramian is one.
Replacing the observation by $Ce_\ell=\sqrt{(1-r^2)/a}$ gives
$I\leq S\leq\kappa I$ and
$\|C\|^2=q(1-r^2)/a\geq q(1-r^2)$.

Under the convention for spectral powers in
Section~\ref{nrm-section}, $A^t=e^{-st}U_t$ with $U_t$ unitary.
Each sampled term is positive and bounded above by $G_\mu$.
Thus every atom of weight $w_t$ satisfies
\begin{equation}\label{geo:thickness-atom}
 w_te^{-2st}\leq\frac{\beta}{q(1-r^2)}.
\end{equation}

Next consider the two systems with $H=Y=\mathbb C$ given by
$A=r$, $C=\sqrt{1-r^2}$ and
$A=r^2$, $C=\sqrt{1-r^4}$.
Both have integer Gramian equal to one, and their spectral values
belong to $\Omega_+$. Their sampling bounds give
\[
 (1-r^2)\int e^{-2st}\,d\mu(t)\leq\beta,\qquad
 (1-r^4)\int e^{-4st}\,d\mu(t)\geq\alpha.
\]
Since $(1-r^4)/(1-r^2)=1+r^2\leq2$, the contribution from
times greater than $R$ to the second expression is at most
$2\beta e^{-2sR}$.
Set $R_k=\log(4\beta/\alpha)/(2s_k)$.
The contribution from $[0,R_k]$ is then at least $\alpha/2$.
By \eqref{geo:thickness-atom}, each atom contributes at most
$2\beta/q_k$. Consequently,
\[
 N_\mu(R_k)\geq\frac{\alpha q_k}{4\beta}
 \gtrsim_{\alpha,\beta}\sqrt{R_k}.
\]
Here $R_k\to\infty$ and $q_k\sim s_k^{-1/2}$.
This also excludes $N_\mu(R)=O(\log R)$ for both real and integer
sampling times.
\end{proof}

\end{document}